\documentclass[12pt]{amsart}

\usepackage{enumerate, amsmath, amsthm, amsfonts, amssymb, xy,  combelow, mathrsfs, graphicx, paralist, fancyvrb, ytableau, braket, xcolor, verbatim, tikz, tikz-cd, mathtools,enumitem}
\usepackage[bookmarks, colorlinks=true, linkcolor=blue, citecolor=blue, urlcolor=blue]{hyperref}

\input xy
\xyoption{all}

\DeclareMathOperator{\quot}{\mathsf{Quot}}

\renewcommand{\tilde}[1]{\widetilde{#1}}
\renewcommand{\bar}[1]{\overline{#1}}
\newtheorem{theorem}{Theorem}
\newtheorem{lemma}{Lemma}
\newtheorem{corollary}{Corollary}
\newtheorem{proposition}{Proposition}

\newtheorem{definition}{Definition}

\theoremstyle{definition}
\newtheorem{remark}{Remark}

\newtheoremstyle{TheoremNum}
        {7pt}{7pt}              
        {\itshape}                      
        {}                              
        {\bfseries}                     
        {.}                             
        { }                             
        {\thmname{#1}\thmnote{ \bfseries #3}}
    \theoremstyle{TheoremNum}

\newcommand{\BA}{\mathbb{A}}
\newcommand{\BC}{\mathbb{C}}
\newcommand{\BN}{\mathbb{N}}
\newcommand{\BP}{\mathbb{P}}
\newcommand{\BQ}{\mathbb{Q}}
\newcommand{\BZ}{\mathbb{Z}}

\newcommand{\CE}{\mathcal{E}}
\newcommand{\CF}{\mathcal{F}}
\newcommand{\CG}{\mathcal{G}}

\newcommand{\CK}{\mathcal{K}}
\newcommand{\CL}{\mathcal{L}}
\newcommand{\CO}{\mathcal{O}}
\newcommand{\CQ}{\mathcal{Q}}

\newcommand{\CY}{\mathcal{Y}}

\newcommand{\DD}{\mathsf{D}}
\newcommand{\KK}{\mathsf{K}}

\newcommand{\tE}{\widetilde{E}}

\newcommand{\tCL}{\widetilde{\CL}}

\newcommand{\fY}{\mathsf{Y}}
\newcommand{\fZ}{\mathsf{Z}}
\newcommand{\fsl}{\mathfrak{sl}}

\newcommand{\be}{\widetilde{e}}

\newcommand{\red}{\text{ red}}
\newcommand{\ered}{\emph{ red}}

\newcommand{\te}{{e}}
\newcommand{\tf}{{f}}
\newcommand{\tm}{{m}}

\newcommand{\ke}{\bar{e}}
\newcommand{\kf}{\bar{f}}
\newcommand{\km}{\bar{m}}

\newcommand{\qe}{\mathsf{e}}
\newcommand{\qf}{\mathsf{f}}
\newcommand{\qm}{\mathsf{m}}
\newcommand{\qh}{\mathsf{h}}

\newcommand{\UU}{U_q^r(L\fsl_2)}

\newcommand{\bd}{\underline{d}}

\newcommand{\fp}{\mathfrak{p}}
\newcommand{\fg}{\mathfrak{g}}
\newcommand{\tfg}{\widetilde{\fg}}

\newcommand{\bull}{\bullet}
\newcommand{\sign}{\int}

\begin{document}

\parindent=30pt

\baselineskip=17pt
\title [Derived categories of Quot schemes and tautological bundles]{Derived categories of Quot schemes on smooth curves and tautological bundles}

\author{Alina Marian and Andrei Negu\cb{t}}

\begin{abstract}

We define a categorical action of the shifted quantum loop group of $\fsl_2$ on the derived categories of Quot schemes of finite length quotient sheaves on a smooth projective curve. As an application, we obtain a semi-orthogonal decomposition of the derived categories of Quot schemes, of representation-theoretic origin. 
We use this decomposition to calculate the cohomology of interesting tautological vector bundles over the Quot scheme. 

\end{abstract}

\address{The Abdus Salam International Centre for Theoretical Physics, Strada Costiera 11, 34151 Trieste, Italy \newline \text{ } \qquad Department of Mathematics, Northeastern University, 360 Huntington Avenue, Boston, MA 02115}
\email{amarian@ictp.it}

\vskip.2in

\address{École Polytechnique Fédérale de Lausanne (EPFL), Lausanne, Switzerland \newline \text{ } \qquad Simion Stoilow Institute of Mathematics (IMAR), Bucharest, Romania}

\email{andrei.negut@gmail.com}

\date{}

\vskip.2in

\maketitle

\setcounter{tocdepth}{1}
\tableofcontents

\vskip.2in

\section{Introduction}

\vskip.2in

\subsection{Quot schemes over curves}
\label{sub:intro quot}

In algebraic geometry, few moduli spaces of sheaves have been shown to admit semiorthogonal decompositions of their derived category which are geometrically explicit with simple components, in the spirit of the classic case of homogeneous varieties \cite{beilinson, 
 kapranov1, kapranov2}.
The present paper uses representation-theoretic techniques to obtain a geometrically transparent semiorthogonal decomposition of the derived category of the Grothendieck Quot scheme parameterizing quotient sheaves of finite length on a smooth projective curve.
Our description gives strong control over the derived category. We are able to compute the cohomology groups of tautological vector bundles over the Quot scheme, proving conjectures of \cite{krug1, os}. We now detail the discussion.  

\medskip

Fix $V$ a locally free sheaf of rank $r \geq 1$ on a smooth complex projective curve $C$. We consider Grothendieck's Quot scheme $\quot_d = \quot_{d}(V)$ parametrizing rank $0$ degree $d$ quotients of $V$: 
\begin{equation}
\label{eqn:quot very intro}
V \twoheadrightarrow F,\quad \text{where rank }F=0, \ \text{deg }F=d.
\end{equation}
Alternatively, we may describe $\quot_d$ as $\{(E \subset V)\}$, where $E = \text{Ker }(V \twoheadrightarrow F)$. We will denote by $\pi$ and $\rho$ the projections from the product $\quot_d \,\times \, C$ to the two factors. On general grounds, the Quot scheme carries a universal short exact sequence
\begin{equation}
\label{eqn:ses intro}
0 \to \CE \to \rho^* (V) \to \CF \to 0 \qquad \text{on}  \qquad \quot_d \,  \times \, C,
\end{equation}
and its deformation-obstruction complex is $\text{Ext}^\bullet_\pi (\mathcal E, \, \mathcal F)$. Since all the quotients $F$ are supported at finitely many points of $C$, this complex is supported in homological degree $\bullet = 0$, hence $\quot_d$ is a smooth projective variety of dimension $rd$. 

\medskip

In \cite{mn}, we defined natural operators on the cohomology (with $\mathbb{Q}$ coefficients) of the Quot scheme $\quot_d$, and identified the algebra they generate with the shifted Yangian of $\mathfrak{sl}_2$. Furthermore, we showed that a subalgebra of commuting operators, admitting a simple geometric interpretation, gives rise to a canonical basis of the cohomology of $\quot_d$. We thereby gave a representation theoretic interpretation of the well-known formulas for the Betti numbers of the Quot scheme (\cite{stromme, bgl, chen, bfp, ricolfi}). In the present paper, we lift the cohomological operators to the derived category of Quot schemes, and study the rich geometry of the resulting functors. We obtain a semi-orthogonal decomposition of the derived category of $\quot_d$, with blocks indexed by the same data as the canonical basis of the cohomology. When $C = \mathbb P^1$, the decomposition yields a full exceptional collection in the derived category. We finally use the semiorthogonal decomposition to calculate the cohomology of natural tautological vector bundles over the Quot scheme. Our main results are Theorems \ref{thm:intro main}, \ref{thm:intro semi}, and \ref{thm:tautological intro}, describing the categorical relations between our functors, the semi-orthogonal decomposition, and the cohomology of tautological bundles. We now explain these results. 

\subsection{The functors $\te_i,\tf_i$ and $\tm_i$}
\label{sub:intro 1}

Throughout the present paper, we write
\begin{equation}
\label{eqn:dd}
\DD_X = D^b(\text{Coh}(X))
\end{equation}
for any smooth algebraic variety $X$. For any $d \geq 0$, consider the nested Quot scheme
\begin{equation}
\label{eqn:intro nested}
\quot_{d,d+1} = \Big\{ E' \subset E \subset V \Big\}
\end{equation}
with the injections above having colengths $1$ and $d$, respectively. As we will recall in Subsection \ref{sub:basic nested}, $\quot_{d,d+1}$ is a smooth projective variety of dimension $r(d+1)$. For any point of the moduli space \eqref{eqn:intro nested}, we have $E/E' \cong \BC_x$ for some closed point $x \in C$. Therefore, there exist maps 
 \begin{equation}
\label{eqn:diagram intro}
\xymatrix{& \quot_{d,d+1} \ar[ld]_{p_-} \ar[d]^{p_C} \ar[rd]^{p_+} & \\ \quot_{d} & C & \quot_{d+1}}
\end{equation}
which remember $E \subset V,\, x, \,E' \subset V$, respectively. Moreover, there is a tautological line bundle
 \begin{equation}
\label{eqn:tautological intro}
\xymatrix{\CL \ar@{.>}[d] \\ \quot_{d,d+1}}
\end{equation}
which encodes $\Gamma(C,E/E')$. As reviewed in Subsection \ref{sub:basic nested}, the maps $p_\pm$ are local complete intersection morphisms, and $p_\pm \times p_C$ are proper (these properties also hold if the curve $C$ is not projective, although we will not deal with this case in the present paper). This allows us to define the functors 
\begin{align}
&\DD_{\quot_d} \xrightarrow{\te_i} \DD_{\quot_{d+1} \times C}, \quad \te_i(\gamma) = R(p_+ \times p_C)_* \Big(\CL^{i} \otimes Lp_-^* (\gamma) \Big) \label{eqn:e intro} \\
&\DD_{\quot_{d+1}} \xrightarrow{\tf_i} \DD_{\quot_{d} \times C}, \quad \tf_i(\gamma) = R(p_- \times p_C)_* \Big(\CL^{i} \otimes Lp_+^* (\gamma) \Big) \label{eqn:f intro}
\end{align}
for all $i \in \BZ$ and $d \geq 0$. We will also encounter the functors of tensor product with the exterior powers of the rank $r$ universal vector bundle from \eqref{eqn:ses intro}, for all $i \in \{0,\dots,r\}$
\begin{equation}
\label{eqn:m intro}
\DD_{\quot_d} \xrightarrow{\tm_i} \DD_{\quot_d \times C}, \quad \tm_i(\gamma) = \wedge^i\CE \otimes \pi^*(\gamma)
\end{equation}

\subsection{The categorical relations}

Our first main result is the identification of the commutation relations satisfied by the functors above. In what follows, we will consider the product $\textcolor{red}{C} \times \textcolor{blue}{C}$ of two copies of our curve $C$, and the colors of our functors will always match the color of the copy of $C$ in which they act. For instance, $\textcolor{red}{\te_i} \circ \textcolor{blue}{\tf_j}$ will refer to the composition
$$
\DD_{\quot_d} \xrightarrow{\textcolor{blue}{\tf_j}} \DD_{\quot_{d-1} \times \textcolor{blue}{C}} \xrightarrow{\textcolor{red}{\te_i} \boxtimes \text{Id}_{\textcolor{blue}{C}}} \DD_{\quot_{d} \times \textcolor{red}{C} \times \textcolor{blue}{C}}
$$
Let $\DD_{\quot} = \bigoplus_{d=0}^{\infty} \DD_{\quot_d}$. If $\Delta : C \hookrightarrow \textcolor{red}{C} \times \textcolor{blue}{C}$ denotes the diagonal, then
$$
\forall \phi : \DD_{\quot} \rightarrow \DD_{\quot \times \textcolor{red}{C} \times \textcolor{blue}{C}}, \quad \ \text{ we write } \ \quad \phi|_\Delta : \DD_{\quot} \rightarrow \DD_{\quot \times C}
$$
$$
\forall \psi : \DD_{\quot} \rightarrow \DD_{\quot \times C}, \quad \text{we write} \quad \Delta_*(\psi) : \DD_{\quot} \rightarrow  \DD_{\quot \times \textcolor{red}{C} \times \textcolor{blue}{C}} 
$$
for the appropriate compositions of $\phi,\psi$ with pull-back and push-forward under $\Delta$.

\begin{theorem}
\label{thm:intro main}

(a) (Proposition \ref{prop:quadratic relation e}) For any $i \leq j \in \BZ$, there is a natural transformation
\begin{equation}
\label{eqn:quadratic relation e intro}
\textcolor{blue}{\te_j} \circ \textcolor{red}{\te_i}  \longrightarrow \textcolor{red}{\te_i} \circ \textcolor{blue}{\te_j}
\end{equation}
of functors $\DD_{\quot} \rightarrow \DD_{\quot \times \textcolor{red}{C} \times \textcolor{blue}{C}}$, whose cone has a filtration with associated graded 
\begin{equation}
\label{eqn:cone ee intro}
\bigoplus_{k=i}^{j-1} \Delta_*\left(\te_k \circ \te_{i+j-k} \Big|_\Delta\right)
\end{equation}

\medskip

\noindent (b) (Proposition \ref{prop:e and f}) For all $i,j \in \BZ$, there is a natural transformation of functors 
\begin{equation}
\label{eqn:e and f intro}
\textcolor{blue}{\tf_j} \circ \textcolor{red}{\te_{i}} \xleftrightharpoons[\text{if }i+j \geq 0]{\text{if }i+j < 0}  \textcolor{red}{\te_{i}} \circ \textcolor{blue}{\tf_j}
\end{equation}
whose cone is $\Delta_*(\alpha_{i+j} \otimes \pi^*(-))$, where the objects $\{\alpha_{\ell} \in \DD_{\quot \times C} \}_{\ell \in \BZ}$ are precisely defined in Proposition \ref{prop:e and f}.

\medskip

\noindent (c) (Proposition \ref{prop:comm e and m}) For all $i\in \BZ$ and $j \in \{0,\dots,r\}$, there is a natural transformation
\begin{equation}
\label{eqn:e and m natural transformation intro}
\textcolor{blue}{\tm_j} \circ \textcolor{red}{\te_i} \longrightarrow \textcolor{red}{\te_i} \circ \textcolor{blue}{\tm_j}
\end{equation}
of functors $\DD_{\quot} \rightarrow \DD_{\quot \times \textcolor{red}{C} \times \textcolor{blue}{C}}$, which is an isomorphism for $j=0$ and has cone
\begin{equation}
\label{eqn:e and m cone intro}
\Delta_* \left( \te_{i+1} \circ \tm_{j-1} \Big|_\Delta \rightarrow \dots \rightarrow \te_{i+j} \circ \tm_{0} \Big|_\Delta \right) 
\end{equation}
for all $j \geq 1$. Moreover, we have an isomorphism of functors $\DD_{\quot} \rightarrow \DD_{\quot \times \textcolor{red}{C} \times \textcolor{blue}{C}}$
\begin{equation}
\label{eqn:iso intro}
\textcolor{blue}{\tm_r} \circ \textcolor{red}{\te_i} \cong  \textcolor{red}{\te_i} \circ \textcolor{blue}{\tm_r} (-\Delta)
\end{equation}

\end{theorem}

In $(a)$-$(c)$ above, the transposed results hold with $e$'s replaced by $f$'s, i.e. replacing all $\te_a \circ \te_b$ by $\tf_b \circ \tf_a$, all $\te_a \circ \tm_b$ by $\tm_b \circ \tf_a$ etc. 

\medskip

Motivated by the $K$-theoretic discussion to follow, we interpret Theorem \ref{thm:intro main} as a categorical action of shifted quantum $L\fsl_2$ on $\DD_{\quot}$, see \eqref{eqn:action intro}. This can be compared with the categorical action of quantum $\fsl_2$ on the equivariant derived categories of cotangent bundles to Grassmannians from \cite{ckl} (as well as the affine version of \cite[Section 4.3]{cl}, although the latter is quite a different presentation from ours). 

\subsection{The operators $e_i,f_i,m_i$ and shifted quantum $L\fsl_2$}
\label{sub:intro 1 bis}

Let us now consider the $K$-theoretic shadows of the functors defined in the preceding Subsection, namely 
\begin{align*}
&\KK_{\quot_d} \xrightarrow{\ke_i} \KK_{\quot_{d+1} \times C}, \quad \ke_i(\gamma) = (p_+ \times p_C)_* \Big([\CL^{i}] \otimes p_-^* (\gamma) \Big) \\
&\KK_{\quot_{d+1}} \xrightarrow{\kf_i} \KK_{\quot_{d} \times C}, \quad \kf_i(\gamma) = (p_- \times p_C)_* \Big([\CL^{i}] \otimes p_+^* (\gamma) \Big) \\
&\KK_{\quot_d} \xrightarrow{\km_i} \KK_{\quot_d \times C}, \quad \ \ \km_i(\gamma) = [\wedge^i\CE] \otimes \pi^*(\gamma)
\end{align*}
where $\KK_X$ denotes the algebraic $K$-theory ring of a smooth algebraic variety $X$. As a consequence of Theorem \ref{thm:intro main}, we obtain the following formulas.

\begin{corollary}
\label{cor:k-th intro}
 
We have the commutation relations 
$$
[\textcolor{red}{\ke_i}, \textcolor{blue}{\ke_j}] = \Delta_* \left( \sum_{k=i}^{j-1} \ke_k \circ \ke_{i+j-k} \Big|_\Delta \right)
$$
for all integers $i \leq j$,
$$
[\textcolor{red}{\ke_i}, \textcolor{blue}{\km_j}] = \Delta_* \left( \sum_{k=1}^{j} (-1)^{k+1} \ke_{i+k} \circ \km_{j-k} \Big|_\Delta \right)
$$
for all $i \in \BZ$ and $j \in \{0,\dots,r\}$, 
$$
\textcolor{blue}{\km_r} \circ \textcolor{red}{\ke_i} = \textcolor{red}{\ke_i} \circ \textcolor{blue}{\km_r} (-\Delta)
$$
for all $i \in \BZ$, and (see \eqref{eqn:wedge} and \eqref{eqn:wedge virtual} for the $\wedge^\bullet$ notation)
$$
[\textcolor{red}{\ke_i}, \textcolor{blue}{\kf_j}] = \Delta_* \left(\text{multiplication by } \int_{\infty - 0} z^{i+j} \cdot \frac {\wedge^\bullet\left(\frac {zq}{V}\right)}{\wedge^\bullet\left(\frac {\CE}z\right)\wedge^\bullet\left(\frac {zq}{\CE}\right)} \right)
$$
for all $i,j \in \BZ$, where $\int_{\infty-0} R(z)$ denotes the constant term in the expansion of a rational function $R(z)$ near $z=\infty$ minus the constant term in the expansion of $R(z)$ near $z=0$. The transposes of the above formulas with $e$'s replaced by $f$'s also hold. 

\end{corollary}

In Subsections \ref{sub:action 1} and \ref{sub:action 2}, we interpret Corollary \ref{cor:k-th intro} as giving an action 
\begin{equation}
\label{eqn:action intro}
U_q^r(L\fsl_2) \curvearrowright \KK_{\quot} = \bigoplus_{d = 0}^{\infty} \KK_{\quot_d}
\end{equation}
of the shifted quantum loop group from Definition \ref{def:quantum loop} (which is a certain $\BZ[q^{\pm 1}]$ integral form of the usual shifted quantum loop group of \cite{ft}). Thus, Theorem \ref{thm:intro main} can be construed as a categorification of the action \eqref{eqn:action intro}.

\subsection{Semi-orthogonal decomposition}
\label{sub:intro 3}

In \cite{mn}, we constructed a basis (of representation theoretic origin) of the cohomology of $\quot_d$, which is indexed by ordered sequences $(0 \leq k_1 \leq \dots \leq k_d < r)$. Such sequences are in bijection with compositions
\begin{equation}
\label{eqn:composition intro}
\bd = \left(d_0,d_1 \dots,d_{r-1} \Big| d_0 + d_1 + \dots + d_{r-1} = d \right)
\end{equation}
where $d_i$ counts how many times $i$ appears in the sequence $(k_1 \leq \dots \leq k_d)$. We will use the latter indexing to construct a semi-orthogonal decomposition of the derived category of $\quot_d$. In what follows, we will write $C^{(d)} = C^d/S_d$ for the $d$-th symmetric power of $C$, and set $C^{(\bd)} = C^{(d_0)} \times C^{(d_1)} \times \dots \times C^{(d_{r-1})}$ for any composition \eqref{eqn:composition intro}.

\begin{theorem}
\label{thm:intro semi}

As $\bd = (d_0,\dots,d_{r-1})$ runs over compositions of $d$, there exist functors
\begin{equation}
\label{eqn:functor intro}
\DD_{C^{(\bd)}} \xrightarrow{\be_{\bd}^{\ered}} \DD_{\quot_d}
\end{equation}
which are fully faithful and semi-orthogonal, i.e. we have natural identifications
\begin{equation}
\label{eqn:fully faithful intro}
\emph{Hom}_{\quot_d} \Big ( \be^{\ered}_{\bd}  (\gamma), \be^{\ered}_{\bd'}  (\gamma') \Big) \cong \begin{cases} 0&\text{if } \bd < \bd' \text{ lexicographically}\\ \emph{Hom}_{C^{(\bd)}} (\gamma, \gamma' ) &\text{if }\bd = \bd' \end{cases}
\end{equation}
The essential images of the functors \eqref{eqn:functor intro} generate $\DD_{\quot_d}$ as a triangulated category.

\end{theorem}

The composition of the functor \eqref{eqn:functor intro} with the direct image $\DD_{C^d} \rightarrow \DD_{C^{(\bd)}}$ of the obvious finite covering map is precisely the composition
\begin{equation}
\label{eqn:functor intro compose}
\underbrace{\be_{0} \circ \dots \circ \be_{0}}_{d_0 \text{ factors}} \circ \underbrace{\be_{1} \circ \dots \circ \be_{1}}_{d_1 \text{ factors}} \circ \dots \circ  \underbrace{\be_{r-1} \circ \dots \circ \be_{r-1}}_{d_{r-1} \text{ factors}} : \DD_{C^d} \rightarrow \DD_{\quot_d}
\end{equation}
where the functor $\be_i : \DD_{\quot_{d} \times C} \rightarrow \DD_{\quot_{d+1}}$ is given by the same correspondence as the functor $\te_i : \DD_{\quot_{d}} \rightarrow \DD_{\quot_{d+1} \times C}$ of \eqref{eqn:e intro}. Thus, while the functor \eqref{eqn:functor intro compose} categorifies the element $\qe_{0}^{d_{0}} \qe_{1}^{d_1} \dots \qe_{r-1}^{d_{r-1}}$ of the shifted quantum loop group (see Definition \ref{def:quantum loop} for the notation), the functor \eqref{eqn:functor intro} categorifies the divided powers \eqref{eqn:divided powers}.

\begin{remark}
\label{rem:toda}

A semi-orthogonal decomposition of $\DD_{\quot_d}$ indexed by the same data was recently obtained in \cite{toda} by a different method, using categorical wall-crossing for the framed one-loop quiver and the machinery of windows and categorical Hall products. It would be interesting to match the decomposition \cite{toda} with our geometrically explicit construction. 

In a curve context, we note also the recent semiorthogonal decomposition \cite{tt, tevelev} of the derived category of the moduli space of rank 2 stable vector bundles with odd determinant. We speculate that it may be possible to take $\quot_d$ and its derived category as a starting point when approaching the derived category of the moduli space $\mathsf M (r, d)$ of stable bundles of arbitrary coprime rank $r$ and degree $d$, building a parallel to the calculation \cite{bgl} of the Betti numbers of $\mathsf M(r,d)$. We leave this direction for future investigation.  

 In higher dimensions, semiorthogonal decompositions of moduli spaces of sheaves are known in some cases, yet they are typically less transparent from a geometric point of view. For the Hilbert scheme of points on a surface for example, \cite{scala, krug2} work on the orbifold side of the Bridgeland-King-Reid equivalence \cite{bkr, haiman} and the geometry of the original Hilbert scheme is less manifest. Meanwhile, the recent work \cite{pt1, pt2, pt3, pt4, pt5, pt6} develops technically deep semiorthogonal decompositions in connection with categorified BPS invariants. The explicit representation theory and geometric potential for calculating the cohomology of natural tautological objects have not been explored yet.

\end{remark}

\subsection{Cohomology of tautological bundles} 

Recall the projections $\pi : \quot_d \times C \rightarrow \quot_d$ and $\rho : \quot_d \times C \rightarrow C$. If $M$ is a vector bundle of rank $n$ on the curve $C$, we denote by
$$M^{[d]} = \pi_*(\mathcal F \otimes \rho^*(M))
$$
its corresponding tautological vector bundle of rank $nd$ on $\quot_d$. We will note in Section \ref{sub:ex taut} that $M^{[d]}$ appears in the $(d-1, 1, 0, \ldots, 0)$ block of the semi-orthogonal decomposition, specifically 
\begin{equation}
\label{eqn:intro tautarbitrary}
M^{[d]} = \be_{(d-1, 1, 0, \ldots, 0)}^{\red} ( \CO^{(d-1)} \boxtimes M)
\end{equation}
Above and henceforth, we write $L^{(k)} \to C^{(k)}$ for the descent of any line bundle $L^k \to C^k$ under the action of the symmetric group $S_{k}$, and the line bundle $\CO^{(d-1)}$ is simply the $L = \CO_C$ version of this construction. Furthermore, when $M$ is a line bundle, we will establish that the exterior powers $\wedge^\ell M^{[d]}$ are in the $(d-\ell, \ell, 0, \ldots, 0)$ block of the semi-orthogonal decomposition, specifically
\begin{equation}
\label{eqn: intro0 wedges}    
\wedge^{\ell} M^{[d]} = \be_{(d-\ell, \ell, 0, \ldots, 0)}^{\red} (\mathcal O^{(d-\ell)} \boxtimes M^{(\ell)}).
\end{equation}
We may then effectively calculate the morphisms between exterior powers of tautological bundles, showing the following. 

\begin{theorem}
\label{thm:tautological intro}

For any line bundles $M_1,\dots,M_k,M$ on $C$ with $k < r$ and any $\ell_1,\dots,\ell_k,\ell$ $\in \{0,\dots,d\}$, we have an isomorphism of graded vector spaces
\begin{equation}
\label{eqn:iso tautological intro}
\emph{Ext}^\bull_{\quot_d}\left(\wedge^{\ell_1} M_1^{[d]} \otimes \dots \otimes \wedge^{\ell_k} M_k^{[d]}, \wedge^{\ell} M^{[d]} \right) \cong 
\end{equation}
$$
 \cong \bigotimes_{i=1}^k S^{\ell_i} H^\bull(C,M \otimes M_i^{\vee})\bigotimes \wedge^{\ell-\ell_1 - \dots - \ell_k} H^\bull(C,V \otimes M) \bigotimes S^{d-\ell}H^\bull(C,\CO_C)
$$
The right-hand side of the expression above is defined to be 0 if $\ell_1+\dots+\ell_k > \ell$.

\end{theorem}

The Euler characteristic version of \eqref{eqn:iso tautological intro} was established in \cite{os} by equivariant localization. The cohomological identification \eqref{eqn:iso tautological intro} was conjectured in \cite{os, krug1}, and partial cases were established in \cite{mos, krug1}. Using \eqref{eqn:intro tautarbitrary}, the semi-orthogonal decomposition of Theorem \ref{thm:intro semi}, also provides an alternate proof of Theorem 1.5 in \cite{krug1}. Thus, we think of Theorem \ref{thm:intro semi} on the semi-orthogonal decomposition and Theorem \ref{thm:tautological intro} on cohomology of tautological bundles as two closely related (albeit logically independent) results. In particular, formula \eqref{eqn: intro0 wedges} shows that exterior powers of tautological bundles lie in a given block of the semi-orthogonal decomposition.

\medskip

We note the parallelism between Theorem \ref{thm:tautological intro} and the calculation of the cohomology of tautological bundles in the same setup of the Hilbert scheme $S^{[d]}$ of points on a smooth projective surface $S$.
In particular, as proven in \cite{scala, krug2}, 
$$
H^\bullet (S^{[d]}, \wedge^\ell M^{[d]}) \simeq \wedge^\ell H^\bullet (S, M) \otimes S^{d- \ell} H^\bullet (S, \mathcal O_S),
$$
which closely resembles \eqref{eqn:iso tautological intro} when $\ell_1= \dots =\ell_k = 0$. 
The similarity of these two formulas on the level of Euler characteristics was first pointed out in \cite{os}, where the question of calculating the cohomology of $\wedge^\ell M^{[d]}$ on $\quot_d$ was also raised. 
From a technical point of view, the surface result above was proved by describing the exterior powers of tautological bundles via the Bridgeland-King-Reid-Haiman equivalence \cite{bkr, haiman}. Our result for the Quot scheme of a curve is proved by expressing exterior powers of tautological bundles as images of very simple objects under the natural quantum $L\mathfrak{sl}_2$ functors, which makes possible the calculation of their cohomology. 

\medskip

Among other cohomology studies, we mention also the calculation of the cohomology of the tangent bundle of the Quot scheme carried out in \cite{bgs}.

\medskip

Finally, we note that by taking linear combinations of products of exterior powers, we could replace the first argument in the Ext group of \eqref{eqn:iso tautological intro} by any product of Schur functors
$$
S^{\lambda_1} M_1^{[d]} \otimes \dots \otimes S^{\lambda_k} M_k^{[d]}
$$
where the sum of the leading components of the partitions $\lambda_1,\dots,\lambda_k$ is less than $r$. In principle, the machinery of Sections \ref{sec:steppingstone} and \ref{sec:semi} allows for the computation of $\text{Ext}$ groups of any Schur functor of tautological bundles in either of the two arguments, but the answer is recursive, and compact formulas cannot be expected in general.

\subsection{Acknowledgements} We thank Pavel Etingof, Eugene Gorsky, Anthony Licata, Dragos Oprea, Zsolt Patakfalvi, Claudiu Raicu, Steven Sam, Shubham Sinha, and Jenia Tevelev for interesting conversations related to the present paper. A.M. would like to thank the organizers of ELGA V in Brazil for a wonderful visit in August 2024. The work of A.N. was partially supported by the NSF grant DMS-1845034, the MIT Research Support Committee and the PNRR grant CF 44/14.11.2022 titled ``Cohomological Hall algebras of smooth surfaces and applications". He would like to thank the Simion Stoilow Institute of Mathematics of the Romanian Academy (IMAR, Bucharest) for their hospitality while the present paper was being written.

\section{Fundamentals}
\label{sec:fundamentals}

\vskip.2in

\subsection{Derived categories of projective bundles}

For a rank $r$ vector bundle $V$ on a smooth variety $X$, we may define its projectivization
\begin{equation}
\label{eqn:projectivization general}
\BP_X(V) = \text{Proj}_X(\text{Sym}^\bullet(V))
\end{equation}
It is endowed with a tautological line bundle $\CO(1)$, a projection map
$$
\pi : \BP_X(V) \rightarrow X
$$
(which is a $\BP^{r-1}$-fibration) and a short exact sequence
\begin{equation}
\label{eqn:taut ses general}
0 \rightarrow \CG \rightarrow \pi^*(V) \rightarrow \CO(1) \rightarrow 0.
\end{equation}
Furthermore $\CG \otimes \CO(-1)$ is the relative cotangent bundle of the morphism $\pi$, so
\begin{equation}
\label{eqn:canonical proj}
\frac {\CK_{\BP_X(V)}}{\pi^*(\CK_X)} \cong \pi^*(\det V) \otimes \CO(-r).
\end{equation}
where $\CK_Z$ denotes the canonical line bundle of a smooth variety $Z$. We recall the following standard result. 

\begin{lemma}
\label{lem:push d}
If $\pi : \BP_X(V) \rightarrow X$ is the projectivization of a rank $r$ vector bundle, then
\begin{equation}
\label{eqn:push d}
R\pi_*(\CO(k)) = \begin{cases} S^k V &\text{if } k \geq 0 \\ 0 &\text{if }-r < k < 0 \\ S^{-k-r} V^\vee \otimes \det(V^\vee)[-r+1] &\text{if } k \leq -r \end{cases}
\end{equation}

\end{lemma}

A well-known theorem of Orlov (\cite{orlov}) states that the derived category of $\BP_X(V)$ admits a semi-orthogonal decomposition
\begin{equation}
\label{eqn:orlov}
\DD_{\BP_X(V)} = \Big \langle \DD_X \otimes \CO, \DD_X \otimes \CO(1), \dots , \DD_X \otimes \CO(r-1)\Big \rangle
\end{equation}
where $\DD_X$ is identified with a full subcategory of $\DD_{\BP_X(V)}$ via $\pi^*$. 

\subsection{Virtual projective bundles} 
\label{sub:virtual}

Suppose we have a map of vector bundles
\begin{equation}
\label{eqn:map of v b}
\varphi : W \rightarrow V 
\end{equation}
on a smooth projective scheme $X$. If the composition
\begin{equation}
\label{eqn:are you regular}
\sigma : \pi^*(W) \xrightarrow{\pi^*(\varphi)} \pi^*(V) \longrightarrow \CO(1) \, \, \, \text{on} \, \, \, \BP_X(V) 
\end{equation}
is a regular section of the vector bundle $\pi^*(W^\vee) \otimes \CO(1)$, then we denote its zero locus by $\BP_X(V-W)$ and call it a virtual projectivization. We have a commutative diagram
$$
\xymatrix{
\BP_X(V-W) \ar[dr]_-{\pi'} \ar@{^{(}->}[r]^-{\iota} & \BP_X(V) \ar[d]^{\pi} \\
& X}
$$
where $\iota$ is the zero locus of $\sigma$. One can prove a virtual analogue of Lemma \ref{lem:push d} by writing
$$
R\pi'_*(\CO(k)) = R\pi_* \Big(\text{Koszul}(\pi^*(W) \otimes \CO(-1), \sigma ) \otimes \CO(k) \Big)
$$
and then applying \eqref{eqn:push d} to compute the expression in the right-hand side. The result is quite simple if $\text{rank } W < \text{rank } V$ and a little messy if $\text{rank } W > \text{rank } V$. In the present paper, we will actually need the boundary case between these two extremes, specifically the following result (see \cite[Lemma 3.4]{zhao1} for a proof).

\begin{lemma}
\label{lem:push d virtual}

Assume $V$ and $W$ are vector bundles of the same rank on a smooth variety $X$, together with a map \eqref{eqn:map of v b} such that \eqref{eqn:are you regular} is a regular section. Then we have
\begin{equation}
\label{eqn:push d virtual}
R\pi'_*(\CO(k)) = \begin{cases} \left( \dots \xrightarrow{\varphi} S^{k-1} V \otimes \wedge^1 W \xrightarrow{\varphi} S^k V \otimes \wedge^0 W \right) &\text{if } k > 0 \\  \frac {\det W}{\det V} \xrightarrow{\det \varphi} \CO_X  &\text{if }k = 0 \\ \left( S^{-k} V^\vee \otimes \wedge^0 W \xrightarrow{\varphi^\vee} S^{-k-1} V^\vee \otimes \wedge^1 W^\vee \xrightarrow{\varphi^\vee} \dots \right) \otimes \frac {\det W}{\det V} &\text{if } k \leq -r \end{cases}
\end{equation}
(the complexes above are both finite, as they terminate whenever we encounter $S^iV$, $S^iV^\vee$ for $i<0$ or $\wedge^jW$, $\wedge^j W^\vee$ for $j > \text{rank }W$).
 
\end{lemma}

\subsection{$K$-theory}

For a smooth projective variety $X$ as before, let us write $K_X$ for the Grothendieck group of coherent sheaves (alternatively locally free sheaves, due to the smoothness of $X$) on $X$. It is a ring with respect to derived tensor product. The $K$-theoretic shadow of the semi-orthogonal decomposition \eqref{eqn:orlov} is the isomorphism
$$
K_{\BP_X(V)} \cong K_X \cdot [\CO] \oplus K_X \cdot [\CO(1)] \oplus \dots \oplus K_X \cdot [\CO(r-1)] 
$$
of $K_X$-modules, where $K_{\BP_X(V)}$ is made into a $K_X$-module via $\pi^*$. As a ring, we have
\begin{equation}
\label{eqn:k of proj}
K_{\BP_X(V)} \cong K_X[z^{\pm 1}] \Big/ \wedge^\bullet \left(\frac Vz\right)
\end{equation}
where $z$ corresponds to $[\CO(1)]$, and we write
\begin{equation}
\label{eqn:wedge}
\wedge^\bullet\left(\frac Vz\right) = \sum_{i=0}^r \frac {[\wedge^i V]}{(-z)^i}
\end{equation}
The $K$-theory class \eqref{eqn:wedge} only depends on the $K$-theory class of $V$. Thus, we can extend this notation to formal differences of vector bundles on $X$, via the formula
\begin{equation}
\label{eqn:wedge virtual}
\wedge^\bullet\left(\frac {V-W}z\right) = \frac {\wedge^\bullet\left(\frac Vz\right)}{\wedge^\bullet\left(\frac Wz\right)}
\end{equation}
The $K$-theoretic version of \eqref{eqn:push d} is the following formula
\begin{equation}
\label{eqn:push k}
\pi_*([\CO(k)]) = \int_{\infty - 0} z^k \wedge^\bullet \left(\frac {-V}z\right)
\end{equation}
where the notation in the right-hand side means that we pick out the constant term of the integrand when expanded near $z = \infty$ and when expanded near $z = 0$, and subtract the two results. Similarly, the $K$-theoretic version of \eqref{eqn:push d virtual} for virtual projectivizations is 
\begin{equation}
\label{eqn:push k virtual}
\pi'_*([\CO(k)]) = \int_{\infty - 0} z^k \wedge^\bullet \left(\frac {W-V}z\right)
\end{equation}

\vskip.2in

\subsection{Nested Quot schemes}
\label{sub:basic nested}

Fix a smooth projective curve $C$, and define the Quot scheme of length $d$ quotients of a fixed rank $r$ vector bundle $V$ as in Subsection \ref{sub:intro 1}.

\begin{definition}
\label{def:nested}

For any positive integers $d_1 < \dots < d_n$, the nested Quot scheme is \begin{equation}
\label{eqn:def nested general}
\quot_{d_1,\dots,d_n} = \Big\{ E^{(d_n)} \subset E^{(d_{n-1})} \subset \dots \subset E^{(d_2)} \subset E^{(d_1)} \subset V \Big\}
\end{equation}
with $\text{length } V/E^{(i)} = i$.

\end{definition}

It is well known that $\quot_{d_1,\dots,d_n}$ is a smooth fine moduli space of dimension $rd_n,$ equipped with the universal flag of subbundles
$$\CE^{(d_n)} \subset \CE^{(d_{n-1})} \subset \dots \subset \CE^{(d_2)} \subset \CE^{(d_1)} \subset \rho^*(V)\, \, \, \text{on} \, \, \, \quot_{d_1,\dots,d_n} \times C.$$
There exist line bundles
\begin{equation}
\label{eqn:line bundles}
\xymatrix{\CL_1, \dots, \CL_n
\ar@{.>}[d]  \\
\quot_{d_1,\dots,d_n}}
\end{equation}
with fibers $\det \Gamma(C,E^{(d_{n-1})}/E^{(d_n)}), \dots, \det \Gamma(C,V/E^{(d_1)})$, that is
\begin{equation}
\CL_i = \det \left (R\pi_*\CE^{(d_{n-i})} - R\pi_*\CE^{(d_{n-i+1})} \right), \, \, \, 1 \leq i \leq n. 
\end{equation}
There is a map
\begin{equation}
\label{eqn:support}
\quot_{d_1,\dots,d_n} \rightarrow C^{(d_n-d_{n-1})} \times C^{(d_{n-1}-d_{n-2})}\times \dots \times C^{(d_1)}
\end{equation}
which keeps track of the support points of the quotients $E^{(d_{n-1})}/E^{(d_n)}, \dots, V/E^{(d_1)}$, respectively. When $d_1,\dots,d_n$ are consecutive integers, we will denote flags as
\begin{equation}
\label{eqn:def nested}
E^{(d+k)} \stackrel{x_1}\subset E^{(d+k-1)} \stackrel{x_2}\subset \dots \stackrel{x_{k-1}}\subset E^{(d+1)} \stackrel{x_k}\subset E^{(d)} \subset V
\end{equation}
where $E' \stackrel{x}\subset E$ means that the quotient $E/E'$ is isomorphic to the skyscraper sheaf at $x \in C$. The most special case is that of $\quot_{d,d+1}$, which parameterizes flags
$$
E^{(d+1)} \stackrel{x}\subset E^{(d)} \subset V
$$
We will often write $\CL = \CL_1$ for the line bundle on $\quot_{d,d+1}$ which parameterizes $\Gamma(C,E^{(d)}/E^{(d+1)})$, and we will write $p_C : \quot_{d,d+1} \rightarrow C$ for the map which records the point $x$ (i.e. the first component of the support map \eqref{eqn:support}).

\subsection{Morphisms between nested Quot schemes}

Since one can always forget part of the flag of subsheaves in \eqref{eqn:def nested general}, there are numerous maps between the schemes $\quot_{d_1,\dots,d_n}$. For instance, an important role will be played by the map
\begin{equation}
\label{eqn:pi minus}
p_- \times p_C : \quot_{d_1,\dots,d_n,d_n+1} \rightarrow \quot_{d_1,\dots,d_n} \times C
\end{equation}
which forgets the sheaf $E^{(d_n+1)}$, but remembers the support point of $E^{(d_n)}/ E^{(d_n+1)}$.

\begin{lemma}
\label{lem:pi minus}

The map $p_- \times p_C$ induces an isomorphism
\begin{equation}
\label{eqn:lem pi minus}
\quot_{d_1,\dots,d_n,d_n+1} \cong \BP_{\quot_{d_1,\dots,d_n} \times C} \left(\CE^{(d_n)} \right)
\end{equation}
where $\CE^{(d_n)}$ denotes the universal sheaf parameterizing the sheaf denoted by $E^{(d_n)}$ in \eqref{eqn:def nested general}. The line bundle $\CL_1$ on the LHS of \eqref{eqn:lem pi minus} is identified with $\CO(1)$ on the RHS.

\end{lemma}

The result above is quite straightforward (see \cite{mn} for a special case, which encapsulates all the features of the general argument). A related scenario is that one considers the following map for all non-negative integers $d,d'$ such that $d+1 < d'$ 
\begin{equation}
\label{eqn:pi minus general}
p_- \times p_C : \quot_{\dots,d,d+1,d',\dots} \rightarrow \quot_{\dots,d,d',\dots} \times C
\end{equation}
which forgets the sheaf $E^{(d+1)}$, but remembers the support point of $E^{(d)}/ E^{(d+1)}$. In \eqref{eqn:pi minus general}, the ellipses in the subscript of $\quot$ denote any sequences of integers smaller than $d$ or greater than $d'$ (which are the same in both domain and codomain).

\begin{lemma}
\label{lem:pi minus general}

The map $p_- \times p_C$ of \eqref{eqn:pi minus general} induces an isomorphism
\begin{equation}
\label{eqn:lem pi minus general}
\quot_{\dots,d,d+1,d',\dots} \cong \BP_{\quot_{\dots,d,d',\dots} \times C}\left(\CE^{(d)}/\CE^{(d')} \right)
\end{equation}
The line bundle on the LHS parameterizing $\Gamma(C,E^{(d)}/E^{(d+1)})$ is identified with $\CO(1)$.

\end{lemma}

In other words, \eqref{eqn:lem pi minus general} states that $\quot_{\dots,d,d+1,d',\dots}$ is regularly embedded (in the sense of Subsection \ref{sub:virtual}) in the projective bundle 
$$
\BP_{\quot_{\dots,d,d',\dots} \times C} \left( \CE^{(d)} \right) 
$$
as the zero locus of the composition
$$
\CE^{(d')} \rightarrow \CE^{(d)} \rightarrow \CO(1).
$$
Finally, the following result provides the opposite scenario to Lemma \ref{lem:pi minus general}. For any non-negative integers $d,d'$ with $d < d'-1$ (and arbitrary integers $<d$ and $>d'$ standing in for the ellipsis) consider the map 
\begin{equation}
\label{eqn:pi plus}
p_+ \times p_C: \quot_{\dots,d,d'-1,d',\dots} \rightarrow \quot_{\dots,d,d',\dots} \times C
\end{equation}
which forgets $E^{(d'-1)}$ in a flag \eqref{eqn:def nested general}, but remembers the support point of $E^{(d'-1)}/E^{(d')}$.

\begin{lemma}
\label{lem:pi plus}

The map $p_+ \times p_C$ of \eqref{eqn:pi plus} induces an isomorphism 
\begin{equation}
\label{eqn:lem pi plus}
\quot_{\dots,d,d'-1,d',\dots} \cong \BP_{\quot_{\dots,d,d',\dots} \times C}\left(\CE^{(d')\vee} \otimes \CK_C - \CE^{(d)\vee} \otimes \CK_C \right)
\end{equation}
The line bundle on the LHS parameterizing $\Gamma(C,E^{(d'-1)}/E^{(d')})$ is identified with $\CO(-1)$.

\end{lemma}

In other words, \eqref{eqn:lem pi plus} states that $\quot_{\dots,d,d'-1,d',\dots}$ is regularly embedded (in the sense of Subsection \ref{sub:virtual}) in the projective bundle 
$$
\BP_{\quot_{\dots,d,d',\dots} \times C}\left(\CE^{(d')\vee} \otimes \CK_C\right)
$$
as the zero locus of the composition
$$
\CE^{(d)\vee} \otimes \CK_C\rightarrow \CE^{(d')\vee} \otimes \CK_C \rightarrow \CO(1)
$$
(the first map is the dual of the inclusion of universal sheaves $\CE^{(d')} \subset \CE^{(d)}$, and the second map is the tautological map on a projectivization). A special case of Lemma \ref{lem:pi plus} corresponding to $d=0$ was proved in \cite[Lemma 1]{mn}, but the general case is an easy generalization, which we leave as an exercise to the reader.

\section{Functors and commutation relations}
\label{sec:main}

\vskip.2in

\subsection{Compositions of functors}
\label{eqn:compositions}

Let us now consider the functors \eqref{eqn:e intro}, namely
\begin{equation}
\label{eqn:e functor}
\te_i : \DD_{\quot_d} \rightarrow \DD_{\quot_{d+1} \times C}, \qquad \te_i(\gamma) = R(p_+ \times p_C)_*\Big(\CL^i \otimes Lp_-^*(\gamma) \Big)
\end{equation}
($\forall d \geq 0$) and study the composition of two such functors, say $\te_i$ and $\te_j$. In the square
\begin{equation}
\label{eqn:quadratic square}
\xymatrix{& \quot_{d-1,d,d+1} \ar[ld]_{p_-} \ar[rd]^{p_+} & \\ \quot_{d-1,d} \times \textcolor{red}{C} \ar[rd]_{p_+ \times \text{Id}_{\textcolor{blue}{C}}} &  & \quot_{d,d+1} \times \textcolor{blue}{C} \ar[ld]^{p_- \times \text{Id}_{\textcolor{red}{C}}} \\ & \quot_{d} \times \textcolor{red}{C} \times \textcolor{blue}{C} &}
\end{equation}
(we color the two copies of $C$ in red and blue to better keep track of them), we claim that the space on the top row is the fiber product of the spaces on the middle row over the space on the bottom row. This is an immediate consequence of the following two statements

\begin{itemize}[leftmargin=*]
    
\item Lemma \ref{lem:pi minus} (respectively \ref{lem:pi plus}) implies that the southwest pointing maps (respectively the southeast pointing maps) are the projectivizations of the same vector bundle (respectively difference of vector bundles)

\item the space on the top row has correct dimension, i.e. the sum of the dimensions of the spaces on the middle row minus the dimension of the space on the bottom row

\end{itemize}

Because of this, the composition of the functors (we color functors with the same color as the copy of $C$ in which they ``take values")
$$
\textcolor{red}{\te_i} \circ \textcolor{blue}{\te_j} : D_{\quot_{d-1}} \xrightarrow{\textcolor{blue}{\te_j}} D_{\quot_{d} \times \textcolor{blue}{C}} \xrightarrow{\textcolor{red}{\te_i} \boxtimes \text{Id}_{\textcolor{blue}{C}}} D_{\quot_{d+1} \times \textcolor{red}{C} \times \textcolor{blue}{C}}
$$
can alternatively be given by the formula
\begin{equation}
\label{eqn:formula quadratic}
\textcolor{red}{\te_i} \circ \textcolor{blue}{\te_j} (\gamma) = R(r_+ \times r_{\textcolor{red}{C} \times \textcolor{blue}{C}})_* \Big( \CL_1^i \otimes \CL_2^j \otimes L r_-^*(\alpha) \Big)
\end{equation}
where the morphisms $r_\pm, r_{\textcolor{red}{C} \times \textcolor{blue}{C}}$ are as in the diagram below
\begin{equation}
\label{eqn:diagram quadratic}
\xymatrix{& \quot_{d-1,d,d+1} \ar[ld]_{r_-} \ar[d]^{r_{\textcolor{red}{C} \times \textcolor{blue}{C}}} \ar[rd]^{r_+} & \\ \quot_{d-1} &\textcolor{red}{C} \times \textcolor{blue}{C} & \quot_{d+1}}
\end{equation}
Similarly, $\textcolor{red}{\te_j} \circ \textcolor{blue}{\te_i}$ is given by a formula analogous to \eqref{eqn:formula quadratic} but with the roles of $(\textcolor{red}{C},i)$ and $(\textcolor{blue}{C},j)$ switched. To compare the functors $\textcolor{red}{\te_i} \circ \textcolor{blue}{\te_j}$ and $\textcolor{blue}{\te_j} \circ \textcolor{red}{\te_i}$, we therefore need to work on a space which allows the switching of the two factors of $C$. This will be achieved in the next Subsection.

\subsection{The quadruple moduli space}
\label{sub:quadruples}

We will now consider the version for curves of the quadruple moduli space of sheaves that was studied for surfaces in \cite{n}. For fixed (but arbitrary) $d \geq 1$, the moduli space $\fY = \fY_{d-1,d,d+1}$ parameterizes quadruples
\begin{equation}
\label{eqn:quadruple}
 \xymatrix{& E' \ar@{^{(}->}[rd]^{y} & \\
E'' \ar@{^{(}->}[ru]^{x} \ar@{^{(}->}[rd]_{y}
& & E  \subset  V \\
& \widetilde{E}' \ar@{^{(}->}[ru]_{x} &} 
	\end{equation}
with the inclusion $E \subset V$ having colength $d-1$, and all other inclusions having colength 1 and supported at the point of $C$ indicated in the diagram. There are line bundles 
$$
\xymatrix{
\CL_1, \CL_2, \tCL_1, \tCL_2 \ar@{.>}[d] \\
\fY}
$$
whose fibers in the notation of \eqref{eqn:quadruple} are $\Gamma(C,E'/E'')$, $\Gamma(C,E/E')$, $\Gamma(C,\widetilde{E}'/E'')$ and $\Gamma(C,E/\widetilde{E}')$, respectively. Note that
\begin{equation}
\label{eqn:line bundle identity}
\CL_1 \otimes \CL_2 \cong \tCL_1 \otimes \tCL_2
\end{equation}
We have projection maps
\begin{equation}
\label{eqn:blow-up}
\pi^\uparrow, \pi^\downarrow : \fY \rightarrow \quot_{d-1,d,d+1}
\end{equation}
that forget $\widetilde{E}'$ and $E'$, respectively. It was shown in \cite[Section 6]{mn} that the maps above are the blow-ups of the smooth subvariety
\begin{equation}
\label{eqn:def z}
\Big\{x = y \text { and } E/E'' \cong \BC_x \oplus \BC_x\Big\}
\end{equation}
of the smooth variety $\quot_{d-1,d,d+1}$. As such, $\fY$ is itself smooth and projective of dimension $r(d+1)$, and moreover we have
\begin{equation}
\label{eqn:push o up and down}
R\pi^\uparrow_*(\CO_{\fY}) = \CO_{\quot_{d-1,d,d+1}} 
\end{equation}
with the RHS in homological degree 0 (and similarly for $\pi^\downarrow$). As a consequence, the compositions $\textcolor{red}{\te_i} \circ \textcolor{blue}{\te_j}$ and $\textcolor{blue}{\te_j} \circ \textcolor{red}{\te_i}$ can alternatively be given by the following formulas
\begin{equation}
\label{eqn:composition e quadruple 1}
\textcolor{red}{\te_i} \circ \textcolor{blue}{\te_j} (\gamma) = R(s_+ \times s_{\textcolor{red}{C} \times \textcolor{blue}{C}})_* \Big( \CL_1^i \otimes \CL_2^j \otimes L s_-^*(\gamma) \Big)
\end{equation}
\begin{equation}
\label{eqn:composition e quadruple 2}
\textcolor{blue}{\te_j} \circ \textcolor{red}{\te_i} (\gamma) = R(s_+ \times s_{\textcolor{red}{C} \times \textcolor{blue}{C}})_* \Big( \tCL_1^j \otimes \tCL_2^i \otimes L s_-^*(\gamma) \Big)
\end{equation}
where the maps are as in the following diagram
\begin{equation}
\label{eqn:diagram quadruple}
\xymatrix{& \fY \ar[ld]_{s_-} \ar[d]^{s_{\textcolor{red}{C} \times \textcolor{blue}{C}}} \ar[rd]^{s_+} & \\ \quot_{d-1} &\textcolor{red}{C} \times \textcolor{blue}{C} & \quot_{d+1}}
\end{equation}
We will now use formulas \eqref{eqn:composition e quadruple 1} and \eqref{eqn:composition e quadruple 2} to compare the functors $\textcolor{red}{\te_i} \circ \textcolor{blue}{\te_j}$ and $\textcolor{blue}{\te_j} \circ \textcolor{red}{\te_i}$.

\subsection{The diagonal and exceptional divisors}

There are two smooth divisors on $\fY$. One is the exceptional divisor of the blow-up \eqref{eqn:blow-up}, namely
\begin{equation}
\label{eqn:exc}
\text{Exc} = \Big\{\text{quadruples as in \eqref{eqn:quadruple} with } x=y \text{ and }E/E'' \cong \BC_x \oplus \BC_x \Big\} \subset \fY
\end{equation}
The other important divisor on $\fY$ is the so-called diagonal divisor
\begin{equation}
\label{eqn:diag}
\text{Diag} = \Big\{\text{quadruples as in \eqref{eqn:quadruple} with } E' = \widetilde{E}'\Big\} \subset \fY
\end{equation}
which is isomorphic to
\begin{equation}
\label{eqn:punctual}
\quot_{d-1,d,d+1}^\bullet = \Big\{E'' \stackrel{x}\subset E' \stackrel{x}\subset E \subset V \Big\}
\end{equation}
The scheme \eqref{eqn:punctual} is smooth for the same reason that $\quot_{d-1,d,d+1}$ is smooth, namely because it is isomorphic to the projectivization of the rank $r$ vector bundle parameterizing $E'_x$ on $\quot_{d-1,d}$. The union of the divisors \eqref{eqn:exc} and \eqref{eqn:diag} is 
\begin{equation}
\label{eqn:exc diag}
\text{Exc} \cup \text{Diag} = \{x = y\} = s^{-1}_{\textcolor{red}{C} \times \textcolor{blue}{C}}(\Delta)
\end{equation}

\begin{lemma}
\label{lem:maps of line bundles}

(\cite[Lemma 6]{mn}) On $\fY$, there exist maps of line bundles
\begin{equation}
\label{eqn:maps diag}
\CL_1 \rightarrow \tCL_2 \quad \text{and} \quad \tCL_1 \rightarrow \CL_2
\end{equation}
whose associated Weil divisor is $\emph{Diag}$, and maps of line bundles
\begin{equation}
\label{eqn:maps exc}
\tCL_2 \rightarrow \CL_1(s_{\textcolor{red}{C} \times \textcolor{blue}{C}}^*(\Delta)) \quad \text{and} \quad \CL_2 \rightarrow \tCL_1(s_{\textcolor{red}{C} \times \textcolor{blue}{C}}^*(\Delta))
\end{equation}
whose associated Weil divisor is $\emph{Exc}$.

\end{lemma}

The diagonal divisor plays an important role in the following result.

\begin{proposition}
\label{prop:quadratic relation e}

For any $i \leq j \in \BZ$, there is a natural transformation
\begin{equation}
\label{eqn:quadratic relation e}
\textcolor{blue}{\te_j} \circ \textcolor{red}{\te_i}  \longrightarrow \textcolor{red}{\te_i} \circ \textcolor{blue}{\te_j} 
\end{equation}
of functors $D_{\quot_{d-1}} \rightarrow D_{\quot_{d+1} \times \textcolor{red}{C} \times \textcolor{blue}{C}}$, whose cone is filtered with associated graded
\begin{equation}
\label{eqn:cone ee}
\bigoplus_{k=i}^{j-1} \Delta_*\left(\te_k \circ \te_{i+j-k} \Big|_\Delta\right)
\end{equation}

\end{proposition}

\begin{proof} Because of \eqref{eqn:composition e quadruple 1} and \eqref{eqn:composition e quadruple 2}, in order to construct the required natural transformations between the functors $\textcolor{blue}{\te_j} \circ \textcolor{red}{\te_i}$ and $\textcolor{red}{\te_i} \circ \textcolor{blue}{\te_j}$, it suffices to construct morphisms
$$
\tCL_1^j \otimes \tCL_2^i \quad \text{and} \quad \CL_1^i \otimes \CL_2^j
$$
on $\fY$. To this end, we take the second map in \eqref{eqn:maps diag} and raise it to the $(j-i)$-th power
$$
\tCL_1^j \otimes \tCL_2^i  \rightarrow \tCL_1^{j-1} \otimes \tCL_2^i \otimes \CL_2  \rightarrow \dots \rightarrow \tCL_1^{i+1} \otimes \tCL_2^{i} \otimes \CL_2^{j-i-1}  \rightarrow \underbrace{\tCL_1^i \otimes \tCL_2^{i}}_{\cong \CL_1^i \otimes \CL_2^i} \otimes \CL_2^{j-i} \cong \CL_1^i \otimes \CL_2^j
$$
(the isomorphisms use \eqref{eqn:line bundle identity}). The cone of the above composition is filtered with associated graded equal to the direct sum of the cones of the individual maps, namely
$$
\bigoplus_{k=i}^{j-1} \text{Cone}\left[\tCL_1^{k+1} \otimes \tCL_2^i \otimes \CL_2^{j-k-1} \rightarrow \tCL_1^{k} \otimes \tCL_2^i \otimes \CL_2^{j-k} \right] \cong \bigoplus_{k=i}^{j-1} \CL_1^k \otimes \CL_2^{i+j-k} \otimes \CO_{\text{Diag}}
$$
where we used the fact that $\CL_1 \cong \tCL_1$ and $\CL_2 \cong \tCL_2$ on the locus $\text{Diag}$. The cones above on $\fY$ give rise, via the assignment
$$
\Psi \in \DD_{\fY} \quad \leadsto \quad \text{functor } R(s_+ \times s_{\textcolor{red}{C} \times \textcolor{blue}{C}})_* \Big( \Psi \otimes L s_-^*(-) \Big)
$$
precisely to the functors $\DD_{\quot_{d-1}} \rightarrow \DD_{\quot_{d+1} \times \textcolor{red}{C} \times \textcolor{blue}{C}}$ specified in \eqref{eqn:cone ee}.

\end{proof}

The natural analogue of Proposition \ref{prop:quadratic relation e} holds with all $\te_a \circ \te_b$ replaced by $\tf_b \circ \tf_a$ in formulas \eqref{eqn:quadratic relation e} and \eqref{eqn:cone ee}, where
\begin{equation}
\label{eqn:f functor}
\tf_i : \DD_{\quot_{d+1}} \rightarrow \DD_{\quot_d \times C}, \qquad \tf_i(\gamma) = R(p_- \times p_C)_*\Big(\CL^i \otimes Lp_+^*(\gamma) \Big)
\end{equation}
are the functors defined in \eqref{eqn:f intro}. 

\subsection{An ``orthogonal" viewpoint} 
\label{sub:orthogonal}

Besides the maps \eqref{eqn:blow-up}, which forget the top/bottom sheaves in \eqref{eqn:quadruple}, we also have the maps
\begin{equation}
\label{eqn:diagram quadruple pm}
\xymatrix{& \fY \ar[ld]_{\pi_+} \ar[rd]^{\pi_-} & \\ \fZ_+ & & \fZ_-}
\end{equation}
which forget the right/left sheaves in \eqref{eqn:quadruple}. Specifically, the varieties $\fZ_{\pm}$ parameterize
\begin{align}
&\fZ_+ = \Big\{ E' \stackrel{x}\supset E'' \stackrel{y}\subset \widetilde{E}' \Big\} \label{eqn:z plus} \\
&\fZ_- = \Big\{ E' \stackrel{y}\subset E \stackrel{x}\supset \widetilde{E}' \Big\} \label{eqn:z minus}
\end{align}
with all sheaves above being finite colength subsheaves of $V$. Diagram \eqref{eqn:diagram quadruple pm} and the spaces \eqref{eqn:z plus}--\eqref{eqn:z minus} are very similar to the analogous notions treated in \cite{zhao2} in the case of surfaces, as will be the following discussion. We will use the notation $\CL_1,\tCL_1$ (respectively $\CL_2, \tCL_2$) for the line bundles on $\fZ_-$ (respectively $\fZ_+$) inherited from the same-named line bundles on $\fY$. As a consequence of the isomorphism
\begin{equation}
\label{eqn:product of proj}
\fZ_- \cong \BP_{\quot_{d-1} \times \textcolor{red}{C}}(\textcolor{red}{\CE}) \times_{\quot_{d-1}} \BP_{\quot_{d-1} \times \textcolor{blue}{C}}(\textcolor{blue}{\CE})
\end{equation}
(which is simply a double application of Lemma \ref{lem:pi minus}) we conclude the following.

\begin{proposition}
\label{prop:z minus}

The variety $\fZ_-$ is smooth of dimension $r(d+1)$. 
    
\end{proposition}

The proper map (of smooth varieties of the same dimension, by Proposition \ref{prop:z minus})
$$
\pi_- : \fY \rightarrow \fZ_-
$$
is an isomorphism outside of the locus $(E',y) = (\tE',x)$. Under the isomorphism \eqref{eqn:product of proj}, this locus is the diagonal copy of $\BP_{\quot_{d-1} \times C}(\CE)$ in $\BP_{\quot_{d-1} \times \textcolor{red}{C}}(\textcolor{red}{\CE}) \times_{\quot_{d-1}} \BP_{\quot_{d-1} \times \textcolor{blue}{C}}(\textcolor{blue}{\CE})$, and is thus isomorphic to the codimension $r$ smooth variety $\quot_{d-1,d}$. Thus, $\pi_-$ is birational. Since both $\fY$ and $\fZ_-$ are smooth, we conclude that
\begin{equation}
\label{eqn:push z minus}
R\pi_{-*}(\CO_{\fY}) = \CO_{\fZ_-} 
\end{equation}

\subsection{The two components of $\fZ_+$: part 1}
\label{sub:two components}

Meanwhile, let us write $\fZ_+$ as the union
\begin{equation}
\label{eqn:two components}
\fZ_+ = \fZ_+' \cup \fZ_+''
\end{equation}
where we consider the following subschemes of $\fZ_+$
\begin{align*} 
&\fZ_+' = \overline{\Big\{(E',x) \neq (\widetilde{E}',y) \Big\}} \\
&\fZ_+'' = \Big\{(E',x) = (\widetilde{E}',y) \Big\}
\end{align*}
Note that $\fZ_+'' \cong \quot_{d,d+1}$, which is smooth of dimension $r(d+1)$. Meanwhile, $\fZ_+' = \text{Im }\pi_+$ and in what follows we will write
$$
\pi'_+ : \fY \rightarrow \fZ_+'
$$
for the restriction of the map $\pi_+$ to its image. Since the map $\pi'_+$ is one-to-one above the open dense locus $\{(E',x) \neq (\widetilde{E}',y)\}$, we conclude that $\fZ_+'$ also has dimension $r(d+1)$. Thus, \eqref{eqn:two components} is a decomposition of $\fZ_+$ into two components of the same dimension. The intersection of these two components is precisely the locus
$$
D = \Big\{E'' \stackrel{x}\subset E' \subset V, \text{ where }x \in \text{Supp}(V/E') \Big\} 
$$
which has dimension $r(d+1)-1$ and is thus a divisor in both $\fZ_+'$ and $\fZ_+''$. The following result will be used in the proofs of Propositions \ref{prop:prime} and \ref{prop:e and f} and will be proved in the Appendix, as it relies on a local argument.

\begin{proposition}
\label{prop:z plus}

The variety $\fZ_+'$ has rational singularities, i.e.
\begin{equation}
\label{eqn:push z plus}
R\pi'_{+*}(\CO_{\fY}) = \CO_{\fZ_+'}
\end{equation}
    
\end{proposition}

\subsection{The two components of $\fZ_+$: part 2}

There are several vector bundles on $\fZ_+'' \cong \quot_{d,d+1}$, constructed as follows. Take the universal rank $r$ vector bundles $\CE''$ and $\CE'$ on $\fZ_+'' \times C$ parameterizing the sheaves $E''$ and $E'$, respectively. Then if we restrict to the graph of the map $p_C : \fZ_+'' \rightarrow C$ that remembers $x$, we obtain rank $r$ vector bundles on $\fZ_+''$ that we will denote by $\CE''_x$ and $\CE_x'$, respectively. Moreover, we will write $V$ instead of the more appropriately named vector bundle $p_C^*(V)$ on $\fZ_+''$.

\begin{proposition}
\label{prop:prime}

On $\fZ_+''$, we have the formula
\begin{equation}
\label{eqn:double prime}
\frac {\det \CE'_x}{\det V} \cong \CO_{\fZ_+''}(-D)
\end{equation}
while on $\fZ_+'$ we have the following formula
\begin{equation}
\label{eqn:prime}
R\pi'_{+*}\left(\CL_1 \otimes \tCL^{-1}_2 \right) = \CO_{\fZ_+'}(-D) \end{equation}
with the RHS in homological degree 0.

\end{proposition}

\begin{proof} Under the isomorphism $\fZ_+'' \cong \quot_{d,d+1}$, the smooth divisor $D \subset \fZ_+''$ corresponds to the image of the projection map
$$
p_+ : \quot_{d-1,d,d+1}^\bullet \rightarrow \quot_{d,d+1}
$$
By analogy with Lemma \ref{lem:pi plus}, $p_+$ is the projectivization of the virtual vector bundle
$$
\BP_{\quot_{d,d+1}}\left({\CE'_x}^{\vee} \otimes \CK_C - V^\vee \otimes \CK_C \right)
$$
Therefore, the fibers of $p_+$ are either empty or projective spaces. This implies that
\begin{equation}
\label{eqn:temp 1}
R^0p_{+*}\left(\CO_{\quot_{d-1,d,d+1}^\bullet} \right) = \CO_D
\end{equation}
On the other hand, Lemma \ref{lem:push d virtual} implies that
\begin{equation}
\label{eqn:temp 2}
Rp_{+*}\left(\CO_{\quot_{d-1,d,d+1}^\bullet} \right) = \left[ \frac {\det \CE'_x}{\det V} \rightarrow \CO_{\fZ_+''} \right]
\end{equation}
with the arrow induced by the tautological map $\CE'_x \rightarrow V$. Since the arrow in the right-hand side of \eqref{eqn:temp 2} is a map of line bundles on a smooth scheme, it must be injective. Therefore, the right-hand side of \eqref{eqn:temp 2} is concentrated in degree 0, and combining this with \eqref{eqn:temp 1} yields
\begin{equation}
\label{eqn:temp 3}
\frac {\det \CE'_x}{\det V} \cong \CO_{\fZ_+''}(-D)
\end{equation}
as we needed to prove. Let us now turn to proving \eqref{eqn:prime}. We start from Lemma \ref{lem:maps of line bundles}, which ensures that we have a short exact sequence of coherent sheaves
\begin{equation}
\label{eqn:ses temp}
0 \rightarrow \CL_1 \otimes \tCL_2^{-1} \rightarrow \CO_{\fY} \rightarrow \CO_{\text{Diag}} \rightarrow 0
\end{equation}
on $\fY$. Recall that we have a commutative diagram
\begin{equation}
\label{eqn:cd temp}
\xymatrix{\text{Diag} \ar[r]^-{\cong} \ar[d]_{\pi_+'} & \quot_{d-1,d,d+1}^{\bullet} \ar[d]^{p^+} \\ D \ar[r]^-{\cong} & \quot_{d,d+1}}
\end{equation}
Pushing forward the short exact sequence \eqref{eqn:ses temp} via $\fY \xrightarrow{\pi_+'} \fZ_+'$ yields an exact triangle
$$
R\pi_{+*}' ( \CL_1 \otimes \tCL_2^{-1} ) \rightarrow R\pi_{+*}' ( \CO_{\fY} ) \rightarrow R\pi_{+*}' ( \CO_{\text{Diag}} ) \rightarrow R\pi_{+*}' ( \CL_1 \otimes \tCL_2^{-1} )[1]
$$
Formula \eqref{eqn:push z plus} tells us that the second term above is $\CO_{\fZ_+'}$, while formula \eqref{eqn:double prime} and the commutative diagram \eqref{eqn:cd temp} tell us that the third term above is $\CO_D$ (with the map between terms being the standard one). Therefore, we conclude \eqref{eqn:prime}.

\end{proof}

\subsection{The commutators of $\te_i$ and $\tf_j$}
\label{sub:e and f}

Using the geometric results developed in the previous Subsections, we will work out the commutators of the functors $\te_i$ and $\tf_j$. 

\begin{proposition}
\label{prop:e and f}

There is a natural transformation of functors
\begin{equation}
\label{eqn:comm e and f}
\textcolor{blue}{\tf_j} \circ \textcolor{red}{\te_{i}} \xleftrightharpoons[\text{if }i+j \geq 0]{\text{if }i+j < 0}  \textcolor{red}{\te_{i}} \circ \textcolor{blue}{\tf_j}
\end{equation}
whose cone is $\Delta_*(\text{tensoring with } \alpha_{i+j})$, where $\Delta : C \hookrightarrow \textcolor{red}{C} \times \textcolor{blue}{C}$ is the diagonal and
\begin{equation}
\label{eqn:alpha}
\Big\{ \alpha_\ell \in \DD_{\quot \times C} \Big\}_{\ell \in \BZ}
\end{equation}
are objects which admit filtrations with associated graded object
\begin{equation}
\label{eqn:bigger}
\bigoplus_{k=0}^{\ell} S^k\CE \otimes \left[S^{\ell-k} {\CE} \otimes \wedge^0 V \rightarrow S^{\ell-k-1} {\CE} \otimes \wedge^1 V \rightarrow \dots  \right] \otimes \CK_C^{-\ell+k} \otimes \frac {\det \CE}{\det V} 
\end{equation}
if $\ell \geq 0$, and 
\begin{equation}
\label{eqn:smaller}
\bigoplus_{k=r}^{-\ell} \frac {S^{k-r} {\CE}^\vee}{\det \CE} [-r+1] \otimes \left[ \dots \rightarrow S^{-\ell-k-1} {\CE}^\vee \otimes \wedge^1 V^\vee  \rightarrow S^{-\ell-k} {\CE}^\vee \otimes \wedge^0 V^\vee \right] \otimes \CK_C^{-\ell-k}
\end{equation}
if $\ell < 0$ (above, $\CE$ denotes the universal rank $r$ vector bundle on $\quot \times C$).

\end{proposition}

\begin{proof} The varieties $\fY, \fZ_-, \fZ_+, \fZ_+', \fZ_+''$ all admit maps to
$$
\quot_d \times \quot_d \times \textcolor{red}{C} \times \textcolor{blue}{C}
$$
by remembering only $(E',\tE',x,y)$. We will write
$$
\sign_* : \DD_{*} \rightarrow \DD_{\quot_d \times \quot_d \times \textcolor{red}{C} \times \textcolor{blue}{C}}
$$
for any $* \in \{\fY,\fZ_-,\fZ_+,\fZ_+',\fZ_+''\}$ for the derived push-forward of the aforementioned map. With this in mind, to obtain \eqref{eqn:comm e and f} we need to construct a morphism \footnote{Here we are implicitly using the formalism of correspondences (for $Z_1,Z_2$ smooth and projective)
$$
\Gamma \in \DD_{Z_1 \times Z_2} \quad \leadsto \quad \DD_{Z_1} \xleftarrow{\Phi_\Gamma} \DD_{Z_2}, \ \Phi_\Gamma = \text{proj}_{1*} \left(\Gamma \otimes \text{proj}_2^*\right)
$$
in which morphisms $\Gamma \rightarrow \Gamma'$ give rise to natural transformations $\Phi_{\Gamma} \Rightarrow \Phi_{\Gamma'}$. In this sense, the left and right-hand sides of \eqref{eqn:want big} are the objects in $\DD_{\quot_d \times \quot_d \times \textcolor{red}{C} \times \textcolor{blue}{C}}$ which give rise to the functors $\DD_{\quot_d} \rightarrow \DD_{\quot_d \times \textcolor{red}{C} \times \textcolor{blue}{C}}$ that feature in the left and right-hand sides of \eqref{eqn:comm e and f}, respectively, so the morphism in \eqref{eqn:want big} gives rise to the natural transformation in \eqref{eqn:comm e and f}.}
\begin{equation}
\label{eqn:want big}
\sign_{\fZ_+} \CL_1^i \otimes \tCL_1^j \xleftrightharpoons[\text{if }i+j \geq 0]{\text{if }i+j < 0} \sign_{\fZ_-} \tCL_2^i \otimes \CL_2^j 
\end{equation}
When $i+j \geq 0$, we will construct \eqref{eqn:want big} as the composition of the following morphisms
\begin{equation}
\label{eqn:want big 1}
\sign_{\fZ_+} \CL_1^i \otimes \tCL_1^j \rightarrow \sign_{\fZ_+'} \CL_1^i \otimes \tCL_1^j 
\end{equation}
\begin{equation}
\label{eqn:want big 2}
\sign_{\fZ_+'} \CL_1^i \otimes \tCL_1^j \stackrel{\eqref{eqn:push z plus}}\cong \sign_{\fY} \CL_1^i \otimes \tCL_1^j
\end{equation}
\begin{equation}
\label{eqn:want big 3}
\sign_{\fY} \CL_1^i \otimes \tCL_1^j \xrightarrow{\eqref{eqn:maps diag}} \sign_{\fY} \tCL_2^i \otimes \CL_2^j
\end{equation}
\begin{equation}
\label{eqn:want big 4}
\sign_{\fY}  \tCL_2^i \otimes \CL_2^j  \stackrel{\eqref{eqn:push z minus}}{\cong} \sign_{\fZ_-} \tCL_2^i \otimes \CL_2^j 
\end{equation}
The second and fourth morphisms above are isomorphisms, so their cones are trivial. Meanwhile, by Proposition \ref{prop:prime}, we have a short exact sequence
\begin{equation}
\label{eqn:sesz}
0 \rightarrow \CO_{\fZ_+''} \otimes \frac {\det \CE'_x}{\det V} \rightarrow \CO_{\fZ_+} \rightarrow \CO_{\fZ_+'} \rightarrow 0
\end{equation}
and so the cone of \eqref{eqn:want big 1} is
\begin{equation}
\label{eqn:big result 1}
\sign_{\fZ_+''} \CL^{i+j} \otimes \frac {\det \CE'_x}{\det V} \cong \Delta_{\quot_d \times C *} \left( S^{i+j}\CE' \otimes \frac {\det \CE'}{\det V} \right)
\end{equation}
The isomorphism in \eqref{eqn:big result 1} uses $\fZ_+'' \cong \quot_{d,d+1}$ and Lemmas \ref{lem:push d} and \ref{lem:pi minus}. Finally, because of the maps
$$
\CL_1^i \otimes \tCL_1^j \rightarrow \CL_1^i \otimes \tCL_1^{j-1} \otimes \CL_2 \rightarrow \dots \rightarrow \CL_1^i \otimes \tCL_1^{-i-1} \otimes \CL_2^{i+j-1} \rightarrow \CL_1^i \otimes \tCL_1^{-i} \otimes \CL_2^{i+j} \stackrel{\eqref{eqn:line bundle identity}}\cong \tCL_2^i \otimes \CL_2^j
$$
(all induced by \eqref{eqn:maps diag}) the cone of \eqref{eqn:want big 3} is filtered with associated graded object
\begin{equation}
\label{eqn:the object above}
\bigoplus_{k=0}^{i+j-1} \sign_{\text{Diag}} \CL_1^k \otimes \CL_2^{i+j-k} 
\end{equation}
where $\sign_{\text{Diag}}$ denotes the derived push-forward associated to the map 
$$
\text{Diag} \rightarrow \quot_d \times \quot_d \times C \times C
$$
that remembers $(E',E',x,x)$. Because $\text{Diag} \cong \quot_{d-1,d,d+1}^{\bullet}$, then Lemmas \ref{lem:push d}, \ref{lem:push d virtual}, \ref{lem:pi minus}, \ref{lem:pi plus} imply that the object \eqref{eqn:the object above} is equal to $\Delta_{\quot_d \times C *}$ applied to
\begin{equation}
\label{eqn:big result 2}
\bigoplus_{k=0}^{i+j-1} S^k\CE' \otimes \left[S^{i+j-k} {\CE'} \otimes \wedge^0 V \rightarrow S^{i+j-k-1} {\CE'} \otimes \wedge^1 V \rightarrow \dots  \right] \otimes \CK_C^{-i-j+k} \otimes \frac {\det \CE'}{\det V}
\end{equation}
Putting together \eqref{eqn:big result 1} and \eqref{eqn:big result 2} defines the object $\alpha_{i+j}$, filtered with associated graded object \eqref{eqn:bigger}. When $i+j < 0$, we will construct \eqref{eqn:want big} as the composition of the following morphisms
\begin{align}
&\sign_{\fZ_-} \tCL_2^i \otimes \CL_2^j
\stackrel{\eqref{eqn:push z minus}}\cong \sign_{\fY} \tCL_2^i \otimes \CL_2^j  \label{eqn:small 1} \\
&\sign_{\fY} \tCL_2^i \otimes \CL_2^j  \stackrel{\eqref{eqn:maps diag}}\longrightarrow \sign_{\fY} \CL_1^{i+1} \otimes \tCL_1^j \otimes \tCL_2^{-1}  \label{eqn:small 2} \\
&\sign_{\fY} \CL_1^{i+1} \otimes \tCL_1^j \otimes \tCL_2^{-1} \stackrel{\eqref{eqn:prime}}{\cong} \sign_{\fZ_+'} \CL_1^{i} \otimes \tCL_1^j \otimes \CO(-D) \label{eqn:small 3} \\
&\sign_{\fZ_+'} \CL_1^{i} \otimes \tCL_1^j \otimes \CO(-D) \rightarrow \sign_{\fZ_+} \CL_1^{i} \otimes \tCL_1^j  \label{eqn:small 4}
\end{align}
Because of 
$$
\tCL_2^i \otimes \CL_2^j \rightarrow \tCL_2^{i+1} \otimes \CL_1^{-1} \otimes \CL_2^j \rightarrow \dots \rightarrow \tCL_2^{-j-1} \otimes \CL_1^{i+j+1} \otimes \CL_2^j \stackrel{\eqref{eqn:line bundle identity}}{\cong} \CL_1^{i+1} \otimes \tCL_1^j \otimes \tCL_2^{-1}
$$
the cone of \eqref{eqn:small 2} can be filtered with associated graded object
$$ 
\bigoplus_{k=1}^{-i-j-1} \sign_{\text{Diag}}  \CL_1^{-k} \otimes \CL_2^{i+j+k} 
$$
where $\sign_{\text{Diag}}$ denotes the derived push-forward associated to the map $\text{Diag} \rightarrow \quot_d \times C$ that remembers $(E',x)$. Because $\text{Diag} \cong \quot_{d-1,d,d+1}^\bullet$, then Lemmas \ref{lem:push d}, \ref{lem:push d virtual}, \ref{lem:pi minus}, \ref{lem:pi plus} imply that the object above is equal to $\Delta_{\quot_d \times C *}$ applied to
\begin{equation}
\label{eqn:small result 1}
\bigoplus_{k=r}^{-i-j-1} \frac {S^{k-r} {\CE'}^\vee}{\det \CE'} [-r+1] \otimes \left[ \dots \rightarrow S^{-i-j-k-1} {\CE'}^\vee \otimes \wedge^1 V^\vee  \rightarrow S^{-i-j-k} {\CE'}^\vee \otimes \wedge^0 V^\vee \right] \otimes \CK_C^{-i-j-k}
\end{equation}
Finally, by Proposition \ref{prop:prime}, we have a short exact sequence
$$
0 \rightarrow \CO_{\fZ_+'(-D)} \rightarrow \CO_{\fZ_+} \rightarrow \CO_{\fZ_+''} \rightarrow 0
$$
Because of this, the cone of \eqref{eqn:small 4} is given by 
\begin{equation}
\label{eqn:small result 2}
\sign_{\fZ_+''} \CL^{i+j} \cong  \frac {S^{-i-j-r} {\CE'}^\vee}{\det \CE'} [-r+1]
\end{equation}
with the isomorphism due to the fact that $\fZ_+'' \cong \quot_{d,d+1}$ and Lemmas \ref{lem:push d} and \ref{lem:pi minus}. Putting together \eqref{eqn:small result 1} and \eqref{eqn:small result 2} gives us the object $\alpha_{i+j}$, with associated graded object \eqref{eqn:smaller}.

\end{proof}

\subsection{The commutators of $\te_i$, $\tf_i$ with $\tm_j$}
\label{sub:e,f and m}

Recall the functors \eqref{eqn:m intro}, namely
\begin{equation}
\label{eqn:m functor}
\DD_{\quot_d} \xrightarrow{\tm_i} \DD_{\quot_d \times C}, \quad \tm_i(\gamma) = \wedge^i\CE \otimes \pi^*(\gamma)
\end{equation}
where $\pi : \quot_d \times C \rightarrow \quot_d$ is the standard projection. It is clear that the cone of \eqref{eqn:comm e and f} can be expressed as complexes built out of various compositions of the functors $\tm_i$ (if we allow division by the functor $\tm_r$, which is invertible because $\wedge^r \CE$ is a line bundle). In the present Subsection, we will calculate the commutation relation of the functors $\tm_j$ with $\te_i$ and $\tf_i$. As before, let $\Delta: C \hookrightarrow \textcolor{red}{C} \times \textcolor{blue}{C}$ denote the diagonal.

\begin{proposition}
\label{prop:comm e and m}
For all $d\geq 0$, $i\in \BZ$, $j \in \{0,\dots,r\}$, there is a natural transformation
\begin{equation}
\label{eqn:e and m natural transformation}
\textcolor{blue}{\tm_j} \circ \textcolor{red}{\te_i} \longrightarrow \textcolor{red}{\te_i} \circ \textcolor{blue}{\tm_j}
\end{equation}
of functors $\DD_{\quot_d} \rightarrow \DD_{\quot_{d+1} \times \textcolor{red}{C} \times \textcolor{blue}{C}}$, whose cone is
\begin{equation}
\label{eqn:e and m cone}
\Delta_* \left( \te_{i+1} \circ \tm_{j-1} \Big|_\Delta \rightarrow \dots \rightarrow \te_{i+j} \circ \tm_{0} \Big|_\Delta \right) 
\end{equation}
Moreover, we have an isomorphism of functors $\DD_{\quot_d} \rightarrow \DD_{\quot_{d+1} \times \textcolor{red}{C} \times \textcolor{blue}{C}}$
\begin{equation}
\label{eqn:iso}
\textcolor{blue}{\tm_r} \circ \textcolor{red}{\te_i} \cong  \textcolor{red}{\te_i} \circ \textcolor{blue}{\tm_r} (-\Delta)
\end{equation}

\end{proposition}

\begin{proof} The universal sheaves $\textcolor{blue}{\CE'}$ and $\textcolor{blue}{\CE}$ on $\quot_{d,d+1} \times \textcolor{blue}{C}$ fit into a short exact sequence
\begin{equation}
\label{eqn:colored ses}
0\rightarrow \textcolor{blue}{\CE'} \rightarrow \textcolor{blue}{\CE} \rightarrow \textcolor{blue}{\Gamma}_*(\CL) \rightarrow 0
\end{equation}
where $\textcolor{blue}{\Gamma}$ is the graph of the map $p_C : \quot_{d,d+1} \rightarrow \textcolor{blue}{C}$. By definition, the compositions
$$
\textcolor{blue}{\tm_j} \circ \textcolor{red}{\te_i}  \quad \text{and} \quad \textcolor{red}{\te_i} \circ \textcolor{blue}{\tm_j}
$$
are given by the following correspondences on $\quot_{d,d+1} \times \textcolor{blue}{C}$
$$
\CL^i \otimes \wedge^j \textcolor{blue}{\CE'}  \quad \text{and} \quad \CL^i \otimes \wedge^j \textcolor{blue}{\CE} 
$$
The short exact sequence \eqref{eqn:colored ses} induces a long exact sequence
\begin{equation}
\label{eqn:colored wedge ses}
0\rightarrow \wedge^j \textcolor{blue}{\CE'} \rightarrow \wedge^j \textcolor{blue}{\CE} \rightarrow \wedge^{j-1} \textcolor{blue}{\CE} \otimes \textcolor{blue}{\Gamma}_*(\CL) \rightarrow \dots \rightarrow \wedge^{0} \textcolor{blue}{\CE} \otimes \textcolor{blue}{\Gamma}_*(\CL^j) \rightarrow 0
\end{equation}
for all $j \in \{0,\dots,r\}$. Tensoring \eqref{eqn:colored wedge ses} with $\CL^i$ induces a long exact sequence
$$
0  \rightarrow \textcolor{blue}{\tm_j} \circ \textcolor{red}{\te_i} \rightarrow \textcolor{red}{\te_i} \circ \textcolor{blue}{\tm_j} \rightarrow \Delta_* \left( \te_{i+1} \circ \tm_{j-1} \Big|_\Delta \right) \rightarrow \dots \rightarrow \Delta_* \left( \te_{i+j} \circ \tm_{0} \Big|_\Delta \right) \rightarrow 0
$$
which is precisely the content of equations \eqref{eqn:e and m natural transformation} and \eqref{eqn:e and m cone}. Finally, taking the determinant of the short exact sequence \eqref{eqn:colored ses} gives us
\begin{equation}
\label{eqn:blue}
\det \textcolor{blue}{\CE'} \cong \det \textcolor{blue}{\CE} \otimes \CO(-\textcolor{blue}{\Gamma})
\end{equation}
Translating the isomorphism \eqref{eqn:blue} of objects in $\DD_{\quot_{d,d+1} \times \textcolor{blue}{C}}$ into an isomorphism of functors $\DD_{\quot_d} \rightarrow \DD_{\quot_{d+1} \times \textcolor{red}{C} \times \textcolor{blue}{C}}$ yields precisely \eqref{eqn:iso}.

\end{proof}

The obvious analogues of formulas \eqref{eqn:e and m natural transformation}, \eqref{eqn:e and m cone}, \eqref{eqn:iso} hold with all $\te_a \circ \tm_b$ and $\tm_b \circ \te_a$ replaced by $\tm_b \circ \tf_a$ and $\tf_a \circ \tm_b$, respectively.

\subsection{The $K$-theoretic version}
\label{sub:k-th limit}

In the present Subsection, we will consider the $K$-theoretic versions of the functors \eqref{eqn:e functor}, \eqref{eqn:f functor}, \eqref{eqn:m functor}
\begin{align}
&\ke_i : \KK_{\quot_d} \rightarrow \KK_{\quot_{d+1} \times C}, \qquad \ke_i(\gamma) = (p_+ \times p_C)_*\Big([\CL^i] \otimes p_-^*(\gamma) \Big) \label{eqn:e operator} \\
&\kf_i : \KK_{\quot_{d+1}} \rightarrow \KK_{\quot_{d} \times C}, \qquad \kf_i(\gamma) = (p_- \times p_C)_*\Big([\CL^i] \otimes p_+^*(\gamma) \Big) \label{eqn:f operator} \\
&\km_i : \KK_{\quot_d} \rightarrow \KK_{\quot_{d} \times C}, \ \qquad \km_i(\gamma) = [\wedge^i\CE] \otimes \pi^*(\gamma) \label{eqn:m operator}
\end{align}
where $\pi : \quot_d \times C \rightarrow \quot_d$ denotes the standard projection. With this in mind, the $K$-theoretic version of Proposition \ref{prop:quadratic relation e} reveals the identities
\begin{equation}
\label{eqn:rel 1}
[\textcolor{red}{\ke_i}, \textcolor{blue}{\ke_j}] = \Delta_* \left( \sum_{k=i}^{j-1} \ke_k \circ \ke_{i+j-k} \Big|_\Delta \right)
\end{equation}
\begin{equation}
\label{eqn:rel 2}
[ \textcolor{blue}{\kf_j}, \textcolor{red}{\kf_i}] = \Delta_* \left( \sum_{k=i}^{j-1} \kf_{i+j-k} \circ \kf_k  \Big|_\Delta \right)
\end{equation}
for all integers $i \leq j$. Meanwhile, the $K$-theoretic version of Proposition \ref{prop:comm e and m} reads
\begin{equation}
\label{eqn:rel 3}
[\textcolor{red}{\ke_i}, \textcolor{blue}{\km_j}] = \Delta_* \left( \sum_{k=1}^{j} (-1)^{k+1} \ke_{i+k} \circ \km_{j-k} \Big|_\Delta \right)
\end{equation}
\begin{equation}
\label{eqn:rel 4}
[\textcolor{blue}{\km_j},\textcolor{red}{\kf_i}] = \Delta_* \left( \sum_{k=1}^{j} (-1)^{k+1} \km_{j-k} \circ \kf_{i+k} \Big|_\Delta \right)
\end{equation}
for all $i \in \BZ$ and $j \in \{0,\dots,r\}$, as well as 
\begin{equation}
\label{eqn:rel 5}
\textcolor{blue}{\km_r} \circ \textcolor{red}{\ke_i} = \textcolor{red}{\ke_i} \circ \textcolor{blue}{\km_r} (-\Delta)
\end{equation}
\begin{equation}
\label{eqn:rel 6}
\textcolor{red}{\kf_i} \circ \textcolor{blue}{\km_r} = \textcolor{blue}{\km_r} \circ \textcolor{red}{\kf_i} (-\Delta)
\end{equation}
for all $i \in \BZ$. Finally, the $K$-theoretic version of Proposition \ref{prop:e and f} reads
\begin{equation}
\label{eqn:rel 7 bis}
[\textcolor{red}{\ke_i}, \textcolor{blue}{\kf_j}] = \Delta_* ( \text{multiplication by }[\alpha_{i+j}])
\end{equation}
where 
\begin{equation}
\label{eqn:rel 7 bis bis}
[\alpha_{\ell}] = \begin{cases} \sum_{k=0}^{\ell} [S^k\CE] \left(\sum_{k'=0}^{\ell-k} (-1)^{k'} [S^{\ell-k-k'} {\CE}][\wedge^{k'} V] \right)  \left[\frac {\det \CE}{\det V} \right] [\CK_C]^{-\ell+k}  &\text{if } \ell \geq 0 \\
(-1)^{r-1} \sum_{k=r}^{-\ell} \frac {[S^{k-r} {\CE}^\vee]}{[\det \CE]} \left( \sum_{k'=0}^{-\ell-k}  (-1)^{k'}  [S^{-k-k'-\ell} {\CE}^\vee][\wedge^{k'} V^\vee]  \right) [\CK_C]^{-\ell-k}  &\text{if } \ell < 0
\end{cases}
\end{equation}

\begin{proposition}
\label{prop:compactly}

Formula \eqref{eqn:rel 7 bis} can be written compactly as
\begin{equation}
\label{eqn:rel 7}
[\textcolor{red}{\ke_i}, \textcolor{blue}{\kf_j}] = \Delta_* \left(\text{multiplication by } \int_{\infty - 0} z^{i+j} \cdot \frac {\wedge^\bullet\left(\frac {zq}{V}\right)}{\wedge^\bullet\left(\frac {\CE}z\right)\wedge^\bullet\left(\frac {zq}{\CE}\right)} \right)
\end{equation}
for all $i,j \in \BZ$, where $\int_{\infty-0} R(z)$ denotes the constant term in the expansion of a rational function $R(z)$ near $z=\infty$ minus the constant term in the expansion of $R(z)$ near $z=0$. 

\end{proposition}

\begin{proof}

We need to prove that 
\begin{equation}
\label{eqn:question}
\int_{\infty - 0} z^{\ell} \cdot \frac {\wedge^\bullet\left(\frac {zq}{V}\right)}{\wedge^\bullet\left(\frac {\CE}z\right)\wedge^\bullet\left(\frac {zq}{\CE}\right)} 
\end{equation}
equals the object $[\alpha_{\ell}]$ of \eqref{eqn:rel 7 bis bis}. We will do so when $\ell \geq 0$, and leave the case $\ell < 0$ as an exercise to the reader. When $\ell \geq 0$, the expansion \eqref{eqn:question} has no constant term near $z = 0$, because it has order $z^{\ell+r}$ and higher. Meanwhile, near $z = \infty$, we have
$$
\frac {\wedge^\bullet\left(\frac {zq}{V}\right)}{\wedge^\bullet\left(\frac {\CE}z\right)\wedge^\bullet\left(\frac {zq}{\CE}\right)} = \frac {\wedge^\bullet\left(\frac {V}{zq}\right)}{\wedge^\bullet\left(\frac {\CE}z\right)\wedge^\bullet\left(\frac {\CE}{zq}\right)} \cdot \left[\frac {\det \CE}{\det V} \right] = 
$$
$$
= \left(\sum_{k=0}^{\infty} \frac {[S^k\CE]}{z^k} \right)\left(\sum_{k'=0}^{\infty} (-1)^{k'} \frac {[\wedge^{k'}V]}{z^{k'}q^{k'}} \right)  \left(\sum_{k''=0}^{\infty} \frac {[S^{k''}\CE]}{z^{k''}q^{k''}} \right)\left[\frac {\det \CE}{\det V} \right]
$$
Extracting the coefficient of $z^{-\ell}$ yields precisely the right-hand side of \eqref{eqn:rel 7 bis bis}.

\end{proof}

\subsection{Shifted quantum loop $\fsl_2$}
\label{sub:action 1}

We will now define an abstract algebraic version of relations \eqref{eqn:rel 1}--\eqref{eqn:rel 6} and \eqref{eqn:rel 7}, which we will later show to be a version of the shifted quantum loop algebra of type $\fsl_2$ (\cite{ft}).

\begin{definition}
\label{def:quantum loop}

Fix $r \in \BN$. Shifted quantum loop $\fsl_2$ is the $\BZ[q,q^{- 1}]$-algebra
\begin{equation}
\label{eqn:def loop}
\UU = \Big\langle \qe_i,\qf_i,\qm_j,p_j \Big \rangle_{i, j \in \BZ, j \in \{0,\dots,r\}} \Big/ \Big(\text{relations \eqref{eqn:rel 0 loop}--\eqref{eqn:rel 7 loop}}\Big) 
\end{equation}
where $p_0,\dots,p_r$ are central elements with $p_0,p_r$ invertible, and
\begin{equation}
\label{eqn:rel 0 loop}
\qm_0 = 1, \qm_r \text{ is invertible, and } [\qm_i,\qm_j] = 0
\end{equation}
for all $i,j\in \{0,\dots,r\}$,
\begin{equation}
\label{eqn:rel 1 loop}
[\qe_i,\qe_j] = (1-q) \sum_{k=i}^{j-1} \qe_k \qe_{i+j-k} 
\end{equation}
\begin{equation}
\label{eqn:rel 2 loop}
[\qf_j,\qf_i] = (1-q) \sum_{k=i}^{j-1} \qf_{i+j-k} \qf_k 
\end{equation}
for all integers $i \leq j$, 
\begin{equation}
\label{eqn:rel 3 loop}
[\qe_i,\qm_j] = (1-q) \sum_{k=1}^{j} (-1)^{k+1} \qe_{i+k} \qm_{j-k}
\end{equation}
\begin{equation}
\label{eqn:rel 4 loop}
[\qm_j,\qf_i] = (1-q) \sum_{k=1}^{j} (-1)^{k+1} \qm_{j-k} \qf_{i+k}
\end{equation}
for all $i \in \BZ$ and $j \in \{0,\dots,r\}$,
\begin{equation}
\label{eqn:rel 5 loop}
\qm_r \qe_i = q \qe_i \qm_r
\end{equation}
\begin{equation}
\label{eqn:rel 6 loop}
\qf_i \qm_r = q \qm_r \qf_i
\end{equation}
for all $i \in \BZ$, as well as
\begin{equation}
\label{eqn:rel 7 loop}
[\qe_i,\qf_j] = (1-q) \begin{cases} \qh_{i+j}^+ &\text{if } i+j \geq 0 \\ 0 &\text{if } - r < i+j < 0 \\ - \qh_{i+j}^- &\text{if } i+j \leq -r \end{cases}
\end{equation}
Above, we define the following power series (let $\qm(z) = \sum_{k=0}^r (-1)^k \qm_kz^{-k}$)
\begin{equation}
\label{eqn:h plus series}
\qh^+(z) = \sum_{k=0}^{\infty} \frac {\qh^+_k}{z^k} = \frac {\qm_r P(z)}{\qm(z)\qm(zq)} 
\end{equation}
\begin{equation}
\label{eqn:h minus series}
\qh^-(z) = \sum_{k=r}^{\infty} \qh^-_{-k} z^k = \frac {\qm_r P(z)}{\qm(z)\qm(zq)} \end{equation}
for $P(z) = p_0  + \dots + p_r z^{-r}$. While the right-hand sides of \eqref{eqn:h plus series} and \eqref{eqn:h minus series} are identical rational functions, we expand them in opposite powers of $z$ in the two formulas.

\end{definition}

\begin{remark}

The word ``shifted" in the name of \eqref{eqn:def loop} refers to the fact that the series $\qh^+(z)$ and $\qh^-(z)$ start at different powers of $z$ (the former at $z^0$ and the latter at $z^r$). This is the fundamental property of shifted quantum loop algebras, which were first defined in \cite{ft} as certain $\BQ(q)$ algebras which prominently featured the relation
\begin{equation}
\label{eqn:quadratic}
\qe(z)\qe(w) (z-wq) = \qe(w) \qe(z) (zq-w)
\end{equation}
where $\qe(z) = \sum_{i \in \BZ} \frac {\qe_i}{z^i}$. In terms of coefficients, \eqref{eqn:quadratic} means that for all $i,j \in \BZ$
\begin{equation}
\label{eqn:quadratic explicit}
\qe_{i+1} \qe_j - q \qe_i \qe_{j+1} = q \qe_j \qe_{i+1} - \qe_{j+1} \qe_i
\end{equation}
It is easy to see that \eqref{eqn:quadratic explicit} follows from \eqref{eqn:rel 1 loop}, and the converse is also true if we invert the integer 2. Thus, Definition \ref{def:quantum loop} is a certain integral version of the usual $\BQ(q)$-algebra that is usually referred to as shifted quantum loop $\fsl_2$. Our particular integral version is naturally well-suited for the categorification carried out in the present paper.

\end{remark}

\subsection{From representation theory to geometry} 
\label{sub:action 2}

We will now bridge the geometric formulas in \eqref{eqn:rel 1}--\eqref{eqn:rel 7} with the representation theoretic formulas in \eqref{eqn:rel 1 loop}--\eqref{eqn:rel 7 loop}. In what follows, we will write for all $n \in \BN$
$$
\KK_{\quot \times C^n} = \bigoplus_{d=0}^{\infty} \KK_{\quot_d \times C^n}
$$
The following notion was introduced in \cite[Definition 4.14]{n}.

\begin{definition}
\label{def:action}

An action $\UU \curvearrowright \KK_{\quot}$ is a $\BQ$-linear map
\begin{equation}
\label{eqn:linear map}
\UU \xrightarrow{\Phi} \emph{Hom} (\KK_{\quot}, \KK_{\quot \times C})
\end{equation}
satisfying the following properties for all $x,y \in \UU$:

\begin{enumerate}
	
\item $\Phi(1) = \pi^*$, where $\pi : \quot \times C \rightarrow \quot$ is the natural projection map;

\item $\Phi(q x)$ coincides with the composition:
\begin{equation}
\label{eqn:const}
\KK_{\quot}  \xrightarrow{\Phi(x)} \KK_{\quot \times C} \xrightarrow{\emph{Id}_{\quot} \times \emph{multiplication by } [\CK_C]}  \KK_{\quot \times C} 
\end{equation}
(where $\CK_C$ denotes the canonical line bundle on $C$, pulled back to $\quot \times C$);

\item $\Phi(\textcolor{red}{x} \textcolor{blue}{y})$ coincides with the composition:
\begin{equation}
\label{eqn:hom}
\KK_{\quot} \xrightarrow{\Phi(\textcolor{blue}{y})} \KK_{\quot \times \textcolor{blue}{C}} \xrightarrow{\Phi(\textcolor{red}{x}) \times \emph{Id}_{\textcolor{blue}{C}}} \KK_{\quot \times \textcolor{red}{C} \times \textcolor{blue}{C}} \xrightarrow{\emph{Id}_{\quot} \times \Delta^*} \KK_{\quot \times C}
\end{equation}

\item $\Delta_* \Phi \left( \frac {[\textcolor{red}{x},\textcolor{blue}{y}]}{1-q} \right)$ coincides with the difference between the compositions:
\begin{align*}
&\KK_{\quot}\xrightarrow{\Phi(\textcolor{blue}{y})} \KK_{\quot \times \textcolor{blue}{C}} \xrightarrow{\Phi(\textcolor{red}{x}) \times \emph{Id}_{\textcolor{blue}{C}}} \KK_{\quot \times \textcolor{red}{C} \times \textcolor{blue}{C}} \\
&\KK_{\quot} \xrightarrow{\Phi(\textcolor{red}{x})} \KK_{\quot \times \textcolor{red}{C}} \xrightarrow{\Phi(\textcolor{blue}{y}) \times \emph{Id}_{\textcolor{red}{C}}} \KK_{\quot \times \textcolor{red}{C} \times \textcolor{blue}{C}} 
\end{align*}
We are implicitly using the fact that $[x,y]$ is a multiple of $1-q$ for any $x,y \in \UU$, which is easily seen to be the case from relations \eqref{eqn:rel 0 loop}--\eqref{eqn:rel 7 loop}.

\end{enumerate}

\end{definition}

The motivation for our notion of action is to systematize the idea of ``operators indexed by $\KK_C$" without invoking the K\"unneth decomposition (which does not hold in algebraic $K$-theory for $C \not \cong \BP^1$). Indeed, from \eqref{eqn:linear map} one can obtain operators
\begin{equation}
\label{eqn:linear operators}
\Phi(x)_{\gamma} \in \text{End}(\KK_{\quot}), \qquad \Phi(x)_{\gamma} = \text{proj}_{\quot *} \left( \text{proj}_C^*(\gamma) \otimes \Phi(x) \right)
\end{equation}
for any $x \in \UU$ and $\gamma \in \KK_C$. It is straightforward to deduce from axioms (1)-(4) above that the operators \eqref{eqn:linear operators} obey the relations in a suitably-defined version of the algebra $\UU$ defined over the parameter space $\KK_C$ (akin to the classic construction of Heisenberg algebra actions on the homology of Hilbert schemes from \cite{nakajima}). We prefer the formalism of Definition \ref{def:action} to that of the operators \eqref{eqn:linear operators}, because the former provides a richer construction in situations where the K\"unneth decomposition fails, and because it makes computations more direct.

Comparing formulas \eqref{eqn:rel 1}--\eqref{eqn:rel 7} with \eqref{eqn:rel 1 loop}--\eqref{eqn:rel 7 loop} shows that the assignments
$$
\qe_i \text{ of \eqref{eqn:def loop}} \leadsto \ke_i \text{ of \eqref{eqn:e operator}}
$$
$$
\qf_i \text{ of \eqref{eqn:def loop}} \leadsto \kf_i \text{ of \eqref{eqn:f operator}}
$$
$$
\qm_i \text{ of \eqref{eqn:def loop}} \leadsto \km_i \text{ of \eqref{eqn:m operator}}
$$
and
$$
P(z) \text{ of \eqref{eqn:h plus series}--\eqref{eqn:h minus series}} \leadsto \text{tensoring with } \frac 1{\det V} \wedge^\bullet\left(\frac V{zq} \right)
$$
give rise to an action
\begin{equation}
\label{eqn:final action}
\UU \curvearrowright \KK_{\quot}.
\end{equation}

\subsection{Divided powers}

A well-known feature of quantum (loop) groups is that $\qe_i,\qf_i$ are the correct generators of the shifted quantum loop group only over $\BQ(q)$. Over $\BZ[q,q^{-1}]$, one should instead consider the algebra generated by the divided powers
\begin{equation}
\label{eqn:divided powers}
\qe_i^{(n)} = \frac {\qe_i^n}{[n]_q!} \quad \text{and} \quad \qf_i^{(n)} = \frac {\qf_i^n}{[n]_q!}
\end{equation}
where $[n]_q! = (1+q)(1+q+q^2)\dots(1+q+\dots+q^{n-1})$. However, we do not know an explicit set of relations generalizing \eqref{eqn:rel 1 loop} and \eqref{eqn:rel 2 loop}, which completely govern the interaction between the divided powers \eqref{eqn:divided powers}. At the categorical level, we expect that the ``divided power" analogues of \eqref{eqn:e intro} and \eqref{eqn:f intro} should be
\begin{align}
&\DD_{\quot_d} \xrightarrow{\te_i^{(n)}} \DD_{\quot_{d+n} \times C}, \quad \te_i^{(n)}(\gamma) = R(p^{(n)}_+ \times p^{(n)}_C)_* \Big(\CL^{i} \otimes L{p_-^{(n)}}^* (\gamma) \Big) \label{eqn:e intro power} \\
&\DD_{\quot_{d+n}} \xrightarrow{\tf_i^{(n)}} \DD_{\quot_{d} \times C}, \quad \tf^{(n)}_i(\gamma) = R(p_-^{(n)} \times p^{(n)}_C)_* \Big(\CL^{i} \otimes L{p_+^{(n)}}^* (\gamma) \Big) \label{eqn:f intro power}
\end{align}
where we consider the following generalization of diagram \eqref{eqn:diagram intro}
\begin{equation}
\label{eqn:diagram intro power}
\xymatrix{& \quot^\bullet_{d,d+n} \ar[ld]_{p^{(n)}_-} \ar[d]^{p^{(n)}_C} \ar[rd]^{p^{(n)}_+} & \\ \quot_{d} & C & \quot_{d+n}}
\end{equation}
where $\quot_{d,d+n}^\bullet$ is the scheme parameterizing $\{E' \subset E \subset V\}$ with the first injection being colength $n$ and supported at a single (but arbitrary) point $x \in C$, and $\CL = \CL_1$ is the line bundle defined by analogy with Subsection \ref{sub:basic nested}. It would be interesting to find a complete set of categorical relations between the functors \eqref{eqn:e intro power} and \eqref{eqn:f intro power}. 

\vskip.2in

\section{The derived category of Quot schemes}
\label{sec:steppingstone}

\vskip.2in

\subsection{Indexing}
\label{sub:indexing}

We will consider non-decreasing sequences
\begin{equation}
\label{eqn:uk}
(k_1 \leq \dots \leq k_d) \in \{0,1,\dots,r-1\}^d
\end{equation}
which are in one-to-one correspondence with compositions of $d$ in $r$ parts
\begin{equation}
\label{eqn:bd}
\bd = (d_0, d_1, \dots, d_{r-1})
\end{equation}
simply by letting $d_i$ be the number of times the number $i$ appears in a sequence \eqref{eqn:uk}. Both sequences and compositions can be ordered lexicographically, i.e.
$$
(k_1,\dots,k_d) > (k_1',\dots,k'_d)
$$
if and only if $k_1>k'_1$ or ($k_1=k_1'$ and $k_2>k_2'$) etc. It is easy to see that two non-decreasing sequences are in a certain lexicographic order with respect to each other if and only if the associated compositions satisfy the opposite lexicographic order.

\subsection{Symmetric powers of curves} 
\label{sub:seq and comp}

For any $d \geq 0$, let us write 
\begin{equation}
\label{eqn:k! to 1}
\rho_{(d)} : C^d \rightarrow C^{(d)}
\end{equation}
for the natural map from the $d$-th power of $C$ to the $d$-th symmetric power of $C$, which is defined as $C^{(d)} = C^d / S_d$ with respect to the standard permutation action. More generally, for any composition $\bd$ as in \eqref{eqn:bd}, we will write
$$
C^{(\bd)} = C^{(d_0)} \times C^{(d_1)} \times \dots \times C^{(d_{r-1})} = C^d / S_{\bd}
$$
where we consider the standard embedding
\begin{equation}
\label{eqn:group}
S_{\bd} = S_{d_0} \times S_{d_1} \times \dots \times S_{d_{r-1}} \hookrightarrow S_d
\end{equation}
The obvious product of the maps \eqref{eqn:k! to 1} will be denoted by $\rho_{\bd} : C^d \rightarrow C^{(\bd)}$, and we note that it is a finite morphism between smooth varieties. Therefore, its direct image and inverse image functors are exact, but we will still denote them by
\begin{equation}
\label{eqn:push pull powers}
\DD_{C^d} \xleftrightharpoons[R\rho_{\bd*}]{L\rho_{\bd}^*} \DD_{C^{(\bd)}}
\end{equation}
in order to emphasize the fact that they are functors between derived categories.

\subsection{Transposed functors}

In \eqref{eqn:e functor}, we defined $\te_i$ as the functor $\DD_{\quot_{d-1}} \rightarrow \DD_{\quot_d \times C}$ associated with the correspondence
$$
R(p_- \times p_+ \times p_C)_*(\CL^i) \in \DD_{\quot_{d-1} \times \quot_d \times C}
$$
We may switch perspective and think of this correspondence as inducing a functor
\begin{equation}
\label{eqn:switch perspective}
\be_i : \DD_{\quot_{d-1} \times C} \rightarrow \DD_{\quot_d}
\end{equation}
More generally, we may iterate $d$ of the functors $\be_i$, and consider
$$
\be_{k_1} \circ \dots \circ \be_{k_d} : \DD_{\quot_{d'-d} \times C^d} \rightarrow \DD_{\quot_{d'}}
$$
for any $d' \geq d$, with the functor $\be_{k_a}$ acting in the $a$-th factor of $C^d$.

\begin{definition}
\label{def:basis}

For any composition $\bd = (d_0,d_1, \dots, d_{r-1})$ of $d$, let 
\begin{equation}
\label{eqn:basis}
\be_{\bd} = \underbrace{\be_{0} \circ \dots \circ \be_{0}}_{d_{0} \text{ times}} \circ  \underbrace{\be_{1} \circ \dots \circ \be_{1}}_{d_1 \text{ times}} \circ \dots \circ \underbrace{\be_{{r-1}} \circ \dots \circ \be_{{r-1}}}_{d_{r-1} \text{ times}} : \DD_{C^d} \rightarrow \DD_{\quot_{d}}.
\end{equation}

\end{definition}

\subsection{A stepping stone}

The following proposition is a stepping stone in the direction of Theorem \ref{thm:intro semi}. However, the reason why this result does not yield a semi-orthogonal decomposition of $\DD_{\quot_d}$ is that the essential images of the functors \eqref{eqn:basis} do not generate the entire $\DD_{\quot_d}$. This will be explored and remedied in Section \ref{sec:semi}.

\begin{proposition}
\label{prop:semi}

For any compositions $\bd, \bd'$ of $d$, we have identifications
\begin{equation}
\label{eqn:semi}
\emph{Hom}_{\quot_d} \Big ( \be_{\bd}  (\gamma), \be_{\bd'}  (\gamma') \Big) \cong \begin{cases} 0&\text{if } \bd < \bd' \\ \emph{Hom}_{C^{(\bd)}} \Big(R\rho_{\bd*}(\gamma), R\rho_{\bd*}(\gamma') \Big) &\text{if }\bd = \bd' \end{cases}
\end{equation}
natural in $\gamma,\gamma' \in \emph{Ob}(\DD_{C^d})$. 

\end{proposition}

\begin{proof} The main technical tool is the following adjunction between functors.

\begin{lemma}
\label{lem:adjoint}

For any $i \in \BZ$ and $d > 0$, there are identifications
\begin{equation}
\label{eqn:lem adjoint}
\emph{Hom}_{\quot_d} \Big(\alpha, \be_i(\beta)\Big) \cong \emph{Hom}_{\quot_{d-1} \times C} \Big(\tf'_{-i-r}(\alpha), \beta \Big) 
\end{equation}
and
\begin{equation}
\label{eqn:lem adjoint opp}
\emph{Hom}_{\quot_d} \Big(\be_i(\beta), \alpha \Big) \ \cong \ \emph{Hom}_{\quot_{d-1} \times C} \Big(\beta, \tf''_{-i}(\alpha) \Big)
\end{equation}
natural in $\alpha \in \emph{Ob}(\DD_{\quot_d})$ and $\beta \in \emph{Ob}(\DD_{\quot_{d-1} \times C})$, where we write
\begin{equation}
\label{eqn:switch perspective prime}
\tf'_{-i-r} = (\det \CE \otimes \tf_{-i-r})[r-1] : \DD_{\quot_d} \rightarrow \DD_{\quot_{d-1} \times C}
\end{equation}
and
\begin{equation}
\label{eqn:switch perspective double prime}
\tf''_{-i} = \frac {\det V}{\det \CE} \otimes \tf_{-i} : \DD_{\quot_d} \rightarrow \DD_{\quot_{d-1} \times C}
\end{equation}
for all $i \in \BZ$.

\end{lemma}

\noindent We will first show how to use Lemma \ref{lem:adjoint} to conclude the proof of Proposition \ref{prop:semi}, and then prove the Lemma at the very end. Iterating \eqref{eqn:lem adjoint} a number of $d$ times gives us
$$
\text{Hom}_{\quot_d} \Big(\alpha, \be_{k'_1} \circ \dots \circ \be_{k'_d}  (\gamma') \Big) \cong \text{Hom}_{C^d} \Big(\tf'_{-k'_d-r} \circ \dots \circ \tf'_{-k'_1-r}(\alpha), \gamma'\Big),
$$
with $\alpha \in \text{Ob}(\DD_{\quot_d})$ and $\gamma' \in \text{Ob}(\DD_{C^d})$. If we let $\alpha = \be_{k_1} \circ \dots \circ \be_{k_d}  (\gamma)$, then
\begin{multline}
\label{eqn:long adjoint}
\text{Hom}_{\quot_d} \Big(\be_{k_1} \circ \dots \circ \be_{k_d}  (\gamma), \be_{k'_1} \circ \dots \circ \be_{k'_d}  (\gamma') \Big) \cong \\ \cong \text{Hom}_{C^d} \Big(\tf'_{-k'_d-r} \circ \dots \circ \tf'_{-k'_1-r} \circ \be_{k_1} \circ \dots \circ \be_{k_d}  (\gamma), \gamma' \Big),
\end{multline}
where naturally $\gamma, \gamma' \in \text{Ob}(\DD_{C^d})$. Therefore, the content of Proposition \ref{prop:semi} comes down to the following isomorphism of functors $\DD_{C^d} \rightarrow \DD_{C^d}$, where we let $(k_1 \leq \dots \leq k_d)$ and $(k_1' \leq \dots \leq k_d')$ be the non-decreasing sequences associated to the compositions $\bd$ and $\bd'$, respectively, as explained in Subsection \ref{sub:indexing}:
\begin{multline}
\label{eqn:long composition}
\tf'_{-k'_d-r} \circ \dots \circ \tf'_{-k'_1-r} \circ \be_{k_1} \circ \dots \circ \be_{k_d}  \cong \\ \cong \begin{cases} 0 &\text{if } (k_1,\dots,k_d) > (k_1',\dots,k_d') \\ L\rho_{\bd}^* \circ R\rho_{\bd*} &\text{if }(k_1,\dots,k_d) = (k_1',\dots,k_d') \end{cases}
\end{multline}
To establish formula \eqref{eqn:long composition}, recall that Proposition \ref{prop:e and f} governs the commutation relation between the functors $\te_i$ and $\tf_j$. Since the commutation relations between the functors $\be_i$ and $\tf_j$ is an equivalent computation (simply a matter of moving a factor of $C$ from the codomain to the domain), we obtain for all $0 \leq j < i < r$
\begin{equation}
\label{eqn:commute 1}
\textcolor{blue}{\tf_{-j-r}} \circ \textcolor{red}{\be_i} \cong \textcolor{red}{\be_i} \circ \textcolor{blue}{\tf_{-j-r}}
\end{equation}
(because $\alpha_{\ell} = 0$ for $\ell \in \{-r+1,\dots,-1\}$), while for all $i \in \{0,\dots,r-1\}$ we have
\begin{equation}
\label{eqn:commute 2}
\text{Cone} \Big[\textcolor{red}{\be_i} \circ \textcolor{blue}{\tf_{-i-r}} \longrightarrow \textcolor{blue}{\tf_{-i-r}} \circ \textcolor{red}{\be_i} \Big] \cong \Big\{ \text{functor of tensoring with }(\det \CE)^{-1}[-r+1] \Big\}.
\end{equation}

\begin{lemma}

For all $0 \leq j < i < r$, we have
\begin{equation}
\label{eqn:commute 3}
\textcolor{blue}{\tf'_{-j-r}} \circ \textcolor{red}{\be_i} \cong \textcolor{red}{\be_i} \circ \textcolor{blue}{\tf'_{-j-r}} (-\Delta)
\end{equation}
while for all $i \in \{0,\dots,r-1\}$, we have
\begin{equation}
\label{eqn:commute 4}
\emph{Cone} \Big[\textcolor{red}{\be_i} \circ \textcolor{blue}{\tf'_{-i-r}} (-\Delta) \longrightarrow \textcolor{blue}{\tf'_{-i-r}} \circ \textcolor{red}{\be_i} \Big] \cong \emph{Id}_{\quot_d \times C}.
\end{equation}
\end{lemma}

\begin{proof} Because $\tf'_i = (\det \CE \otimes \tf_i)[r-1]$, formulas \eqref{eqn:commute 3} and \eqref{eqn:commute 4} are immediate consequences of formulas \eqref{eqn:commute 1} and \eqref{eqn:commute 2}, respectively, together with the fact that
\begin{equation}
\label{eqn:temp temp}
\textcolor{blue}{\det \CE} \otimes \textcolor{red}{\be_i}(-) = \textcolor{red}{\be_i}(\textcolor{blue}{\det \CE} \otimes -) (-\Delta).
\end{equation}
Formula \eqref{eqn:temp temp} is simply a reformulation of  \eqref{eqn:iso}, once we move the factor of $C$ from the codomain of $\te_i$ to the domain of $\be_i$. 
    
\end{proof}

With formulas \eqref{eqn:commute 3} and \eqref{eqn:commute 4} in mind, let us show that the LHS of \eqref{eqn:long composition} vanishes if $k_1' < k_1$. Indeed, this inequality implies that we could use \eqref{eqn:commute 3} to move $\tf'_{-k_1'-r}$ unhindered to the right of all of the $\be$'s, where it would annihilate $\DD_{C^d}$. For the same reason, if $k_1' = k_1$, then the only way that moving $\tf'_{-k_1'-r}$ to the right could produce a non-zero contribution is if $\tf'_{-k_1'-r}$ precisely annihilates some $\be_{k_i}$ with $k_i = k_1$ and leaves an identity functor behind, according to \eqref{eqn:commute 4}. Repeating this argument for $\tf'_{-k_2'-r}, \dots, \tf'_{-k_d' - r}$ shows that the LHS of \eqref{eqn:long composition} vanishes if $(k_1',\dots,k_d') < (k_1,\dots,k_d)$. 

\medskip

Meanwhile, when $(k_1',\dots,k_d') = (k_1,\dots,k_d)$, the argument above shows that the only non-trivial contributions to \eqref{eqn:long composition} happen when each $\tf'_{-i-r}$ annihilates some $\be_i$ and leaves an identity functor behind, according to \eqref{eqn:commute 4}. Thus, it suffices to consider only the composition $(d) = (0,\dots,0,d,0,\dots,0)$ and prove that
\begin{equation}
\label{eqn:brace}
\underbrace{\tf'_{-i-r} \circ \dots \circ \tf'_{-i-r}}_{d \text{ times }} \circ \underbrace{\be_{i} \circ \dots \circ \be_{i}}_{d \text{ times }} \cong L\rho_{(d)}^* \circ R\rho_{(d)*}
\end{equation}
for any $i \in \{0,\dots,r-1\}$. We will write $E = \be_i$ and $F = \tf'_{-i-r}$, and interpret these functors as correspondences. With this in mind, \eqref{eqn:brace} becomes equivalent to
\begin{equation}
\label{eqn:seek}
F^d E^d \cong \CO_{C^d \times_{C^{(d)}} C^d} \in \DD_{C^d \times C^d}.
\end{equation}
We will prove \eqref{eqn:seek} by induction on $d$. To this end, consider the following complex in $\DD_{C^d \times C^d}$ (as a general rule, we will write $\Delta_{\sharp \flat}$ for the codimension one diagonal in $C^d \times C^d$ that identifies the factors indexed by $\sharp$ and $\flat$)
$$
P_d = F^{d-1} E^d F (-\Delta_{\bullet 1} - \dots - \Delta_{\bullet d}) \longrightarrow P_{d-1} = F^{d-1} E^{d-1} F E  (-\Delta_{\bullet 1} - \dots - \Delta_{\bullet, d-1})  
$$
$$
\longrightarrow \dots \longrightarrow P_1 = F^{d-1} E F E^{d-1} (-\Delta_{\bullet 1})  \longrightarrow P_0 = F^{d} E^d 
$$
where $1,\dots,d$ denote the factors of $C$ corresponding to the $d$ copies of $E$, and $\bullet$ denotes the factor of $C$ corresponding to the right-most copy of $F$. The arrows in the complex above are induced by \eqref{eqn:commute 4}. Thus, the induction hypothesis of \eqref{eqn:seek} (for $d-1$ instead of $d$) implies
\begin{equation}
\label{eqn:cone}
\text{Cone} \Big[ P_k \rightarrow P_{k-1} \Big] = \CO_{\Delta_{\bullet k} \cap (C^d \times_{C^{(d)}} C^d)} (-\Delta_{\bullet 1} - \dots - \Delta_{\bullet,k-1})
\end{equation}
Thus, formula \eqref{eqn:seek} follows from the $k=0$ particular case of the following claim
\begin{equation}
\label{eqn:claim}
P_k \cong \CO_{(\Delta_{\bullet, k+1} \cup \dots \cup \Delta_{\bullet d}) \cap (C^d \times_{C^{(d)}} C^d)} (-\Delta_{\bullet 1} - \dots - \Delta_{\bullet,k})
\end{equation}
for all $k \in \{0,\dots,d\}$. We will prove \eqref{eqn:claim} by descending induction on $k$ (with the base case $k=d$ being trivial, as $P_d = 0$). First of all, note that formula \eqref{eqn:cone} holds if $P_k$ are replaced with the coherent sheaves in the right-hand side of \eqref{eqn:claim}, because
\begin{equation}
\label{eqn:ses z}
0 \rightarrow \CO_{Z_1}(-\Delta_{\bullet k}) \rightarrow \CO_Z \rightarrow \CO_{Z_2} \rightarrow 0 
\end{equation}
where $Z = Z_1 \sqcup Z_2$ denotes
\begin{align*}
&Z=(\Delta_{\bullet k} \cup \dots \cup \Delta_{\bullet d}) \cap (C^d \times_{C^{(d)}} C^d)\\
&Z_1=(\Delta_{\bullet, k+1} \cup \dots \cup \Delta_{\bullet d}) \cap (C^d \times_{C^{(d)}} C^d) \\
&Z_2=\Delta_{\bullet k}  \cap (C^d \times_{C^{(d)}} C^d)
\end{align*}
The intersection $Z_1 \cap Z_2$ corresponds to the divisor $\Delta_{\bullet k}$ in $Z_1$. Therefore, to prove that \eqref{eqn:cone} implies the induction step of \eqref{eqn:claim}, it suffices to show that
\begin{equation}
\label{eqn:one dim ext}
\text{Ext}^1_{C^d \times C^d} \Big(\CO_{Z_2}(-\Delta_{\bullet 1} - \dots \Delta_{\bullet, k-1}) , \CO_{Z_1} (-\Delta_{\bullet 1} - \dots - \Delta_{\bullet k}) \Big) \cong \BC
\end{equation}
with a generator of the above one-dimensional vector space given by the non-split short exact sequence \eqref{eqn:ses z}. To compute the $\text{Ext}^1$ space above, let us rearrange it as
\begin{equation}
\label{eqn:hom space}
\text{Hom}_{C^d \times C^d} \left(\CO_{Z_2} (\Delta_{\bullet k}) [-1], \CO_{Z_1}\right) \cong \text{Hom}_{Z_1}\left(\CO_{Z_2} (\Delta_{\bullet k}) [-1] \Big|_{Z_1} , \CO\right)
\end{equation}
with the isomorphism given by adjunction and the restriction being derived. To compute this derived restriction, we must resolve the structure sheaf of $Z_2 = \Delta_{\bullet k} \cap (C^d \times_{C^{(d)}} C^d)$ by vector bundles. To this end, let us choose (\'etale locally) a coordinate $x$ on $C$, and denote the corresponding coordinates on the factors of $C^d \times C^d$ by $x_1,\dots,x_d,x_\bullet, x_1',\dots,x_{d-1}'$ (the $x_i'$ correspond to the first $d-1$ copies of the correspondence $F$ in \eqref{eqn:seek}). If we let
$$
p_i = (x_1')^i + \dots + (x_{d-1}')^i
$$
then $Z_2$ is cut out by the equations $x_\bullet - x_k$ and $p_i - x_1^i - \dots - x_{k-1}^i - x_{k+1}^i - \dots - x_d^i$ for all $i \in \{1,\dots,d-1\}$ (the latter condition ensures that the unordered sets $\{x_1,\dots,x_{k-1},x_{k+1},\dots,x_d\}$ and $\{x_1',\dots,x_{d-1}'\}$ coincide). Therefore, we have
$$
\CO_{Z_2} \cong \text{Koszul} \left(x_\bullet - x_k, p_i - x_1^i - \dots - x_{k-1}^i - x_{k+1}^i - \dots - x_d^i \right)_{1 \leq i \leq d-1}
$$
in $\DD_{C^d \times C^d}$, and so the Hom space in \eqref{eqn:hom space} is
$$
\text{Hom}_{Z_1}\left( \left[ \dots \rightarrow \text{vector bundle} \xrightarrow{\left(x_\bullet - x_k,  p_i - x_1^i - \dots - x_{k-1}^i - x_{k+1}^i - \dots - x_d^i \right)_{1 \leq i \leq d-1}} \underline{\CO} \right], \CO (-\Delta_{\bullet k}) \right)
$$
with the underlined $\CO$ in homological degree 1. By dualizing, the expression above is
$$
H^0\left(Z_1, \left[ \underline{\CO(-\Delta_{\bullet k})} \xrightarrow{\left(x_\bullet - x_k,  p_i - x_1^i - \dots - x_{k-1}^i - x_{k+1}^i - \dots - x_d^i \right)_{1 \leq i \leq d-1}} \text{vector bundle} \rightarrow \dots \right] \right)
$$
with the underlined term in homological degree $-1$. However, the fact that $Z_1 \subset C^d \times_{C^{(d)}} C^d$ implies that
$$
p_i + x_\bullet^i = x_1^i + \dots + x_d^i \quad \Rightarrow \quad p_i - x_1^i - \dots - x_{k-1}^i - x_{k+1}^i - \dots - x_d^i = x_k^i - x_\bullet^i
$$
holds on $Z_1$, for all $i$. Therefore, the $H^0$ space above can be written as
$$
H^0\left(Z_1, \left[ \CO(-\Delta_{\bullet k}) \xrightarrow{x_\bullet - x_k} \underline{\CO} \right] \otimes \left[\underline{\CO} \xrightarrow{\text{multiples of }x_\bullet - x_k} \text{vector bundle} \rightarrow \dots \right] \right) = 
$$
$$
= H^0\left(Z_1 \cap Z_2, \left[ \underline{\CO} \xrightarrow{0}  \text{ vector bundle} \rightarrow \dots \right] \right) 
$$
with the underlined terms in homological degree 0 (in the equality above, we used the fact that $Z_1 \cap Z_2$ is cut out by the divisor $\Delta_{\bullet k}$ in $Z_1$). All terms beyond the underlined one do not contribute anything to the 0-th cohomology group, so it remains to prove that $H^0(Z_1 \cap Z_2, \CO)$ is one-dimensional. In order to do this, we will show that $Z_1 \cap Z_2$ is a reduced and connected projective variety (if we assume that $C$ itself is connected, which we do). To this end, note that $Z_1 \cap Z_2$ is the union as $\ell \in \{k+1,\dots,d\}$ of
\begin{equation}
\label{eqn:varieties}
\Delta_{\bullet \ell} \cap \Delta_{\bullet k} \cap (C^d \times_{C^{(d)}} C^d) = \Big\{x_k = x_{\ell} = x_\bullet\} \cong C^{d-1} \times_{C^{(d-1)}} C^{d-1}
\end{equation}
which is well-known to be reduced. As for connectedness, we observe that any point in \eqref{eqn:varieties} lies in the component of the small diagonal $x_1 = \dots = x_d = x_\bullet = x_1' = \dots = x_{d-1}'$, which is isomorphic to $C$ itself. 

Let us now compare the short exact sequence that produces \eqref{eqn:cone}, namely
\begin{equation}
\label{eqn:ses p}
0 \rightarrow P_{k} \rightarrow P_{k-1} \rightarrow \CO_{Z_2} (-\Delta_{\bullet 1}  - \dots - \Delta_{\bullet,k-1}) \rightarrow 0 
\end{equation}
with \eqref{eqn:ses z}. Having shown that there is (up to isomorphism) a single non-trivial extension of $\CO_{Z_2}$ by $\CO_{Z_1}(-\Delta_{\bullet k})$ would imply the induction step of \eqref{eqn:claim}, namely
$$
P_{k} \cong \CO_{Z_1}(-\Delta_{\bullet 1} - \dots -\Delta_{\bullet k}) \quad \Rightarrow \quad P_{k-1} \cong \CO_Z(-\Delta_{\bullet 1} - \dots -\Delta_{\bullet, k-1})
$$
as soon as we can show that the short exact sequence \eqref{eqn:ses p} is not split. In other words, it remains to show that $P_{k-1}$ cannot be isomorphic to
\begin{equation}
\label{eqn:cannot be}
\CO_{Z_1}(-\Delta_{\bullet 1} - \dots -\Delta_{\bullet k}) \bigoplus \CO_{Z_2} (-\Delta_{\bullet 1}  - \dots - \Delta_{\bullet,k-1})
\end{equation}
However, note that $P_{k-1}$ is $S_{d-k+1}$-invariant with respect to the permutation of the factors of $C$ corresponding to the right-most $d-k+1$ copies of $E$ (namely the copies which parameterize the points $x_k,\dots,x_d$) because of Proposition \ref{prop:quadratic relation e} for $i=j$. However, the coherent sheaf \eqref{eqn:cannot be} is clearly not invariant, hence it cannot be isomorphic to $P_{k-1}.$

\begin{proof} \emph{of Lemma \ref{lem:adjoint}:} Let us prove \eqref{eqn:lem adjoint}. As $\be_i = Rp_{+*} \left( \CL^i \otimes L(p_- \times p_C)^* \right)$, we have
\begin{equation}
\label{eqn:lem adjoint 1}
\text{Hom}_{\quot_d} \Big(\alpha, \be_i(\beta) \Big) = \text{Hom}_{\quot_d} \Big(\alpha, Rp_{+*} \left( \CL^i \otimes L(p_- \times p_C)^*(\beta) \right) \Big) \cong 
\end{equation}
$$
\text{Hom}_{\quot_{d-1,d}} \Big(Lp_{+}^*(\alpha), \CL^i \otimes L(p_- \times p_C)^*(\beta) \Big) = \text{Hom}_{\quot_{d-1,d}} \Big(\CL^{-i} \otimes Lp_{+}^*(\alpha), L(p_- \times p_C)^*(\beta) \Big)
$$
with the $\cong$ being the usual adjunction. However, because the morphism $p_- \times p_C$ is the projectivization of the rank $r$ vector bundle $\CE$ (so in particular it is smooth) then
$$
(p_- \times p_C)^! = \left( \frac {\CK_{\quot_{d-1,d}}}{\CK_{\quot_{d-1} \times C}} \otimes L(p_- \times p_C)^* \right)[r-1] = \left( \det \CE \otimes \CL^{-r} \otimes L(p_- \times p_C)^* \right) [r-1]
$$
with the second equality being \eqref{eqn:canonical proj}. Plugging this into \eqref{eqn:lem adjoint 1} gives us
\begin{align*}
\text{Hom}_{\quot_d} \Big(\alpha, \be_i(\beta) \Big) &\cong \text{Hom}_{\quot_{d-1,d}} \Big(\det \CE \otimes \CL^{-i-r} \otimes Lp_{+}^*(\alpha)[r-1], (p_- \times p_C)^!(\beta) \Big) \\
&\cong \text{Hom}_{\quot_{d-1} \times C} \Big(R(p_- \times p_C)_* \left( \det \CE \otimes \CL^{-i-r} \otimes Lp_{+}^*(\alpha) \right), \beta \Big) \\
&= \text{Hom}_{\quot_{d-1} \times C} \Big (\det \CE \otimes \tf_{-i-r}(\alpha)[r-1], \beta \Big)
\end{align*}
with the identification between the first two rows being the adjunction between $(p_- \times p_C)^!$ and $R(p_- \times p_C)_*$. This establishes \eqref{eqn:lem adjoint}. Formula \eqref{eqn:lem adjoint opp} is analogous and thus left as an exercise to the reader (it will also not be used in the present paper). \end{proof} \end{proof}

\section{The semi-orthogonal decomposition}
\label{sec:semi}

\vskip.2in

\subsection{Direct summands}
\label{sub:failure}

Proposition \ref{prop:semi} is close in spirit to a semi-orthogonal decomposition of the category $\DD_{\quot_d}$ itself, but its failure rests in the fact that the essential images of the functors $\be_{\bd}$ do not generate $\DD_{\quot_d}$. In turn, this is fundamentally caused by the fact that the essential images of the functors $R\rho_{\bd*}$ are not the entire $\DD_{C^{(\bd)}}$. To remedy this issue, we start from the following result, which holds for any composition $\bd$ of $d$.

\begin{lemma}
\label{lem:direct summand}

Any object in $\DD_{C^{(\bd)}}$ is a direct summand of an object in $\emph{Im }R\rho_{\bd*}$.

\end{lemma}

\begin{proof} We will prove the stronger fact that there is an action of
\begin{equation}
\label{eqn:invariant functor}
S_{\bd} \text{ on the functor } R\rho_{\bd*} \circ L \rho_{\bd}^*, \text{ whose invariant functor is } \text{Id}_{C^{(\bd)}}.
\end{equation}
To do this, we invoke the projection formula $R\rho_{\bd*} \circ L \rho_{\bd}^* (\gamma) \cong R\rho_{\bd*}(\CO_{C^d}) \otimes \gamma$. It therefore suffices to show that the action $S_{\bd} \curvearrowright C^d$ from \eqref{eqn:group} induces an action of 
\begin{equation}
\label{eqn:invariant claim}
S_{\bd} \curvearrowright R\rho_{\bd*}(\CO_{C^d}), \text{ whose invariant object is } \CO_{C^{(\bd)}}.
\end{equation}
This statement is \'etale local on the curve $C$, so we may assume $C = \BA^1$. In this case, both $C^d$ and $C^{\bd}$ are isomorphic to $\BA^d$, and $\rho_{\bd}$ is given by the formula
$$
\rho_{\bd} (x_1,\dots,x_{d_{0}}, \dots, y_1,\dots,y_{d_{r-1}}) = 
$$
$$
= (e_1(x_1,\dots,x_{d_{0}}), \dots, e_{d_{0}}(x_1,\dots,x_{d_{0}}), \dots, e_1(y_1,\dots,y_{d_{r-1}}), \dots, e_{d_{r-1}}(y_1,\dots,y_{d_{r-1}}))
$$
where $e_k$ denotes the $k$-th elementary symmetric function. The action of $S_{\bd}$ on the domain of $\rho_{\bd}$ permutes the $x$,$y$,$\dots$ variables independently of each other. With this in mind, \eqref{eqn:invariant claim} is simply equivalent to the fundamental theorem of symmetric functions.

\end{proof}

\subsection{The nested Quot scheme associated with a composition}

For any composition $\bd = (d_0, d_1, \dots, d_{r-1})$ of $d$, we let $\quot_{\bd}$ denote the nested Quot scheme
\begin{equation}
\label{eqn:nested quot scheme}
\quot_{\bd}  = \quot_{d_{r-1},d_{r-1}+d_{r-2},\dots,d_1+\dots+d_{r-1},d_0+\dots+d_{r-1}} 
\end{equation}
where the RHS is defined in \eqref{eqn:def nested general}. The support map \eqref{eqn:support} takes the form
\begin{equation}
\label{eqn:support composition}
s_{\bd} : \quot_{\bd} \rightarrow C^{(\bd)},
\end{equation}
and we also set
\begin{equation}
\label{eqn:support composition0}
s_{(d)} : \quot_{d} \rightarrow C^{(d)},
\end{equation}
for the special case $\bd = (d, 0, \dots, 0)$.
We will further write $s_d : \quot_{1,\dots,d} \rightarrow C^d$ for the map which records the support points of a full flag, so that the following diagram commutes
\begin{equation}
\label{eqn:cartesian}
\xymatrix{& \quot_{1,\dots,d} \ar[ld]_{s_d} \ar[rd]^{p_{\bd}} \\ C^{d} \ar[dr]_{\rho_{\bd}} & & \quot_{\bd} \ar[ld]^{s_{\bd}} \\ & C^{(\bd)} &}
\end{equation}
where $p_{\bd}$ forgets some of the constituent sheaves in a full flag, in accordance with the indices that appear in \eqref{eqn:nested quot scheme}.

\begin{proposition}
\label{prop:cartesian}

For any composition $\bd$ of $d$, we have
\begin{equation}
\label{eqn:zeta push}
R\zeta_*(\CO_{\quot_{1,\dots,d}}) = \CO_{C^d \times_{C^{(\bd)}} \quot_{\bd}} 
\end{equation}
where 
\begin{equation}
\label{eqn:zeta def}
\zeta : \quot_{1,\dots,d}  \rightarrow C^d \times_{C^{(\bd)}} \quot_{\bd}
\end{equation}
is induced by diagram \eqref{eqn:cartesian}.

\end{proposition}

As a consequence of \eqref{eqn:zeta push}, we have the following base change equivalences
\begin{equation}
\label{eqn:equiv cartesian}
Rp_{\bd*} \circ Ls_d^* \cong Ls_{\bd}^* \circ R\rho_{\bd*}
\end{equation}
\begin{equation}
\label{eqn:equiv cartesian other way}
Rs_{d*} \circ Lp_{\bd}^* \cong L\rho_{\bd}^* \circ Rs_{\bd*}
\end{equation}
for any composition $\bd$ of $d$. The proof of Proposition \ref{prop:cartesian} will be given in the Appendix, as it relies on a local (i.e. $C = \BA^1$) computation.

\subsection{The main functors}
\label{sub:mainfunctors}

For any composition $\bd = (d_0, d_1, \dots, d_{r-1})$ of $d$, the nested Quot scheme \eqref{eqn:nested quot scheme} carries the line bundle
\begin{equation}
\label{eqn:line bundle composition}
\CL_{\bd} = \bigotimes_{i=1}^r \CL_i^{i-1}
\end{equation}
with the line bundles $\CL_1,\dots,\CL_r$ as in \eqref{eqn:line bundles}. Let
$$
t_{\bd} : \quot_{\bd} \to \quot_d
$$
denote the projection map which only remembers the deepest sheaf $E^{(d)}$ in a flag. We now introduce the functors which are the building blocks of the derived category $\DD_{\quot_d}$.

\begin{definition}
\label{def: mainfunctors} 

For any composition $\bd$ of $d$, we let 
$$
\be_{\bd}^{\ered}: \DD_{C^{(\bd)}} \to \DD_{\quot_d}
$$ 
denote the composition
\begin{equation}
\label{eqn:the functors}
\be_{\bd}^{\ered} : \DD_{C^{(\bd)}} \xrightarrow{Ls_{\bd}^*} \DD_{\quot_{\bd}} \xrightarrow{\otimes \CL_{\bd}} \DD_{\quot_{\bd}} \xrightarrow{Rt_{\bd*}} \DD_{\quot_d}.
\end{equation}
Equivalently, $\be_{\bd}^{\ered}$ is determined by the correspondence
\begin{equation}
\label{eqn:correspondence red}
\mathbb L_{\bd} = R (t_{\bd} \times s_{\bd})_* \left(\CL_{\bd} \right) \in \emph{Ob}(\DD_{\quot_d \times C^{(\bd)}}).
\end{equation}

\end{definition}

\subsection{Examples: tautological bundles}
\label{sub:ex taut}

The most basic case of the functors \eqref{eqn:the functors} occurs for the composition $(d) = (d,0,\dots,0)$, in which case
\begin{equation}
\label{eqn:example d}
\be_{(d)}^{\red} (\gamma) = Ls_{(d)}^*(\gamma)
\end{equation}
for all $\gamma \in \text{Ob}(\DD_{C^{(d)}})$. Some more involved examples can be provided using tautological bundles. To define them, recall the universal short exact sequence
$$
0 \rightarrow \CE \rightarrow \rho^*(V) \rightarrow \CF \rightarrow 0
$$
on $\quot_d \times C$, where $\pi,\rho : \quot_d \times C \rightarrow \quot_d, C$ are the standard projections. 

\begin{definition}
\label{def:tautological bundle early}

For any line bundle $M$ on $C$, we let
\begin{equation}
\label{eqn:tautological bundle early}
M^{[d]} = \pi_* \left( \CF \otimes \rho^*(M) \right)
\end{equation}
and call it a tautological bundle on $\quot_d$.

\end{definition}

\noindent In Proposition \ref{prop:tautological bundle}, we will prove that for any line bundle $M \to C$ and any $\ell$, we have
\begin{equation}
\label{eqn:wedge to e statement}
\wedge^\ell M^{[d]} = \be_{(d-\ell,\ell,0,\dots,0)}^{\red}  \left(\CO^{(d-\ell)} \boxtimes M^{(\ell)} \right)
\end{equation}
where $M^{(\ell)}$ denotes the line bundle on $C^{(\ell)}$ obtained by descending $M^\ell$ on $C^\ell$, 
\begin{equation}
\label{eqn:sym line bundle}
\rho_{(\ell)}^*(M^{(\ell)}) = M^\ell.
\end{equation}
As an illustration of \eqref{eqn:wedge to e statement}, let us prove its particular case $\ell = 1$ which holds in fact for vector bundles $M$ of arbitrary rank. By the definition of $\be^{\red}_{(d-1,1,0,\dots,0)}$, consider the nested Quot scheme
$$
\quot_{1, d} = \{E^{(d)} \subset E^{(1)} \subset V\}
$$
with the quotients $V/E^{(d)}$ and $V/E^{(1)}$ of lengths $d$ and $1$ respectively. We will write $p_C: \quot_{1, d} \to C$ for the morphism which remembers the support point of $V/E^{(1)}$. With this in mind, Lemma \ref{lem:pi minus general} implies that the map 
$$
t_{(d-1,1,0,\dots,0)} \times p_C: \quot_{1, d} \to \quot_d \times C
$$ 
is the projectivization of the rank 0 universal quotient sheaf $\CF^{(d)} = V/\CE^{(d)}$. Therefore
$$ 
\be_{(d-1,1,0,\dots,0)}^{\red}  \left(\CO^{(d-1)} \boxtimes M \right) = R\pi_* \left[ R(t_{(d-1,1,0,\dots,0)} \times p_C)_*(\mathcal L_2 \otimes p_C^*(M)) \right]
$$
where $\pi : \quot_d \times C \rightarrow \quot_d$ denotes the standard projection. The line bundle $\CL_2$ (in the notation \eqref{eqn:line bundles}) corresponds to $\CO(1)$ on the projectivization, so we may invoke Lemma \ref{lem:push d virtual} to calculate the interior of the square bracket above
\begin{equation}
 \label{eqn:taut0}
 \be_{(d-1,1,0,\dots,0)}^{\red}  \left(\CO^{(d-1)} \boxtimes M \right) = R\pi_* \left( \CF^{(d)} \otimes \rho^*(M) \right) = M^{[d]},
\end{equation}
exactly as prescribed by the $\ell=1$ case of \eqref{eqn:wedge to e statement}.

\subsection{Invariants}

For any composition $\bd$ of $d$, the functor $\be_{\bd}^{\red}$ is related to $\be_{\bd}: \DD_{C^d} \to \DD_{\quot_d}$ of \eqref{eqn:basis} in a simple manner which we now describe. Consider the commutative diagram
\begin{equation}
\label{eqn:consider the maps}
\xymatrix{\quot_{1,\dots,d} \ar@/_2pc/[rr]_{t_d} \ar[r]^-{p_{\bd}} & \quot_{\bd} \ar[r]^-{t_{\bd}} & \quot_d}
\end{equation}
where $t_d, t_{\bd}$ only remember the sheaf $E^{(d)}$ in a flag. We note that 
$$
p_{\bd}^*(\mathcal L_{\bd}) = {\tilde {\mathcal L}_{\bd}} := (\CL_1 \dots \CL_{d_0})^{0} \otimes (\CL_{d_0+1} \dots \CL_{d_0+d_1})^{1} \otimes \dots \otimes (\CL_{d_0+\dots+d_{r-2}+1} \dots \CL_d)^{r-1}
$$
on $\quot_{1,\dots,d}$, where $\CL_1,\dots,\CL_d$ are the line bundles from \eqref{eqn:line bundles}, associated with successive quotients on the fully nested Quot scheme $\quot_{1, \ldots, d}$. By definition, we have
$$
\be_{\bd} (-) = Rt_{d*} ({\tilde \CL}_{\bd} \otimes Ls_d^*(-))
$$
Equivalently, $\be_{\bd}$ is the functor determined by the correspondence
\begin{equation}
\label{eqn:correspondence non red}
{\tilde {\mathbb L}_{\bd}} = R(t_d \times s_d)_* ( {\tilde {\mathcal L}_{\bd}} ) \in \text{Ob}(\DD_{\quot_d \times C^{d}})
\end{equation}
(compare with \eqref{eqn:correspondence red}). As ${\tilde {\CL}}_{\bd} = p_{\bd}^*( \CL_{\bd})$, we invoke $t_d = t_{\bd} \circ p_{\bd}$ of \eqref{eqn:consider the maps} to write instead
\begin{equation}
\label{eqn:e to e red}
\be_{\bd} = Rt_{\bd*} \left( \CL_{\bd} \otimes Rp_{\bd*}  \circ Ls_d^* \right) \stackrel{\eqref{eqn:equiv cartesian}}= Rt_{\bd*} \left( \CL_{\bd} \otimes Ls_{\bd}^*\circ R\rho_{\bd*} \right) = \be_{\bd}^{\red} \circ R\rho_{\bd*}
\end{equation}

\begin{lemma}
\label{lem:invariants}

There is an action 
$$
S_{\bd} \curvearrowright \be_{\bd} \circ L\rho_{\bd}^*(\gamma) \in \emph{Ob}(\DD_{\quot_d})
$$
which is natural in $\gamma \in \emph{Ob}(\DD_{C^{(\bd)}})$, such that
\begin{equation}
\label{eqn:invariants}
\be_{\bd}^{\ered} (\gamma) = \Big( \be_{\bd} \circ L\rho_{\bd}^*(\gamma) \Big)^{S_{\bd}}
\end{equation}

\end{lemma}

\begin{proof} Although we phrased the statement above in terms of objects, it is actually a statement about correspondences. Specifically, \eqref{eqn:correspondence non red} implies that the functor $\be_{\bd} \circ L\rho_{\bd}^*$ is given by the correspondence
$$
R(\text{Id}_{\quot_d} \times \rho_{\bd})_* ( {\tilde {\mathbb L}_{\bd}} ) = R(t_d 
 \times \rho_{\bd} \circ s_{d})_* ( \tCL_{\bd} ) = 
$$
\begin{equation}
\label{eqn:corr}
= R(t_{\bd}  \times s_{\bd})_* \left( Rp_{\bd*} (\tCL_{\bd})  \right) = R(t_{\bd}  \times s_{\bd})_* \left( \CL_{\bd} \otimes Rp_{\bd*} (\CO_{\quot_{1,\dots,d}}) \right)
\end{equation}
The domain and target of the morphisms $s_{\bd}$, $t_{\bd}$ are fixed by $S_{\bd}$, as is the line bundle $\CL_{\bd}$. However, due to \eqref{eqn:equiv cartesian}, we observe that
\begin{equation}
\label{eqn:an equality}
Rp_{\bd*} (\CO_{\quot_{1,\dots,d}}) = L s_{\bd}^* \Big( R\rho_{\bd*}(\CO_{C^d}) \Big)
\end{equation}
Because the map $\rho_{\bd}$ is a $S_{\bd}$ quotient, the coherent sheaf $R\rho_{\bd*}(\CO_{C^d})$ has an action of $S_{\bd}$, whose invariants we have already seen are simply $\CO_{C^{(\bd)}}$. Therefore, the right-hand side of \eqref{eqn:corr} has an action of $S_{\bd}$, whose invariants are $R(t_{\bd}  \times s_{\bd})_* ( \CL_{\bd})$. By \eqref{eqn:correspondence red}, the latter object is the correspondence which produces the functor $\be_{\bd}^{\red}$, so we are done.

\end{proof}

\subsection{The main theorem}

We are now ready to upgrade Proposition \ref{prop:semi} to our Theorem \ref{thm:intro semi}, which states that the functors $\be_{\bd}^{\red}$ are fully faithful and semi-orthogonal in the sense that
\begin{equation}
\label{eqn:fully faithful}
\text{Hom}_{\quot_d} \Big ( \be^{\red}_{\bd}  (\gamma), \be^{\red}_{\bd'}  (\gamma') \Big) \cong \begin{cases} 0&\text{if } \bd < \bd' \text{ lexicographically}\\ \text{Hom}_{C^{(\bd)}} (\gamma, \gamma' ) &\text{if }\bd = \bd' \end{cases}
\end{equation}
and moreover their essential images generate $\DD_{\quot_d}$ as a triangulated category.

\begin{proof} \emph{of Theorem \ref{thm:intro semi}:} As a consequence of \eqref{eqn:invariants}, $\be_{\bd}^{\red}$ is a direct summand of $\be_{\bd} \circ L\rho_{\bd}^*$, so the $\bd < \bd'$ statement of \eqref{eqn:fully faithful} follows from \eqref{eqn:semi}. Let us now prove the $\bd = \bd'$ statement of \eqref{eqn:fully faithful}. By analogy with Lemma \ref{lem:adjoint}, the functor $\be^{\red}_{\bd}$ has the left adjoint
\begin{equation}
\label{eqn:the functors 2}
 \tf_{\bd}^{'\red} : \DD_{\quot_d} \xrightarrow{Lt_{\bd}^*} \DD_{\quot_{\bd}} \xrightarrow{\omega_{s_{\bd}} \otimes \CL_{\bd}^{-1}}  \DD_{\quot_{\bd}} \xrightarrow{Rs_{\bd*}} \DD_{C^{(\bd)}}
\end{equation}
where $\omega_{s_{\bd}}$ denotes the relative canonical line bundle of the morphism $s_{\bd}$, shifted by $d(r-1)$ in the derived category (it arises in the formula above as the difference between the functors $s_{\bd}^!$ and $Ls_{\bd}^*$). We have the following analogue of \eqref{eqn:e to e red}
\begin{equation}
\label{eqn:f to f red}
\tf_{\bd}^{'} = Rs_{d*} (\omega_{s_d} \otimes \tCL_{\bd}^{-1} \otimes Lt_{d}^* ) = L \rho_{\bd}^* \circ Rs_{\bd*} \left(\omega_{s_{\bd}} \otimes \CL_{\bd}^{-1} \otimes Lt_{\bd}^* \right)  =  L \rho_{\bd}^* \circ \tf_{\bd}^{'\red} 
\end{equation}
where we write $\tf'_{\bd}$ for the same product of $\tf'_{-i-r}$ as $\be_{\bd}$ is a product of $\be_i$'s, but in opposite order. The middle equality above is a consequence of \eqref{eqn:equiv cartesian other way} and the identity
\begin{equation}
\label{eqn:push canonical}
R\zeta_*(\omega_{s_d}) = {\tilde \rho_{\bd}}^{\,*}(\omega_{s_{\bd}}),
\end{equation}
where ${\tilde \rho_{\bd}}: \quot_{\bd} \times_{C^{(\bd)}} C^d \to \quot_{\bd}$ is the base change of $\rho_{\bd}: C^d \to C^{(\bd)}.$
In turn, \eqref{eqn:push canonical} is due to the fact that the map $\zeta$ of \eqref{eqn:zeta def} is a resolution of rational singularities, see \eqref{eqn:zeta push}. Just like the functors $\be_{\bd}^{\red}$ and $\be_{\bd}$ are given by the correspondences \eqref{eqn:correspondence red} and \eqref{eqn:correspondence non red} respectively, the functors $\tf_{\bd}^{'\red}$ and $\tf_{\bd}^{'}$ are given by the following correspondences
\begin{align}
&{\mathbb L}_{\bd}^{\text{adj}} = R (t_{\bd} \times s_{\bd})_* \left(\omega_{s_{\bd}} \otimes \CL_{\bd}^{-1}  \right) \in \text{Ob}(\DD_{\quot_d \times C^{(\bd)}}) \label{eqn:correspondence red f} \\
&{\tilde {\mathbb L}_{\bd}}^{\text{adj}} = R (t_{d} \times s_{d})_* \left(\omega_{s_d}\otimes {\tilde \CL_{\bd}}^{-1} \right) \in \text{Ob}(\DD_{\quot_d \times C^{d}}) \label{eqn:correspondence non red f}
\end{align}
respectively. We have the following isomorphism in the derived category of $C^d \times C^d$
\begin{equation}
\label{eqn:iso correspondences}
\Phi : L(\rho_{\bd} \times \rho_{\bd})^* \left (\mathbb L_{\bd}^{\text{adj}} \circ \mathbb L_{\bd} \right ) \cong {\tilde {\mathbb L}_{\bd}}^{\text{adj}} \circ {\tilde {\mathbb L}_{\bd}} \stackrel{\eqref{eqn:long composition}}\cong \CO_{C^d \times_{C^{(\bd)}} C^d} = L(\rho_{\bd} \times \rho_{\bd})^* \left (\CO_{\Delta} \right)
\end{equation}
where the left-most isomorphism follows by translating \eqref{eqn:e to e red} and \eqref{eqn:f to f red} into the language of correspondences. We claim that $\Phi$ above can be upgraded to an isomorphism of $S_{\bd} \times S_{\bd}$ equivariant objects. Indeed, the left and right-hand sides of this equation are manifestly $S_{\bd} \times S_{\bd}$ equivariant sheaves. Thus, the failure of the isomorphism $\Phi$ to commute with the $S_{\bd} \times S_{\bd}$ action can be measured by an invertible degree 0 endomorphism of the object $\CO_{C^d \times_{C^{(\bd)}} C^d}$. All such endomorphisms are given by multiplication with a non-zero scalar, and this scalar must be 1 because $\Phi$ is compatible with the non-zero counit map $\mathbb L_{\bd}^{\text{adj}} \circ \mathbb L_{\bd} \rightarrow \CO_{\Delta}$. Thus, we conclude that $\Phi$ is an isomorphism in the $S_{\bd} \times S_{\bd}$ equivariant derived category of $C^d \times C^d$, so it descends to an isomorphism 
$$
\mathbb L_{\bd}^{\text{adj}} \circ \mathbb L_{\bd} \cong \CO_{\Delta}
$$
in the derived category of $C^{(\bd)} \times C^{(\bd)}$. This is precisely equivalent to the fully faithfulness of $\be_{\bd}^{\red}$, i.e. \eqref{eqn:fully faithful} for $\bd = \bd'$. To complete the proof of Theorem \ref{thm:intro semi}, it remains to prove that the essential images of the functors $\be_{\bd}^{\red}$ generate the category $\DD_{\quot_d}$.

\begin{lemma}
\label{lem:vice versa}

For any composition $\bd$, we have an equality of essential images
\begin{equation}
\label{eqn:direct summands}
\emph{Im }\be^{\ered}_{\bd}  = \Big(\text{direct summands of objects in }  \emph{Im }\be_{\bd} \Big)
\end{equation}
    
\end{lemma}

\begin{proof}

In \eqref{eqn:invariants}, we showed that any object in $\text{Im }\be^{\red}_{\bd}$ is the $S_{\bd}$ invariant part of an object in $\text{Im }\be_{\bd}$, which implies the inclusion $\subseteq$ of \eqref{eqn:direct summands}. For the opposite inclusion, consider a direct summand $A$ of some $\be_{\bd}  (\gamma)$, which corresponds to a decomposition
$$
\be_{\bd}  (\gamma) \cong A \oplus \overline{A}
$$
The endomorphism $\phi : \be_{\bd}  (\gamma) \rightarrow \be_{\bd}  (\gamma)$ given by projection onto $A$ satisfies $\phi^2 = \phi$. Since $\be_{\bd} = \be_{\bd}^{\red} \circ R\rho_{\bd *}$, the $\bd = \bd'$ case of formula \eqref{eqn:fully faithful} gives rise to an endomorphism
$$
\phi' : R\rho_{\bd*}(\gamma) \rightarrow R\rho_{\bd*}(\gamma)
$$
such that ${\phi'}^2 = \phi'$. Because the derived category of coherent sheaves on a smooth variety is idempotent complete, we may write
$$
R\rho_{\bd*}(\gamma) \cong A' \oplus \overline{A}'
$$
with $\phi'$ being projection onto $A'$. Since the morphism $\phi$ is obtained by applying the functor $\be^{\red}_{\bd}$ to the morphism $\phi'$, we conclude that
\begin{equation}
\label{eqn:need}
\be^{\red}_{\bd}  (A') \cong A
\end{equation}
This establishes the $\supseteq$ inclusion of \eqref{eqn:direct summands}.

\end{proof}

\noindent Finally, we need to show that the essential images of the functors $\be_{\bd}^{\red}$ generate $\DD_{\quot_d}$ as a triangulated category. Consider the maps in diagram \eqref{eqn:cartesian} for the composition $(d) = (d,0,\dots,0)$. For any $\gamma \in \text{Ob}(\DD_{\quot_d})$, there exists an action
$$
S_d \curvearrowright \gamma \otimes Ls_{(d)}^* \Big( R\rho_{(d)*}(\CO_{C^d}) \Big)  \stackrel{\eqref{eqn:an equality}}\cong \gamma \otimes Rp_{(d)*}(\CO_{\quot_{1,\dots,d}}) = Rp_{(d)*} \Big(Lp_{(d)}^*(\gamma) \Big) 
$$
By \eqref{eqn:invariant claim}, the $S_d$ invariant object of the action above is precisely $\gamma$, so we conclude that any object in $\DD_{\quot_d}$ is a direct summand of an object in the category
\begin{equation}
\label{eqn:the category}
Rp_{(d)*} \left(\text{Ob}(\DD_{\quot_{1,\dots,d}}) \right)
\end{equation}
To understand the category \eqref{eqn:the category}, we recall the composition
\begin{equation}
\label{eqn:tower}
s_d : \quot_{1,\dots,d} \rightarrow \quot_{1,\dots,d-1} \times C \rightarrow \dots \rightarrow \quot_{1} \times C^{d-1} \rightarrow C^d
\end{equation}
for which (according to Lemma \ref{lem:pi minus}) all the $d$ arrows are projectivizations of rank $r$ vector bundles. Iterating \eqref{eqn:orlov} $d$ times implies that $\DD_{\quot_{1,\dots,d}}$ is generated by
$$
\Big\{\CL_1^{k_1} \otimes \dots \otimes \CL_d^{k_d} \otimes \text{Im }Ls^*_d\Big\}_{0 \leq k_1, \dots, k_d < r}.$$ In turn, this means that the category $Rp_{(d)*} \left(\DD_{\quot_{1,\dots,d}} \right)$ is generated by $$\Big\{ \text{Im } \be_{k_1} \circ \dots \circ \be_{k_d} : \DD_{C^d} \rightarrow \DD_{\quot_d} \Big\}_{0 \leq k_1, \dots, k_d < r}.$$
However, because of Proposition \ref{prop:quadratic relation e} \footnote{Or more precisely, because of the straightforward analogue of this Proposition when the $\te_i$ are replaced by $\be_i$. Explicitly, this analogous statement is that $\forall i < j$ there is a natural transformation
\begin{equation}
\label{eqn:quadratic relation e tilde}
\textcolor{blue}{\te_j} \circ \textcolor{red}{\te_i}  \rightarrow \textcolor{red}{\te_i} \circ \textcolor{blue}{\te_j} \ \  : \ \ D_{\quot_{d-1} \times \textcolor{red}{C} \times \textcolor{blue}{C}} \rightarrow D_{\quot_{d+1}}
\end{equation}
whose cone is filtered with associated graded $\bigoplus_{k=i}^{j-1} \te_k \circ \te_{i+j-k} \circ (\text{Id}_{\quot_{d-1}} \times \Delta)_* \circ (\text{Id}_{\quot_{d-1}} \times \Delta)^*$.}, the essential images of the functors $\be_{k_1} \circ \dots 
 \circ \be_{k_d}$ are actually generated as triangulated subcategories by the essential images of the same functors with $k_1 \leq \dots \leq k_d$. Putting all the facts above together implies that $\DD_{\quot_d}$ consists of direct summands of complexes built out of 
$$
\Big\{ \text{Im } \be_{k_1} \circ \dots \circ \be_{k_d} : \DD_{C^d} \rightarrow \DD_{\quot_d} \Big\}_{0 \leq k_1 \leq  \dots \leq k_d < r}.
$$
By Lemma \ref{lem:vice versa}, we conclude now that $\DD_{\quot}$ is generated by the categories $\text{Im }\be_{\bd}^{\red}$, by invoking the following result applied in our case to the category \eqref{eqn:the category} (cf \cite[Corollary A.12]{ls}; note that the admissibility hypothesis of \emph{loc. cit.} does indeed apply to semi-orthogonal decompositions of the form \eqref{eqn:orlov}, which are the only ones considered in the present paper, due to Theorem 3.1 of \emph{loc. cit.}).

\begin{lemma}

If a triangulated category $\mathscr{C}$ has a semi-orthogonal decomposition in terms of triangulated subcategories $\mathscr{C}_1,\dots,\mathscr{C}_n$, then the idempotent completion of $\mathscr{C}$ has a semi-orthogonal decomposition in terms of idempotent completions of $\mathscr{C}_1,\dots,\mathscr{C}_n$.

\end{lemma}\end{proof}

\section{The cohomology of tautological bundles on Quot schemes}
\label{sec:tautological}

\vskip .2in

\subsection{Tautological bundles} 

Recall the rank $d$ tautological bundle $M^{[d]}$ of Definition \ref{def:tautological bundle early}, associated to any line bundle $M$ on $C$. The main purpose of this Section is to prove Theorem \ref{thm:tautological intro} on the cohomology of various exterior powers of such tautological bundles. Before we do so in full generality, we will deal with the special case of the determinant line bundle $\det M^{[d]} \to \quot_d.$ Using the Grothendieck-Riemann-Roch theorem, it is easy to show that
\begin{equation}
\label{eqn:det}
\det M^{[d]} \cong \CL \otimes s_{(d)}^*(M^{(d)})
\end{equation}
where $\CL = \det R\pi_* \CF$ and $s_{(d)} : \quot_d \rightarrow C^{(d)}$ is the support morphism.  

\begin{proposition}
\label{prop:tautological}

For any line bundle $M$ on $C$, we have an isomorphism 
\begin{equation}
\label{eqn:iso tautological lemma}
H^\bull\left(\quot_d, \det M^{[d]} \right) \cong \wedge^d H^\bull(C,V \otimes M)
\end{equation}
as graded vector spaces.

\end{proposition}

\begin{proof} We henceforth write $\CE^{(b)}_a$ for the universal sheaf on $\quot_{1,\dots,k} \times C^{\ell}$ which is pulled back from $\quot_b$ times the $a$-th copy of $C$, for all applicable indices. We will often write $C_1 \times \dots \times C_d$ instead of $C^d$, in order to differentiate between the $d$ factors. Similarly, for any vector bundle $X$ on $C$, we will write $X_a$ for its pull-back to $C^d$ from $C_a$. Consider the maps in the diagram \eqref{eqn:cartesian} for the composition $(d) = (d,0,\dots,0)$
$$
\xymatrix{& \quot_{1,\dots,d} \ar[ld]_{s_d} \ar[rd]^{p_{(d)}} \\ C^{d} \ar[dr]_{\rho_{(d)}} & & \quot_{d} \ar[ld]^{s_{(d)}} \\ & C^{(d)} &}
$$
We have the following equalities
\begin{equation}
\label{eqn:first step}
\begin{split}
 H^\bull\left(\quot_d, \det M^{[d]} \right) &= H^\bull \left ( C^{(d)}, \, {Rs_{(d)*}} (\CL \otimes  s_{(d)}^*(M^{(d)}) ) \right) \\
 & = H^\bull \left ( C^{(d)}, \, {Rs_{(d)*}} (\CL) \otimes M^{(d)} \right) \\
 & = H^\bull \left ( C^{d}, \, L\rho_{(d)}^*({Rs_{(d)*}} (\CL))  \otimes  M^{d} \right)^{S_d} \\
 & \stackrel{\eqref{eqn:equiv cartesian other way}}= H^\bullet \Big(C^d, Rs_{d*} \left( \CL_1 \otimes \dots \otimes \CL_d \right) \otimes M_1 \otimes \dots \otimes M_d   \Big)^{S_d}
\end{split} 
\end{equation}
As it is isomorphic to a pullback from $C^{(d)}$, the complex $Rs_{d*} \left( \CL_1 \otimes \dots \otimes \CL_d \right)$ is naturally $S_d$ equivariant. If we let $\chi$ denote the structure sheaf $C^d$ endowed with the sign character of $S_d$, we claim the following isomorphism of $S_d$-equivariant complexes:
\begin{equation}
\label{eqn:big equation}
Rs_{d*} \left( \CL_1 \otimes \dots \otimes \CL_d \right) \cong \left (V_1 \otimes \dots \otimes V_d   \mathsf{\, + \, Error} \right ) \otimes \chi
\end{equation}
where on the right-hand side the tensor product $V_1 \otimes \dots \otimes V_d$ has the natural $S_d$ action by permutation of the factors, and $\mathsf{\, + \, Error}$ refers to taking extensions by objects in the derived category of $C^d$ which are preserved by some transposition in $S_d$. Once we establish \eqref{eqn:big equation}, Proposition \ref{prop:tautological} follows as a consequence of \eqref{eqn:first step}, since 
$$
H^\bullet \left(C^d, \Big(V_1 \otimes \dots \otimes V_d  \mathsf{\, + \, Error}\Big) \otimes M_1 \otimes \dots \otimes M_d \otimes \chi  \right)^{S_d} = \wedge^d H^\bullet(C,V \otimes M)
$$
with the sign character being responsible for the exterior algebra in the right-hand side. The error term does not contribute $S_d$ invariants (upon tensoring with the sign character $\chi$) since each one of its summands is preserved by some transposition.

\medskip 

It remains thus to prove the isomorphism \eqref{eqn:big equation}. For a clearer exposition, we will first establish it {\it non-equivariantly}, and then account for the $S_d$ equivariance. Our approach is to interpret $s_d$ as the projection from the total space $\quot_{1, \ldots, d}$ of a tower of projective bundles to the base, i.e. to factor it as
\begin{multline}
\label{eqn:tower final}
s_d : \quot_{1,\dots,d} \xrightarrow{p_1} \quot_{1,\dots,d-1} \times C_1 \xrightarrow{p_2} \dots \\ \dots \xrightarrow{p_{d-1}} \quot_{1} \times C_1 \times \dots \times C_{d-1} \xrightarrow{p_d} C_1 \times \dots \times C_{d}
\end{multline}
Each map $p_j$ is the projectivization of the rank $r$ locally free sheaf $\CE^{(d-j)}_j$ on $\quot_{1,\dots,d-j}$ $\times C_1 \times \dots \times C_j$, with tautological line bundle $\CL_j$. Thus, formula \eqref{eqn:push d} gives us
\begin{equation}
\label{eqn:push final}
Rp_{j*}(\CL_j^n) = S^{n} \CE_j^{(d-j)} 
\end{equation}
for all $n \geq 0$. We prove \eqref{eqn:big equation} by factoring $Rs_{d*} = Rp_{d*} \circ \dots \circ Rp_{1*}$ and computing the push-forwards one at a time. Specifically, we will prove the following isomorphism in $\DD_{\quot_{1,\dots,d-k} \times C_1 \times \dots \times C_k}$ by induction on $k$,
\begin{equation}
\label{eqn:josie}
Rp_{k*} \circ \dots \circ Rp_{1*} \left(\CL_1 \otimes \dots \otimes \CL_d \right) = \CE^{(d-k)}_1 \otimes \dots \otimes \CE^{(d-k)}_k \otimes \CL_{k+1} \otimes \dots \otimes \CL_d \mathsf {\, +\, Error}.
\end{equation}
The key to establishing the induction step of \eqref{eqn:josie} is to calculate 
$$
Rp_{k+1*} \left(\CE^{(d-k)}_1 \otimes \dots \otimes \CE^{(d-k)}_k \otimes \CL_{k+1} \otimes \dots \otimes \CL_d \right).
$$
On $\quot_{1,\dots,d-k} \times C_1 \times \dots \times C_k$, we have the following exact sequence $\forall a \in \{1,\dots,k\}$
$$
0 \rightarrow \CE^{(d-k)}_a \rightarrow \CE^{(d-k-1)}_a \rightarrow \CL_{k+1} \otimes \CO_{\Delta_{a\sharp}} \rightarrow 0 
$$
where $\Delta_{a\sharp}$ denotes the codimension one diagonal identifying the curve $C_a$ with the $d-k$-th support point of $\quot_{1,\dots,d-k}$. Thus, on $\quot_{1,\dots,d-k} \times C_1 \times \dots \times C_k$ we obtain an isomorphism
$$
\CE^{(d-k)}_1 \otimes \dots \otimes \CE^{(d-k)}_k \otimes \CL_{k+1} \otimes \dots \otimes \CL_d \cong 
$$
$$
\left( \CE^{(d-k-1)}_1 - \CL_{k+1} \otimes \CO_{\Delta_{1\sharp}} \right) \otimes \dots \otimes \left( \CE^{(d-k-1)}_k - \CL_{k+1} \otimes \CO_{\Delta_{k\sharp}} \right) \otimes \CL_{k+1} \otimes \dots \otimes \CL_d 
$$
where $-$ indicates some non-trivial extension (shifted by 1 in the derived category) which will not be important to us. Applying the functor $Rp_{k+1*}$ to the formula above, yields an extension of the main term
\begin{equation}
\label{eqn:main object}
\CE^{(d-k-1)}_1 \otimes \dots \otimes \CE^{(d-k-1)}_k \otimes \CE^{(d-k-1)}_{k+1} \otimes \CL_{k+2} \otimes \dots \otimes \CL_d
\end{equation}
by shifts of the following objects, as $A = \{a_1 < \dots < a_t\}$ runs over non-empty subsets of $\{1,\dots,k\}$:
\begin{equation}
\label{eqn:error object}
\Delta_{a_1,\dots,a_t,k+1 *} \left(S^{t+1} \CE^{(d-k-1)} \right) \bigotimes_{b \in \{1,\dots,k\} \backslash A} \CE^{(d-k-1)}_b \otimes \CL_{k+2} \otimes \dots \otimes \CL_d.
\end{equation}
When $k+1=d$, the term \eqref{eqn:main object} is precisely $V_1 \otimes \dots \otimes V_d$ and the various terms \eqref{eqn:error object} are the error objects in formula \eqref{eqn:big equation}. 

\vskip.1in

To establish \eqref{eqn:big equation} as an {\it equivariant} isomorphism,  we need to describe the $S_d$ action on the left-hand side and to show that every error object is preserved by some transposition $\sigma_i = (i \ i+1)$. As $S_d$ does not act on $\quot_{1, \dots, d}$, we will use the following auxiliary construction. For any $i \in \{1,\dots,d-1\}$, consider the moduli space $\fY_i$ of subsheaves
\begin{equation}
\label{eqn:yi}
\xymatrix{& & & E^{(d-i)} \ar@{^{(}->}[rd]^-{x_{i+1}} & & & & \\
E^{(d)} \ar@{^{(}->}[r]^-{x_{1}} & \dots \ar@{^{(}->}[r]^-{x_{i-1}} & E^{(d-i+1)} \ar@{^{(}->}[ru]^{x_i} \ar@{^{(}->}[rd]_{x_{i+1}}
& & E^{(d-i-1)} \ar@{^{(}->}[r]^-{x_{i+2}} & \dots \ar@{^{(}->}[r]^-{x_{d}}  & E^{(0)} = V \\
& & & \widetilde{E}^{(d-i)} \ar@{^{(}->}[ru]_{x_i} & & & &}  
\end{equation}
which has a natural $S_2$ action by permuting $E^{(d-i)}$ and $\tE^{(d-i)}$. We have maps $\fY_i \xrightarrow{\pi^\uparrow_i,\pi^\downarrow_i} \quot_{1,\dots,d}$ given by forgetting $\tE^{(d-i)}$ and $E^{(d-i)}$, respectively. The composition 
$$
r_i : \fY_i \xrightarrow{\pi^\uparrow_i} \quot_{1,\dots,d} \xrightarrow{s_d} C^d
$$
is equivariant with respect to $S_2 = \langle \sigma_i \rangle$. The following is proved in the Appendix. 

\begin{lemma}
\label{lem:big}

For any $i \in \{1,\dots,d-1\}$, we have 
\begin{equation}
\label{eqn:blow up i}
R\pi^\uparrow_{i*}(\CO_{\fY_i}) = \CO_{\quot_{1,\dots,d}} 
\end{equation}
(and similarly for $\downarrow$). For any $\gamma \in \emph{Ob}(\DD_{\quot_{d}})$, the object
$$
\Gamma = Rs_{d*} \circ Lp_{(d)}^*(\gamma)  = Rr_{i*} (L\pi_i^{\uparrow *} \circ Lp_{(d)}^*(\gamma) ) \in \emph{Ob} \left(\DD_{C^d}\right)
$$
has the following two $S_2$ equivariant structures (i.e. isomorphisms $\sigma_i^*(\Gamma) \cong \Gamma$ which square to the identity)

\begin{itemize}[leftmargin=*]

\item the first is inherited from the $S_d$ action on 
$$
\Gamma \stackrel{\eqref{eqn:equiv cartesian other way}}{\cong}L\rho_{(d)}^* \circ Rs_{(d)*} (\gamma)
$$

\item the second is obtained by pushing forward the equality
$$
\sigma_i^* \left(L\pi_i^{\uparrow *} \circ Lp_{(d)}^*(\gamma) \right) = L\pi_i^{\uparrow *} \circ Lp_{(d)}^*(\gamma)
$$
under the equivariant map $r_i$ (thus obtaining $\sigma^*_i(\Gamma)$ in the LHS and $\Gamma$ in the RHS).

\end{itemize}

\noindent The two equivariant structures above coincide.

\end{lemma}

With Lemma \ref{lem:big} in mind, let us upgrade \eqref{eqn:big equation} to an isomorphism of $S_d$ equivariant sheaves. It suffices to check by induction on $k$ that 
\begin{multline}
\label{eqn:cat}
Rp_{k+1*} \circ Rp_{k*} \circ \dots \circ Rp_{1*} \left(\CL_1 \otimes \dots  \otimes \CL_d \right) = \\ = \left ( \CE^{(d-k-1)}_1 \otimes \dots  \otimes  \CE^{(d-k-1)}_k \otimes  \CE^{(d-k-1)}_{k+1} \otimes \CL_{k+2} \otimes \dots \CL_d \mathsf {\, + \, Error} \right ) \otimes \chi
\end{multline}
holds $S_{k+1}$-equivariantly on $\quot_{1, \dots, d -k-1} \times C_1 \times \dots \times C_{k+1}$. It is moreover enough to verify the action of the transpositions $\sigma_i = (i \ i+1)$ with $i \in \{1,\dots, k\}$.  The induction hypothesis and the fact that the induction step was symmetric in $\CE^{(b)}_1,\dots,\CE^{(b)}_k$ imply that the isomorphism \eqref{eqn:cat} is preserved by the transpositions $\sigma_1,\dots,\sigma_{k-1}$, so it remains to deal with the transposition $\sigma_k$. 

\medskip

To this end, we will use the space $\fY_k$ on which the transposition $\sigma_k$ acts. It comes with the line bundles $\CL_1, \dots, \CL_d$ pulled back from $\quot_{1, \ldots, d}$ via $\pi_k^{\uparrow},$ as well as  
$$
\tilde{\CL}_k = \sigma_k^*(\CL_k) ,  \, \, \,   \tilde{\CL}_{k+1} = \sigma_k^*(\CL_{k+1}), \, \, \text{satisfying} \, \, \CL_k \otimes \CL_{k+1} \cong \tilde{\CL}_k \otimes \tilde{\CL}_{k+1}.
$$
We let $\pi: \fY_k \times C \to \fY_k$ be the usual projection, and denote by $\Delta_1, \dots, \Delta_d$ the diagonals identifying points of $C$ with each of the support points $x_1,\dots,x_d$ in the notation of \eqref{eqn:yi}. We note that $S_2$-equivariantly, we have then on $\fY_k,$
\begin{equation}
L\pi_k^{\uparrow *} \circ Lp_{(d) }^*(\CL) = \det R\pi_*(\CF^{(d)}) = \det R\pi_* \left ( \sum_{i=1}^{d} \CL_i \otimes \CO_{\Delta_i} \right ) = \CL_1 \otimes \dots \otimes \CL_d \otimes \chi,
\end{equation}
since the $S_2$ action interchanges the $k$th and $k+1$st sheets inside the determinant. In addition to the alternation, the action of $\sigma_k$ on the object \eqref{eqn:main object} is obtained by calculating
\begin{equation}
\label{eqn:steve}
Rp_{k+1 *} \circ Rp_{k*} \circ R\pi_{k*}^{\uparrow}\left(\CL_k \otimes \CL_{k+1} \otimes \dots\right)
\end{equation}
as well as
\begin{equation}
\label{eqn:steve?}
Rp_{k+1 *} \circ Rp_{k*} \circ R\pi_{k*}^{\uparrow}\left(\sigma_k^*(\CL_k) \otimes \sigma_k^*(\CL_{k+1}) \otimes \dots\right)
\end{equation}
Formula \eqref{eqn:steve} equals
\begin{align*}
&Rp_{k+1 *} \left(\CE^{(d-k)}_k \otimes \CL_{k+1} \otimes \dots\right) = \\
=& Rp_{k+1 *} \left( \Big( \CE^{(d-k-1)}_k - \CL_{k+1} \otimes \CO_{\Delta_{k,k+1}} \Big) \otimes \CL_{k+1} \otimes \dots\right) = \\
=& \left( \CE^{(d-k-1)}_k  \otimes  \CE^{(d-k-1)}_{k+1} - \Delta_{k,k+1 *}\left(S^2 \CE^{(d-k-1)}\right) \right) \otimes \dots
\end{align*}
while \eqref{eqn:steve?} produces the same result, but with $\CE^{(d-k-1)}_k$ and $\CE^{(d-k-1)}_{k+1}$ swapped, while the term $\Delta_{k,k+1 *}(S^2 \CE^{(d-k-1)})$ is $\sigma_k$-invariant. This shows that the transposition $\sigma_{k}$ acts on the right-hand side of \eqref{eqn:cat} by the usual permutation, as claimed. Moreover, we may similarly control the $S_d$ equivariance of the objects \eqref{eqn:error object}: if $a_t = k$ then the argument above explains why the transposition $\sigma_k$ preserves the object in question. Meanwhile, if $a_t < k$, recall that the induction hypothesis implies that the transposition $(a_t \ k)$ sign-permutes $\CE^{(d-k-1)}_{a_t}$ and $\CE^{(d-k-1)}_{k}$ in the right-hand side of \eqref{eqn:josie}. Therefore, this transposition also permutes the object \eqref{eqn:error object} and the same-named object with $k$ instead of $a_t$. Because of this, the transposition
$$
(a_t \ k) (k \ k+1) (a_t \ k) = (a_t \ k+1)
$$
preserves the object \eqref{eqn:error object}, as we needed to show.

\end{proof}

\subsection{Exterior powers}

In Theorem \ref{thm:tautological intro}, we are interested in generalizing Proposition \ref{prop:tautological} to exterior powers of tautological bundles. To this end, we will compute these exterior powers in terms of nested Quot schemes.

\begin{proposition}
\label{prop:tautological bundle}

For any integers $0 \leq \ell \leq d$ and any line bundle $M$ on $C$, we have
\begin{equation}
\label{eqn:wedge to e}
\wedge^\ell M^{[d]} = Rp_{+*} \left(\CL_2 \otimes s_{(\ell)}^* ( M^{(\ell)} ) \right) = \be_{(d-\ell,\ell,0,\dots,0)}^{\ered}  \left(\CO^{(d-\ell)} \boxtimes M^{(\ell)} \right),
\end{equation}
where we have the maps
\begin{align*}
&p_+ : \quot_{\ell,d} \rightarrow \quot_d, & &(E^{(d)} \subset E^{(\ell)} \subset V) \mapsto (E^{(d)} \subset V) \\
&s_{(\ell)} : \quot_{\ell,d} \rightarrow  C^{(\ell)}, & &(E^{(d)} \subset E^{(\ell)} \subset V) \mapsto \emph{supp } V/E^{(\ell)}
\end{align*}
and the line bundle $\CL_2$ on $\quot_{\ell,d}$ parametrizes $\det \Gamma(C,V/E^{(\ell)})$, see \eqref{eqn:line bundles}.
The second equality in \eqref{eqn:wedge to e} simply reflects the definition of the functor $\be_{(d-\ell,\ell,0,\dots,0)}^{\ered}$, see \eqref{eqn:the functors}.

\end{proposition}

Since $\left ( \wedge^\ell M^{[d]} \right ) ^\vee = \wedge^{d-\ell} M^{[d]} \otimes (\det M^{[d]})^{-1}$, formulas \eqref{eqn:det} and \eqref{eqn:wedge to e} imply that
\begin{equation}
\label{eqn:wedge to e dual}
\left ( \wedge^\ell M^{[d]} \right ) ^\vee =  Rp'_{+*} \left({\CL}_1^{-1} \otimes {s'}_{(\ell)}^* (M^\vee)^{(\ell)}  \right),
\end{equation}
where $p_+' : \quot_{d-\ell,d} \rightarrow \quot_d$ and $s'_{(\ell)} : \quot_{d-\ell,d} \rightarrow C^{(\ell)}$ are the obvious maps, and the line bundle $\CL_1$ on $\quot_{d-\ell,d}$ parameterizes $\det \Gamma(C,E^{(d-\ell)}/E^{(d)})$, see \eqref{eqn:line bundles}. 

We also note that Proposition \ref{prop:tautological bundle} can be construed as the higher rank generalization of \cite[formula (16)]{krug1bis}.
 
\begin{proof} In the proof at hand, we let $\pi : \quot \times C^\ell \rightarrow \quot$ be the natural projection, irrespective of the indices of $\quot$. We make use of the nested Quot scheme $$\quot_{1,\dots, \ell, d} = \{ E^{(d)} \subset E^{(\ell)} \subset E^{(\ell-1)} \subset \dots \subset E^{(1)} \subset V\}$$ in which the inclusion $E^{(\ell)} \subset V$ is refined by a full intermediate flag. Consider the maps
$$
r : \quot_{1,\dots, \ell, d} \to \quot_{\ell,d} \quad \text{and} \quad s_\ell : \quot_{1,\dots,\ell,d} \rightarrow C^\ell,
$$
forgetting the intermediate full flag, respectively remembering the supports of successive quotients in it. 
We also set 
$$
t_+: \quot_{1,\dots,\ell,d} \xrightarrow{(p_+ \circ r) \times s_\ell} \quot_d \times C^\ell
$$
A straightforward generalization of Proposition \ref{prop:cartesian} gives 
$$
R(r \times s_\ell)_*(\mathcal{O}_{\quot_{1,\dots,\ell,d}}) = \CO_{\quot_{\ell,d} \times_{C^{(\ell)}} C^\ell}
$$
simply by replacing diagram \eqref{eqn:cartesian} by the following commutative diagram
\begin{equation}
\label{eqn:cartesian 2}
\xymatrix{& \quot_{1,\dots,\ell,d} \ar[ld]_{t_+} \ar[rd]^{r} \\ \quot_d \times C^{\ell} \ar[dr]_{\text{Id} \times \rho_{(\ell)}} & & \quot_{\ell,d} \ar[ld]^{p_+ \times s_{(\ell)}} \\ & \quot_d \times C^{(\ell)} &}
\end{equation}
The obvious analogues of the isomorphisms \eqref{eqn:equiv cartesian} and \eqref{eqn:equiv cartesian other way} hold, so in particular there exists an $S_{\ell}$ action on $Rr_*(\CO_{\quot_{1,\dots, \ell,d}})$ with invariant object $\CO_{\quot_{\ell,d}}$. Therefore
\begin{equation}
Rp_{+*} (\gamma ) = \left( Rp_{+*} \circ Rr_* \circ Lr^* (\gamma ) \right)^{S_{\ell}} = \left ( R\pi_* \circ Rt_{+*} \circ Lr^*  (\gamma) \right)^{S_\ell}
\end{equation}
for any $\gamma \in \text{Ob}(\DD_{\quot_{\ell,d}})$. If we take $\gamma = \CL_2 \otimes s_{(\ell)}^* ( M^{(\ell)})$ as in \eqref{eqn:wedge to e}, we have 
$$
r^*(\CL_2 \otimes s_{(\ell)}^*( M^{(\ell)})) = \CL_1' \otimes \dots \otimes \CL_{\ell}' \otimes s_\ell^*(M^\ell) 
$$
as an equality of line bundles on $\quot_{1,\dots,\ell,d}$, where $\CL_1',\dots,\CL_\ell'$ are the successive determinant line bundles $\CL_k' = \det R\pi_* (\CE^{(\ell-k)} -\CE^{(\ell-k+1)} )$ (these were called $\CL_2,\dots,\CL_{\ell+1}$ in Subsection \ref{sub:basic nested}). We make this relabeling because $\CL_k'$ will play the same role on $\quot_{1,\dots,\ell,d}$ as $\CL_k$ played on $\quot_{1,\dots,d}$ in the proof of Proposition \ref{prop:tautological}. Therefore, 
\begin{equation}
\label{eqn:rhs wedge}
\text{RHS of \eqref{eqn:wedge to e}} =R \pi_{*} \left( Rt_{+*} (\CL_1' \otimes \dots \otimes \CL_{\ell}' ) \otimes \rho^*(M^\ell) \right)^{S_\ell}
\end{equation}
where $t_+ : \quot_{1,\dots,\ell,d} \rightarrow \quot_d \times C^\ell$ and in the last equation, $\quot_d \times C^\ell \xrightarrow{\pi,\rho} \quot_d, C^\ell$ denote the standard projections. Let us now show that \eqref{eqn:wedge to e} follows from the $S_\ell$ equivariant equality
\begin{equation}
\label{eqn:the fact}
Rt_{+*} (\CL_1' \otimes \dots \otimes \CL_{\ell}' ) = \left( \CF^{(d)}_1 \otimes \dots \otimes \CF^{(d)}_\ell \mathsf {\, + \, Error} \right) \otimes \chi
\end{equation}
where $\CF^{(d)} = V/\CE^{(d)}$, $S_\ell$ permutes the factors of the tensor product on the right, $\chi$ is the alternating character and $\mathsf {\, + \, Error}$ refers to taking extensions by objects in the derived category of $\quot_d \times C^\ell$ which are preserved by some transposition in $S_\ell$ (compare with \eqref{eqn:big equation}). Indeed, applying the functor $R\pi_* ( - \otimes \rho^*(M^\ell))^{S_\ell}$ to both sides of the equation \eqref{eqn:the fact} implies
$$
\text{RHS of \eqref{eqn:wedge to e}} = (M^{[d]} \otimes \dots \otimes M^{[d]} \otimes \chi)^{S_\ell} = \wedge^\ell M^{[d]}
$$
because all the objects $\mathsf {Error} \otimes \chi$ in \eqref{eqn:the fact} go to 0 under the functor $R\pi_*(-)^{S_\ell}$, as they are negated by some transposition in $S_\ell$. 

\medskip

It remains to prove \eqref{eqn:the fact}. To do so, let us factor $t_+$ as
$$
\quot_{1,\dots,\ell,d} \xrightarrow{t_1} \quot_{1,\dots,\ell-1,d} \times C_1 \xrightarrow{t_2} \dots \xrightarrow{t_\ell} \quot_{d} \times C_1 \times \dots \times C_\ell
$$
where $t_k$ is the projectivization of the rank 0 sheaf $\CE^{(\ell-k)}_k/\CE^{(d)}_k$ (by Lemma \ref{lem:pi minus general}, hence the tautological line bundle is $\CL_{k}'$) and we will use $C_1,\dots,C_\ell$ to identify the $\ell$ factors of $C^\ell$. Thus, \eqref{eqn:the fact} is the $k=\ell$ version of the statement
\begin{equation}
\label{eqn:the step}
Rt_{k*} \circ \dots \circ Rt_{1*} \left(\CL_1' \otimes \dots \otimes \CL_{k}'\right) = (\CE^{(\ell-k)}_1/\CE^{(d)}_1) \otimes \dots \otimes (\CE^{(\ell-k)}_k/\CE^{(d)}_k) \mathsf {\, + \, Error}
\end{equation}
which we will prove by induction on $k$. To prove the induction step, we must apply $Rt_{k+1*}$ to the formula above. To this end, we will note the short exact sequence
$$
0 \rightarrow \CE^{(\ell-k)}_a \rightarrow \CE^{(\ell-k-1)}_a \rightarrow \CL_{k+1}' \otimes \CO_{\Delta_{a\sharp}} \rightarrow 0
$$
for all $a \in \{1,\dots,k\}$ on $\quot_{1,\dots,\ell-k,d} \times C_1 \times \dots \times C_k$, where $\Delta_{a\sharp}$ denotes the codimension 1 diagonal obtained from identifying the curve $C_a$ with the $\ell-k$-th support point of $\quot_{1,\dots,\ell-k,d}$. The short exact sequence above induces a short exact sequence 
$$
0 \rightarrow \CE^{(\ell-k)}_a/\CE^{(d)}_a \rightarrow \CE^{(\ell-k-1)}_a/\CE^{(d)}_a \rightarrow \CL_{k+1}' \otimes \CO_{\Delta_{a\sharp}} \rightarrow 0
$$
which can be used to write
$$
(\CE^{(\ell-k)}_1/\CE^{(d)}_1) \otimes \dots \otimes (\CE^{(\ell-k)}_k/\CE^{(d)}_k) \otimes \CL_{k+1}' = 
$$
$$
= \left(\CE^{(\ell-k-1)}_1/\CE^{(d)}_1 - \CL_{k+1}' \otimes \CO_{\Delta_{1\sharp}} \right) \otimes \dots \otimes \left(\CE^{(\ell-k-1)}_k/\CE^{(d)}_k - \CL_{k+1}' \otimes \CO_{\Delta_{k\sharp}} \right)  \otimes \CL_{k+1}'
$$
where $-$ denotes some extension, shifted by 1, in the derived category. Lemmas \ref{lem:push d virtual} and \ref{lem:pi minus general} imply that
$$
Rt_{k+1*} (\CL_{k+1}^{'n}) = S^n(\CE^{(\ell-k-1)}_{k+1}/\CE^{(d)}_{k+1})
$$
for all $n\geq 0$, so we conclude that 
$$
Rt_{k+1*} \left((\CE^{(\ell-k)}_1/\CE^{(d)}_1) \otimes \dots \otimes (\CE^{(\ell-k)}_k/\CE^{(d)}_k) \otimes \CL_{k+1}' \right)
$$
is equal to an extension of
\begin{equation}
\label{eqn:main object bis}
(\CE^{(\ell-k-1)}_1/\CE^{(d)}_1) \otimes \dots \otimes (\CE^{(\ell-k-1)}_k/\CE^{(d)}_k) \otimes (\CE^{(\ell-k-1)}_{k+1}/\CE^{(d)}_{k+1})
\end{equation}
and the following objects as $A = \{a_1<\dots < a_t\}$ runs over non-empty subsets of $\{1,\dots,k\}$
\begin{equation}
\label{eqn:error object bis}
\Delta_{a_1,\dots,a_t,k+1 *} \left(S^{t+1} (\CE^{(\ell-k-1)}/\CE^{(d)}) \right) \bigotimes_{b \in \{1,\dots,k\} \backslash A} (\CE^{(\ell-k-1)}_b/\CE^{(d)}_b).
\end{equation}
This concludes the induction step and thereby the proof of \eqref{eqn:the fact} as an equality of coherent sheaves. The $S_\ell$ action given by swapping factors of the tensor product and multiplying by the alternating character is argued exactly as in Proposition \ref{prop:tautological}. 

\end{proof}

\subsection{The main result}

We will now prove Theorem \ref{thm:tautological intro}. This is the main conjecture of \cite{krug1} which built upon Question 20 in the earlier paper \cite{os}. The statement is the following equality of graded vector spaces 
\begin{equation}
\label{eqn:iso tautological}
\text{Ext}^\bull_{\quot_d}\left(\wedge^{\ell_1} M_1^{[d]} \otimes \dots \otimes \wedge^{\ell_k} M_k^{[d]}, \wedge^{\ell} M^{[d]} \right) \cong 
\end{equation}
$$
 \cong \bigotimes_{i=1}^k S^{\ell_i} H^\bull(C,M \otimes M_i^{\vee})\bigotimes \wedge^{\ell-\ell_1 - \dots - \ell_k} H^\bull(C,V \otimes M) \bigotimes S^{d-\ell}H^\bull(C,\CO_C)
$$
for any line bundles $M_1,\dots,M_k,M$ on $C$ with $k < r$ and any $\ell_1,\dots,\ell_k,\ell \in \{0,\dots,d\}$. By taking suitable linear combinations of products of exterior powers, we could replace the first argument in the Ext group of \eqref{eqn:iso tautological} by any product of Schur functors
$$
S^{\lambda_1} M_1^{[d]} \otimes \dots \otimes S^{\lambda_k} M_k^{[d]}
$$
where the sum of the leading components of the partitions $\lambda_1,\dots,\lambda_k$ is less than $r$ (recall that $\wedge^\ell$ is the Schur functor corresponding to the partition $(1,\dots,1)$).

\begin{proof}\emph{of Theorem \ref{thm:tautological intro}}. Let us abbreviate 
\begin{equation}
\label{eqn:abbreviate}
\wedge^{\ell_1 \dots \ell_k} M_{1\dots k}^{[d]} = \wedge^{\ell_1} M_1^{[d]} \otimes \dots \otimes \wedge^{\ell_k} M_k^{[d]}
\end{equation}
The strategy of proof is to reduce \eqref{eqn:iso tautological} to $\ell = d$, and then to $\ell_1 = \dots  = \ell_k = 0$, in which case the required formula \eqref{eqn:iso tautological} is precisely \eqref{eqn:iso tautological lemma}. Let us first reduce to the case $\ell = d.$ To this end, we consider the nested Quot scheme 
$$
\quot_{\ell, d} = \{ E^{(d)} \subset E^{(\ell)} \subset V \},
$$ 
We have the associated morphism which forgets $E^{(\ell)}$,
$$
p_+: \quot_{\ell, d} \to \quot_d,
$$ 
and also the morphism which forgets $E^{(d)}$ but remembers the support of $E^{(\ell)}/E^{(d)}$
\begin{equation}
\label{eqn:p times s}
p_- \times s_{(d-l)}: \quot_{\ell, d} \to \quot_{\ell} \times C^{(d-\ell)}.
\end{equation}
Thanks to Proposition \ref {prop:tautological bundle}, we may write 
\begin{align*}
&{\text{Ext}}^\bullet_{\quot_d} \left( \wedge^{\ell_1 \dots \ell_k} M_{1\dots k}^{[d]}, \wedge^{\ell} M^{[d]} \right) = \\ 
&= {\text{Ext}}^\bullet_{\quot_{\ell, d}} \left(\wedge^{\ell_1 \dots \ell_k} M_{1\dots k}^{[d]}, \CL_2 \otimes s_{(\ell)}^*(M^{(\ell)}) \right) \\ 
&= H^\bullet \left(\quot_{\ell,d}, (\wedge^{\ell_1 \dots \ell_k} M_{1\dots k}^{[d]})^\vee \otimes \CL_2 \otimes s_{(\ell)}^*(M^{(\ell)}) \right) \\ 
&= H^\bullet\left(\quot_{\ell} \times C^{(d-l)}, R(p_- \times s_{(d-l)})_* \left((\wedge^{\ell_1 \dots \ell_k} M_{1\dots k}^{[d]})^\vee \otimes \CL_2 \otimes s_{(\ell)}^*(M^{(\ell)}) \right) \right) \\
&= H^\bullet\left(\quot_{\ell} \times C^{(d-l)}, R(p_- \times s_{(d-l)})_* \left((\wedge^{\ell_1 \dots \ell_k} M_{1\dots k}^{[d]})^\vee\right) \otimes \wedge^{\ell} M^{[\ell]} \right)
\end{align*}
where $\CL_2$ parameterizes $\det \Gamma(C,V/E^{(\ell)})$ and is thus pulled back from $\quot_{\ell}$. Let us compute the push-forward in the formula above. On $\quot_{\ell, d} \times C $ we have the maps
$$ 
V \to \CF^{(d)} \to \CF^{(\ell)}
$$
and let $\CF_{d-\ell} = \text{Ker }( \CF^{(d)} \to \CF^{(\ell)})$. If we set 
$$
M_{i,d-\ell} = R\pi_* (\mathcal F_{d-\ell} \otimes \rho^* (M_i)), \, \, M_i^{[\ell]} = R\pi_* (\mathcal F^{(\ell)} \otimes \rho^* (M_i)),
$$ 
with $\quot \times C \xrightarrow{\pi,\rho} \quot,C$ the usual projections, then we have a short exact sequence
$$
0 \to M_{i,d-\ell} \to M_i^{[d]} \to M_i^{[\ell]} \to 0
$$
on $\quot_{\ell,d}$. Clearly every exterior power  $\wedge^{\ell_i} M_i^{[d]}$ occurring in the product $\wedge^{\ell_1 \dots \ell_k} M_{1\dots k}^{[d]}$ can be resolved by vector bundles of the form $\wedge^{a_i} M_{i,d-\ell} \otimes \wedge^{b_i} M_i^{[\ell]}$ where $a_i + b_i = \ell_i$. With this in mind, we claim that the push-forward of the map \eqref{eqn:p times s} satisfies
\begin{equation}
\label{eqn: l to d}
R(p_- \times s_{(d-l)})_* \left [ \left(\wedge^{\ell_1 \dots \ell_k} M_{1\dots k}^{[d]}\right)^\vee \right ] = \left(\wedge^{\ell_1 \dots \ell_k} M_{1\dots k}^{[\ell]}\right)^\vee.
\end{equation}
To conclude \eqref{eqn: l to d}, we need to prove that
\begin{equation}
\label{eqn:aaa}
R(p_- \times s_{(d-\ell)})_* \Big(\wedge^{a_1} M_{1,d-\ell} \otimes \dots \otimes \wedge^{a_k} M_{k,d-\ell} \Big)^\vee = 0
\end{equation}
if $a_1+\dots+a_k > 0$, which we will do as follows. The diagram
$$
\xymatrixcolsep{4pc}\xymatrix{\quot_{\ell,\dots,d} \ar[d] \ar[r]^-{\widetilde{p}_- \times s_{d-l}} & \quot_{\ell} \times C^{d-\ell} \ar[d]^{\rho_{(d-\ell)}} \\
\quot_{\ell,d} \ar[r]^-{p_- \times s_{(d-l)}} & \quot_{\ell} \times C^{(d-\ell)}}
$$
satisfies the natural base change isomorphisms analogous to \eqref{eqn:equiv cartesian} and \eqref{eqn:equiv cartesian other way}. Therefore, equality \eqref{eqn:aaa} holds if and only if it holds after being pulled back along the flat morphism $\rho_{(d-\ell)}$, which is equivalent to the relation
\begin{equation}
\label{eqn:bbb}
R(\widetilde{p}_- \times s_{d-\ell})_* \Big(\wedge^{a_1} M_{1,d-\ell} \otimes \dots \otimes \wedge^{a_k} M_{k,d-\ell} \Big)^\vee = 0
\end{equation}
if $a_1+\dots+a_k > 0$. On $\quot_{\ell,\dots,d}$, the rank $d-\ell$ vector bundle $M_{i,d-\ell}$ can be expressed as an extension of the line bundles $\CL_1,\dots,\CL_{d-\ell}$ tensored with $s_{d-\ell}^*(\text{various line bundles})$. Therefore, the left-hand side of \eqref{eqn:bbb} is an extension of objects of the form
\begin{equation}
\label{eqn:ccc}
R(\widetilde{p}_- \times s_{d-\ell})_* \Big(\CL_1^{-n_1} \otimes \dots \otimes \CL_{d-\ell}^{-n_{d-\ell}} \otimes s_{d-\ell}^*(\text{line bundle}) \Big)
\end{equation}
for various $n_1,\dots,n_{d-\ell} \in \{0,\dots,k\}$ satisfying $n_1+\dots+n_{d-\ell} = a_1+\dots+a_k > 0$. By Lemma \ref{lem:pi minus}, the map $\widetilde{p}_- \times s_{d-\ell}$ is a tower of $d-\ell$ projectivizations of rank $r$ vector bundles. Since $k < r$, formula \eqref{eqn:push d} guarantees that all of the push-forwards \eqref{eqn:ccc} are zero, thus proving \eqref{eqn:bbb} $\Rightarrow$ \eqref{eqn:aaa} $\Rightarrow$ \eqref{eqn: l to d}. The latter formula implies that
\begin{align*}
\text{LHS of \eqref{eqn:iso tautological}} &= 
H^\bullet \left(\quot_{\ell} \times C^{(d-l)}, \left(\wedge^{\ell_1 \dots \ell_k} M_{1\dots k}^{[\ell]}\right)^\vee \otimes \wedge^{\ell} M^{[\ell]} \right) \\
&= H^\bullet\left(\quot_{\ell}, \left(\wedge^{\ell_1 \dots \ell_k} M_{1\dots k}^{[\ell]}\right)^\vee \otimes \wedge^{\ell} M^{[\ell]} \right) \otimes H^\bullet (C^{(d-\ell)}, \mathcal O_{C^{(d-\ell)}})  \\
&= \text{Ext}^\bullet_{\quot_{\ell}}\left(\wedge^{\ell_1 \dots \ell_k} M_{1\dots k}^{[\ell]}, \wedge^{\ell} M^{[\ell]} \right)  \otimes S^{d-\ell} H^\bullet (C, \mathcal O_C).
\end{align*}
The formulas above reduce \eqref{eqn:iso tautological} to the case $\ell = d$, which we henceforth assume. Let us now reduce it further to the case $\ell_1 = \dots = \ell_k = 0$. It suffices to prove that
\begin{equation}
\label{eqn: wedgereduction}    
{\text{Ext}}^\bullet_{\quot_d} \left( \wedge^{\ell_1 \dots \ell_k} M_{1\dots k}^{[d]},  \wedge^{d} M^{[d]} \right) 
\end{equation}
$$
={\text{Ext}}^\bullet_{\quot_{d -\ell_1}} \left( \wedge^{\ell_2 \dots \ell_k} M_{2\dots k}^{[d-\ell_1]}, \wedge^{d -\ell_1} M^{[d -\ell_1]} \right) \otimes H^\bullet (C^{\,(\ell_1)}, (M \otimes M_1^{\vee})^{(\, \ell_1) }).
$$ 
Indeed, iterating formula \eqref{eqn: wedgereduction} $k$ times results in 
$$
{\text{Ext}}^\bullet_{\quot_d} \left( \wedge^{\ell_1 \dots \ell_k} M_{1\dots k}^{[d]}, \wedge^{d} M^{[d]} \right) 
$$
$$
=H^\bullet \left(\quot_{d -\ell_1 - \cdots - \ell_k}, \wedge^{d -\ell_1 -\cdots - \ell_k} M^{[d -\ell_1 - \cdots - \ell_k]} \right) \bigotimes_{i=1}^{k}\, S^{\ell_i}\,  H^\bullet (C, M \otimes M_i^{\vee}).
$$
which reduces \eqref{eqn:iso tautological} to the case $\ell = d$ and $\ell_1 = \cdots = \ell_k = 0$. This case is the already proven formula \eqref{eqn:iso tautological lemma} of Proposition \ref{prop:tautological}. The proof of Theorem \ref{thm:tautological intro} is therefore complete once we establish equation \eqref{eqn: wedgereduction}. To do so, let us consider the Quot scheme
$$
\quot_{d-\ell_1,\,  d} = \{ E^{(d)} \subset E^{(d- \ell_1)} \subset V \}
$$ 
and the natural morphisms
\begin{equation}
\label{eqn:twomaps}
\quot_{d-\ell_1} \times C^{(\ell_1)} \xleftarrow{p_-' \times s_{(\ell_1)}}  \quot_{d-\ell_1,\,  d}  \stackrel{p_+'}{\longrightarrow}  \quot_d.
\end{equation}
We may use formula \eqref{eqn:wedge to e dual} in order to write on $\quot_d$
\begin{equation}
\label{eqn:pause}
\text{LHS of \eqref{eqn: wedgereduction}} = H^\bullet \left(\quot_d, \left(\wedge^{\ell_1 \dots \ell_k} M_{1\dots k}^{[d]}\right)^\vee \otimes \wedge^d M^{[d]} \right) 
\end{equation}
\begin{align*}
&= H^\bullet \left(\quot_{d},  \left(\wedge^{\ell_2 \dots \ell_k} M_{2\dots k}^{[d]}\right)^\vee \otimes Rp_{+*}' \left (\CL_1^{-1} \otimes s_{(\ell_1)}^* (M_1^\vee)^{(\ell_1)} \right)  \otimes \wedge^d M^{[d]}\right) \\
&= H^\bullet \left(\quot_{d-\ell_1,d},   \left(\wedge^{\ell_2 \dots \ell_k} M_{2\dots k}^{[d]}\right)^\vee \otimes \CL_1^{-1} \otimes s_{(\ell_1)}^* (M_1^\vee)^{(\ell_1)}  \otimes \wedge^d M^{[d]}  \right)
\end{align*}
where $\CL_1$ parameterizes $\det \Gamma(C,E^{(d-\ell_1)}/E^{(d)})$ on $\quot_{d-\ell_1,d}$. However, since we have the following short exact sequence
$$
0 \to M_{\ell_1} \to M^{[d]} \to M^{[d-\ell_1]} \to 0
$$
of rank $\ell_1,d,d-\ell_1$ vector bundles on $\quot_{d-\ell_1,d},$ with $$M_{\ell_1} = R\pi_*(\text{Ker }(\CF^{(d)} \twoheadrightarrow \CF^{(d-\ell_1)}) \otimes \rho^*(M)),$$ we obtain the corresponding equality of determinants
$$
\wedge^d M^{[d]} = \wedge^{d-\ell_1} M^{[d-\ell_1]} \otimes \wedge^{\ell_1} M_{\ell_1} = \wedge^{d-\ell_1} M^{[d-\ell_1]} \otimes \CL_1 \otimes s_{(\ell_1)}^*(M^{(\ell_1)}).
$$
Plugging this into \eqref{eqn:pause} gives us
\begin{align*}
&\text{LHS of \eqref{eqn: wedgereduction}} = H^\bullet \left(\quot_{d-\ell_1,d} , \left(\wedge^{\ell_2 \dots \ell_k} M_{2\dots k}^{[d]}\right)^\vee \otimes \wedge^{d-\ell_1} M^{[d-\ell_1]} \otimes s_{(\ell_1)}^*((M \otimes M_1^\vee)^{(\ell_1)})  \right) \\
&= H^\bullet \left(\quot_{d-\ell_1} \times C^{(\ell_1)}, R(p'_- \times s_{(\ell_1)})_*  \left( \left(\wedge^{\ell_2 \dots \ell_k} M_{2\dots k}^{[d]}\right)^\vee \right) \otimes \wedge^{d-\ell_1} M^{[d-\ell_1]} \otimes \rho^*((M \otimes M_1^\vee)^{(\ell_1)}) \right)
\end{align*}
where $\rho : \quot_{d-\ell_1} \times C^{(\ell_1)} \to C^{(\ell_1)}$ denotes the projection onto the second factor. For the same reason that formula \eqref{eqn: l to d} held, the only tensor product of wedges which survives the push-forward under $R(p'_- \times s_{(\ell_1)})_*$ is
$$
\left(\wedge^{\ell_2 \dots \ell_k} M_{2\dots k}^{[d-\ell_1]}\right)^\vee
$$
and we thus conclude that the left-hand side of \eqref{eqn: wedgereduction} is
\begin{align*}
&H^\bullet \left(\quot_{d-\ell_1} \times C^{(\ell_1)}, \left(\wedge^{\ell_2 \dots \ell_k} M_{2\dots k}^{[d-\ell_1]}\right)^\vee \otimes \wedge^{d-\ell_1} M^{[d-\ell_1]} \otimes \rho^*( (M \otimes M_1^\vee)^{(\ell_1)}) \right) \\ 
&= H^\bullet \left(\quot_{d-\ell_1}, \left(\wedge^{\ell_2 \dots \ell_k} M_{2\dots k}^{[d-\ell_1]}\right)^\vee \otimes \wedge^{d-\ell_1} M^{[d-\ell_1]} \right) \otimes H^\bullet(C^{(\ell_1)}, (M \otimes M_1^\vee)^{(\ell_1)}) \\
&= H^\bullet \left(\quot_{d-\ell_1}, \left(\wedge^{\ell_2 \dots \ell_k} M_{2\dots k}^{[d-\ell_1]}\right)^\vee \otimes \wedge^{d-\ell_1} M^{[d-\ell_1]} \right) \otimes S^{\ell_1}H^\bullet(C, M \otimes M_1^\vee) 
\end{align*}
which is precisely the right-hand side of \eqref{eqn: wedgereduction}. \end{proof}

\section{Appendix}

\vskip.2in

\subsection{The local picture} In the present Section, we will prove all ``local" statements that were used in the present paper, i.e. the ones for which one could assume that $C = \BA^1$. To this end, let us remark that while the scheme $\quot_d$ is not defined for the curve $C = \BA^1$, an \'etale local model for $\quot_d$ is given by
\begin{equation}
\label{eqn:local quot}
\quot_d^{\BA^1} = \Big\{ A : \BC^d \rightarrow \BC^d, v_1,\dots,v_r \in \BC^d \text{ which are cyclic for }A \Big\} \Big / GL_d(\BC)
\end{equation}
where the action of $GL_d$ conjugates $A$ and multiplies $v_1,\dots,v_r$. Above, the $r$-tuple of vectors $v_1,\dots,v_r$ is called ``cyclic for $A$" if the entire $\BC^d$ is linearly spanned by $\{A^k v_i\}_{k \geq 0, i \in \{1,\dots,r\}}$. The meaning of the local picture above is that the data \eqref{eqn:local quot} precisely determines a map of $\BC[x] = \CO_{\BA^1}$-modules
$$
\BC[x]^{\oplus r} \twoheadrightarrow \BC^d, \qquad (\underbrace{\dots,0,x^k,0,\dots}_{x^k \text{ on }i\text{-th position}}) \mapsto A^kv_i
$$
thus making $F = \BC^d$ into a coherent sheaf on $\BA^1$. The nested Quot schemes $\quot_{d_1,\dots,d_n}^{\BA^1}$ are also defined as above, but requiring $A$ to be appropriately block triangular.

\begin{remark}

We expect that Theorem \ref{thm:intro main} also holds for $\quot_d^{\BA^1}$, and that our proof will run through with the obvious modifications.
    
\end{remark}

\subsection{The Propositions} We will now use the local picture above to provide proofs of three key technical results. The first of these is Proposition \ref{prop:z plus}, which states that the variety $\fZ_+'$ of Subsection \ref{sub:two components} has rational singularities, i.e.
$$
R\pi'_{+*}(\CO_{\fY}) = \CO_{\fZ_+'} 
$$
Because the study of the singularities of $\fZ_+'$ is local, we may assume $C = \BA^1$.

\begin{proof} \emph{of Proposition \ref{prop:z plus}}: We recall that $\quot_{d+1}^{\BA^1}$ parameterizes $A : \BC^{d+1} \rightarrow \BC^{d+1}$ together with cyclic vectors $v_1,\dots,v_r \in \BC^{d+1}$, modulo $G = GL_{d+1}(\BC)$. Then the $C = \BA^1$ versions of the schemes $\fZ_+, \fZ_+', \fY$ admit the following presentations
\begin{equation}
\label{eqn:z plus affine}
\fZ_+^{\BA^1} = \Big\{A,v_1,\dots,v_r \text{ as above}, \lambda, \mu \in \BC, \text{lines } \ell, m \subset \BC^{d+1} \Big|A\ell = \lambda \ell, A m = \mu m\Big\} \Big/ G
\end{equation}
\begin{multline}
\label{eqn:y affine}
\fY^{\BA^1} = \Big\{A,v_1,\dots,v_r,\lambda, \mu, \ell, m \text{ as in \eqref{eqn:z plus affine}, together with a plane } S \subseteq \BC^{d+1} \Big| \\ S \supset \ell, m \text{ and } A S \subseteq S, A(S/\ell) = \mu (S/\ell), A(S/m) = \lambda(S/m) \Big\} \Big/G
\end{multline}
In the formulas above, expressions such as $A\ell = \lambda \ell$ and $A(S/m) = \lambda(S/m)$ mean that $A$ acts as multiplication by $\lambda$ on the one-dimensional vector spaces $\ell$ and $S/m$. The map $\pi_+ : \fY \rightarrow \fZ_+$ forgets $S$. We claim that the image $\fZ_+'$ of $\pi_+$ can described as
\begin{equation}
\label{eqn:z plus prime affine}
{\fZ_+^{'\BA^1}} = \Big\{A,v_1,\dots,v_r,\lambda, \mu, \ell, m \text{ as in \eqref{eqn:z plus affine}} \Big| (z-\lambda)(z-\mu) \text{ divides } \det (z I_{d+1}-A)\Big\} \Big/G
\end{equation}
Indeed, for any $A,v_1,\dots,v_r,\lambda,\mu,\ell,m$ as in \eqref{eqn:z plus affine}, the existence of a 2-dimensional subspace $S$ as in \eqref{eqn:y affine} is equivalent to either $\lambda \neq \mu$ (in which case $S$ is forced to be $\ell \oplus m$) or $\lambda = \mu$ but with algebraic multiplicity at least 2. We will perform the following simplifications, which do not affect the singularities of the morphism
$$
\pi'_+ : \fY^{\BA^1} \rightarrow \fZ_+^{'\BA^1}
$$

\begin{itemize}[leftmargin=*]

\item ignore the free $G$ action

\item ignore the vectors $v_1,\dots,v_r$, which constitute an open subset in an affine bundle

\item perform a translation and change of basis so that
\begin{equation}
\label{eqn:assumptions}
\ell = \text{span} \begin{pmatrix} 1 \\ 0 \\ \vdots \\ 0 \end{pmatrix}, \quad \lambda = 0
\end{equation}

\end{itemize}

In the aforementioned basis, we must have
$$
A = \begin{pmatrix} 0 & f_1 & \dots & f_d \\ 0 & b_{11} & \dots & b_{1d} \\ \vdots & \vdots & \ddots & \vdots \\ 0 & b_{d1} & \dots & b_{dd} \end{pmatrix}
$$
We may assume that
$$
m = \text{span} \begin{pmatrix} 1 \\ w_1 \\ \vdots \\ w_d \end{pmatrix}
$$
because if the first entry of $m$ were 0, then the fact that $(w_1,\dots,w_d) \neq (0,\dots,0)$ would imply that we would actually be on the smooth locus $\{\ell \neq m\}$ of the variety \eqref{eqn:z plus prime affine}. Therefore, the three bullets above allow us to replace ${\fZ_+^{'\BA^1}}$ by the variety 
\begin{equation}
\label{eqn:the coordinates}
Z'_+ = \Big\{ (B,w,\mu,f) \Big| B w = \mu w, \mu = f \cdot w, \det (\mu  I_d - B) = 0 \Big\}
\end{equation}
where $B$ is the matrix with coefficients $\{b_{ij}\}_{1\leq i,j \leq d}$, $f$ is the row vector with entries $f_i$ and $w$ is the vector with entries $w_i$. Similarly, we may replace $\fY^{\BA^1}$ by the variety
\begin{equation}
\label{eqn:the coordinates resolution}
Y = \Big\{ (B,w,\mu,f,s) \text{ as above} \Big| B s = \mu s, w \in s  \Big\}
\end{equation}
where $s$ is a line in $\BC^d$. Our problem has been reduced to showing that 
\begin{equation}
\label{eqn:p rational}
Rp_*(\CO_Y) = \CO_{Z_+'}
\end{equation}
where $p : Y \rightarrow Z_+'$ is the map that forgets $s$.

\begin{lemma}
\label{lem:cohen-macaulay}

For any $d \geq 0$, consider the varieties
\begin{equation}
\label{eqn:the scheme}
\bar{Z}'_+ = \Big\{B:\BC^d \rightarrow \BC^d, \mu \in \BC, w \in \BC^d \Big| B w = \mu w,  \det (\mu  I_d - B) = 0 \Big\}
\end{equation}
\begin{equation}
\label{eqn:the scheme resolution}
\bar{Y} = \Big\{B,\mu,w \text{ as above and a line }s \subset \BC^d \Big| (B- \mu I_d) s = 0, w \in s \Big\}
\end{equation}
The map $\bar{p} : \bar{Y} \rightarrow \bar{Z}_+'$ that forgets $s$ has the property that $R\bar{p}_*(\CO_{\bar{Y}}) = \CO_{\bar{Z}_+'}$.

\end{lemma}

Because $p : Y \rightarrow Z_+'$ is obtained from $\bar{p} : \bar{Y} \rightarrow \bar{Z}_+'$ via base change along the local complete intersection morphism which adds the coordinates $f_1,\dots,f_d$ and imposes the equation $\mu = f \cdot w$ (this uses the straightforward fact that $\dim Y - \dim \bar{Y} = \dim Z'_+ - \dim \bar{Z}'_+ = d-1$), then the fact that $\bar{p}$ is a resolution of rational singularities implies the sought-for fact that $p$ is a resolution of rational singularities, i.e. \eqref{eqn:p rational}.

\begin{proof} We thank Claudiu Raicu for showing us the following argument, following \cite[Theorem 5.1.3]{weyman}. By a translation, we may assume that $\mu = 0$. Let $\BP^{d-1}$ be the projective space of lines $s \subset \BC^d$, on which we have the usual short exact sequence
$$
0 \rightarrow \CO_{\BP^{d-1}}(-1) \rightarrow \CO_{\BP^{d-1}}^{\oplus d} \rightarrow \CQ \rightarrow 0
$$
By definition, we have
$$
\bar{Y} = \text{Spec}_{\BP^{d-1}} \Big( S^\bullet \left( \text{Hom}(\CO_{\BP^{d-1}}^{\oplus d},\CQ) \oplus \CO_{\BP^{d-1}}(1) \right) \Big)
$$
(the dual of the Hom space above parameterizes the matrix $B$ and the dual of $\CO_{\BP^{d-1}}(1)$ parameterizes the vector $w \in s$). Since the symmetric algebra above can be resolved by the symmetric powers of the vector bundles $\CQ$ and $\CO_{\BP^{d-1}}(1)$, it has no higher cohomology. Similarly, any global function on $\bar{Y}$ can be expressed in terms of the symmetric powers of $\CO_{\BP^{d-1}}(1)$ (which gives a polynomial function in the entries of the vector $w$) and in terms of the symmetric powers of $\text{Hom}(\CO_{\BP^{d-1}}^{\oplus d},\CQ)$ (which gives a polynomial function in the entries of the matrix $B$ satisfying $Bw = 0$), which is precisely the coordinate ring of the affine variety $\bar{Z}_+'$. 

\end{proof} \end{proof}

The next result we will prove is Proposition \ref{prop:cartesian}, which states that the map $\zeta$ of \eqref{eqn:zeta def} is a resolution of rational singularities, i.e.
$$
R\zeta_*(\CO_{\quot_{1,\dots,d}}) = \CO_{C^d \times_{C^{(\bd)}} \quot_{\bd}} 
$$

\begin{proof} \emph{of Proposition \ref{prop:cartesian}:} The statement \eqref{eqn:zeta push} is \'etale local in the base $C^{(\bd)}$, so we may assume that $C = \BA^1$ and $(\bd) = (d,0,\dots,0)$. In this setting, we recall that
$$
\quot_d^{\BA^1} \xrightarrow{j} (\fg \times \BC^{rd}) \Big/ G
$$
where $G = GL_d(\BC)$, $\fg = \text{End}(\BC^d)$ and $\BC^{rd}$ parameterizes the vectors $v_1,\dots,v_r$ of \eqref{eqn:local quot}. The map $j$ above is the open embedding encoding the cyclicity condition. Similarly, we have 
$$
\quot_{1,\dots,d}^{\BA^1} \xrightarrow{j'} (\tfg \times \BC^{rd}) \Big/ G
$$
where $\tfg$ is the Grothendieck-Springer resolution which parameterizes pairs of $A \in \fg$ together with a full flag $\Phi$ of $\BC^d$ reserved by $A$. With this in mind, the commutative diagram \eqref{eqn:cartesian} can be refined to a commutative diagram
$$
\xymatrix{ \quot_{1,\dots,d}^{\BA^1} \ar[d]_{p_{(d)}} \ar[r]^-{j'} & (\tfg \times \BC^{rd}) /G\ar[d]^{\mu} \ar[r]^-{\lambda'} & (\BA^1)^{d} \ar[d]^{\rho_{(d)}} \\ \quot_d^{\BA^1} \ar[r]^-j & (\fg \times \BC^{rd}) /G \ar[r]^-{\lambda} & (\BA^1)^{(d)}}
$$
where the maps $\lambda$ and $\lambda'$ keep track of the eigenvalues of the matrix $A$, and the map $\mu$ forgets the full flag $\Phi$. However, we have the following isomorphism of $G$-equivariant coherent sheaves (see \cite[Theorem 7.13]{gaitsgory})
\begin{equation}
\label{eqn:eta}
R\eta_*(\CO_{\tfg}) =  \CO_{\CY}
\end{equation}
where $\eta : \tfg \rightarrow \CY =  \fg \times_{(\BA^1)^{(d)}} (\BA^1)^d$ is the natural map induced by the commutative diagram
$$
\xymatrix{ \tfg \ar[d]_-\mu \ar[r]^-{\lambda'} & (\BA^1)^{d} \ar[d]^{\rho_{(d)}} \\ \fg \ar[r]^-{\lambda} & (\BA^1)^{(d)}}
$$
Therefore, the statement analogous to \eqref{eqn:eta} holds when

\begin{itemize}

\item replacing $\fg$ and $\tfg$ by $\fg \times \BC^{rd}$ and $\tfg \times \BC^{rd}$, respectively

\item taking the quotient stacks of $\fg \times \BC^{rd}$ and $\tfg \times \BC^{rd}$ by $G$

\item considering the open subsets determined by the cyclicity condition

\end{itemize}

\noindent After making the three modifications above, \eqref{eqn:eta} yields precisely \eqref{eqn:zeta push}. \end{proof}

Finally, we must show that the moduli space $\fY_i$ of sheaves \eqref{eqn:yi} satisfies
\begin{equation}
\label{eqn:blow up i app}
R\pi^\uparrow_{i*}(\CO_{\fY_i}) = \CO_{\quot_{1,\dots,d}} 
\end{equation}
where $\pi^\uparrow_i : \fY_i \rightarrow \quot_{1,\dots,d}$ forgets the sheaf denoted by $\tE^{(d-i)}$. As a consequence of \eqref{eqn:blow up i app}, we will deduce that the two $S_2$ equivariant structures on the object
$$
\Gamma = Rs_{d*} \circ Lp_{(d)}^*(\gamma) \in \text{Ob} \left(\DD_{C^d}\right)
$$
(for any $\gamma \in \text{Ob}(\DD_{\quot_{d}})$) described in the statement of Lemma \ref{lem:big} coincide.

\begin{proof} \emph{of Lemma \ref{lem:big}:} It is easy to see that $\fY_i$ can be defined as a relative projectivization over the moduli space that parameterizes the central square of \eqref{eqn:yi}, just like $\quot_{1,\dots,d}$ can be defined as a relative projectivization over the moduli space $\quot_{i-1,i,i+1}$. If one can show that $\fY_i$ has dimension equal to that of $\quot_{1,\dots,d}$, then the base change isomorphism would allow us to deduce \eqref{eqn:blow up i app} from \eqref{eqn:push o up and down}. To perform this dimension count, let us work \'etale locally (i.e. assume $C = \BA^1$), in which case
$$
\quot_{1,\dots,d}^{\BA^1} = \Big\{A \text{ upper triangular } d \times d \text{ matrix and cyclic vectors }v_1,\dots,v_r \Big\} \Big/B
$$
where $B$ denotes the Borel subgroup of upper triangular $d \times d$ matrices. Similarly
$$
\fY_i^{\BA^1} = \left\{A,v_1,\dots,v_r \text{ as above, } [x : y] \in \BP^1 \text{ s.t. } \begin{pmatrix} a_{i,i} & a_{i,i+1} \\ 0 & a_{i+1,i+1} 
\end{pmatrix} \begin{pmatrix} x \\ y \end{pmatrix} = a_{i+1,i+1} \begin{pmatrix} x \\y \end{pmatrix}  \right\} \Big/B
$$
The map $\pi_i^\uparrow : \fY_i^{\BA^1} \rightarrow \quot_{1,\dots,d}^{\BA^1}$ forgets $[x:y]$ and has fibers equal to $\BP^1$ if $a_{i,i} = a_{i+1,i+1}$ and $a_{i,i+1} = 0$, and equal to a point otherwise. This readily implies the fact that $\fY_i^{\BA^1}$ has expected dimension $rd$, since both $\quot_{1,\dots,d}^{\BA^1}$ and its subvariety $\{a_{i,i} = a_{i+1,i+1}\}$ are smooth and connected of dimensions $rd$ and $rd-1$, respectively (on account of them both being towers of projective bundles, see \eqref{eqn:punctual} and the subsequent paragraph) and so imposing the additional equation $a_{i,i} = a_{i+1,i+1}$ will cut out a subvariety of dimension $rd-2$. 

\medskip

To calculate the $S_2$ equivariant structure of $\Gamma$, let us rewrite
\begin{multline*}
\fY_i^{\BA^1} = \Big\{A \in \fp_i,v_1,\dots,v_r \text{ cyclic, lines } \ell, m \subset S \Big| \\  A\ell = \lambda \ell, Am = \mu m, A(S/\ell) = \mu(S/\ell), A(S/m) = \lambda(S/m) \Big\} \Big/ P_i
\end{multline*}
where $P_i$ is the parabolic subgroup that preserves the flag $\Phi_i = (\BC^1 \hookrightarrow \dots \hookrightarrow \BC^{i-1} \hookrightarrow \BC^{i+1} \hookrightarrow \dots \hookrightarrow \BC^d)$, $\fp_i$ is its Lie algebra, $S$ denotes $\BC^{i+1}/\BC^{i-1}$ and $\lambda,\mu$ denote the eigenvalues of $A$ on the 2-dimensional subspace $S$. The action of $S_2$ on $\fY_i^{\BA^1}$ permutes $(\ell,\lambda)$ and $(m,\mu)$ and it is clear that the map $r_i : \fY_i^{\BA^1} \rightarrow (\BA^1)^d$ which remembers the eigenvalues $(\dots,\lambda,\mu,\dots)$ is $S_2$ equivariant. This map fits in the following diagram 

$$
\xymatrix{\fY_i^{\BA^1} \ar@/_2pc/[rrd]_{p_{(d)} \circ \pi_i^\uparrow} \ar@/^3pc/[rrr]^{r_i} \ar@/^1pc/[rr]^{\nu_i} \ar[r]_-{\pi^\uparrow_i} & \quot_{1,\dots,d}^{\BA^1} \ar[rd]_{p_{(d)}} \ar[r]_-{\zeta} & \quot_{d}^{\BA^1} \times_{(\BA^1)^{(d)}} (\BA^1)^d \ar[r]^-{\text{proj}_2} \ar[d]^{\text{proj}_1}& (\BA^1)^d \\ & & \quot_{d}^{\BA^1}}
$$
where $\pi_i^\uparrow$ forgets the line $m$ and $\zeta$ forgets that $A$ is upper triangular, but remembers its sequence of diagonal entries. The first two horizontal maps in the sequence above have the property that derived direct image takes the structure sheaf to the structure sheaf (see \eqref{eqn:zeta push} and \eqref{eqn:blow up i app}), so we conclude that the same is true for the map $\nu_i$. Since
$$
\Gamma =  R\text{proj}_{2*} \circ R\zeta_* \circ L \zeta^* \circ L \text{proj}_1^*(\gamma)
$$
then the discussion above implies both the formulas
$$
\Gamma =  R\text{proj}_{2*} \circ L \text{proj}_1^*(\gamma) 
$$
and
$$
\Gamma = Rr_{i*} \circ L(p_{(d)} \circ \pi_i^\uparrow)^*(\gamma) = R\text{proj}_{2*} \circ R\nu_{i*} \circ L \nu_i^* \circ L \text{proj}_1^*(\gamma)
$$
By definition, the former of these formulas defines the $S_2$ equivariance in the first bullet (in the statement) of Lemma \ref{lem:big}, which the latter formula defines the $S_2$ equivariance in the second bullet therein. The fact that they coincide is a consequence of the fact that $\nu_i$ is $S_2$-equivariant and has the property that $R\nu_{i*} \circ L\nu_i^* = \text{Id}$. \end{proof}

\end{document}